\documentclass[oneside]{article}
\usepackage[utf8]{inputenc}
\usepackage[T1]{fontenc}
\usepackage[a4paper, left=2cm, right=2cm, top=2.5cm, bottom=2.5cm]{geometry} 
\usepackage{amsfonts}
\usepackage{hyperref}
\usepackage{cite}
\hypersetup{
pdftitle={On the convex hull of two planar random walks},
pdfsubject={Mathematics, Probability},
pdfauthor={Ivanković Daniela, Kralj Tomislav, Sandrić Nikola, Šebek Stjepan},
pdfkeywords={}
}
\usepackage{amsmath}
\usepackage{xcolor}
\usepackage{amsthm}
\usepackage{pdflscape}
\usepackage{pgfplots}
\usepackage{mathrsfs}
\usepackage{verbatim}
\usepackage{fancyhdr}
\usepackage{lipsum}
\usepackage{titlesec}
\usepackage{tikz}
\usetikzlibrary{angles, quotes}
\usepackage{setspace}
\usepackage[labelformat=simple]{subcaption}

\titleformat{\section}{\large\scshape\centering}{\thesection}{1em}{}
\titleformat{\subsection}{\normalsize\scshape}{\thesubsection}{1em}{}

\titleformat{\title}{\normalsize\scshape}{\thetitle}{1em}{}
\titleformat{\author}{\normalsize\scshape}{\theauthor}{1em}{}

\newtheorem{theorem}{Theorem}[section]
\newtheorem{thm}{Theorem}[section]
\newtheorem{corollary}[thm]{Corollary}
\newtheorem{proposition}[thm]{Proposition}
\newtheorem{lemma}[thm]{Lemma}

\theoremstyle{definition}

\theoremstyle{remark}
\newtheorem{remark}[thm]{Remark}

\newcounter{myequations}

\newcommand{\e}{\mathbf{e}}
\newcommand{\E}{\mathbb{E}}
\newcommand{\I}{\mathcal{I}}
\newcommand{\V}{\operatorname{\mathbb{V}ar}}

\newcommand{\chull}{\operatorname{chull}}
\newcommand{\Per}{\operatorname{Per}}
\newcommand{\diam}{\operatorname{diam}}

\numberwithin{equation}{section}

\title{On the convex hull of two planar random walks}

\author{Daniela Ivankovi\'c, Tomislav Kralj, Nikola Sandri\'c, Stjepan \v{S}ebek}

\begin{document}

\title{\scshape{\Large On the convex hull of two planar random walks}}
\author{\textsc{\large Daniela Ivankovi\'c\footnote{University of Zagreb, Faculty of Science, e-mail: \href{mailto:daniela.ivankovic@math.hr}{daniela.ivankovic@math.hr}}\ , Tomislav Kralj\footnote{University of Zagreb, Faculty of Science, e-mail: \href{mailto:tomislav.kralj@math.hr}{tomislav.kralj@math.hr}}\  , Nikola Sandri\'c\footnote{University of Zagreb, Faculty of Science, e-mail: \href{mailto:nikola.sandric@math.hr}{nikola.sandric@math.hr}}\ , Stjepan \v{S}ebek\footnote{University of Zagreb, Faculty of Electrical Engineering and Computing, e-mail: \href{mailto:stjepan.sebek@fer.unizg.hr}{stjepan.sebek@fer.unizg.hr}}}}
\date{}

\maketitle

\begin{abstract}
In this paper, we study the limiting behavior of the perimeter and diameter functionals of the convex hull spanned by the first $n$ steps of two planar random walks. As the main results, we obtain the strong law of large numbers and the central limit theorem for the perimeter and diameter of these random sets.

\end{abstract}

\bigskip


\tableofcontents

\setlength{\parindent}{0pt} 
\setlength{\parskip}{0.5em} 

\section{Introduction}

Random walks are one of the most important classes of stochastic processes. They are used as a mathematical model of many phenomena arising in finance, physics, computer science, and biology, among other fields. In this paper, we focus on the geometric properties of sets of points generated by random walks. Understanding these properties can provide valuable insights into the behavior of various systems modeled by random walks. Investigating geometric properties of random walks in Euclidean space is a topic of persistent interest (see e.g.\ \cite{rudnick1987shapes}). The interest in the convex hull of a random walk, which is a classical geometrical characteristic of the walk \cite{Snyder-Steele, Spitzer-Widom}, has recently increased significantly on several fronts; among many papers, we mention \cite{Eldan1, Eldan2, Tikhomirov, VZ, wade2015convex, wade2015convex2, mcredmond2018convex, mcredmond2017expected, McRedmond_phd, Xu_phd, CSS, CSSW}. We also refer the reader to \cite{Majumdar} for a survey of the state of the field around 2010, that includes motivation in terms of modeling the home range of roaming animals.

In many real-world scenarios, random walks are subject to drifts, which are tendencies of the walk to move in a particular direction. In this paper, we consider two independent planar random walks with drifts. We are particularly interested in the diameter and the perimeter of the convex hull spanned by the first $n$ steps of the walks. The convex hull of a set $A$ is the smallest convex set containing $A$, and its geometric properties like perimeter and diameter can provide useful information about the dispersion of the points of $A$.

\begin{figure}[h]
    \centering
    \begin{tikzpicture}[x=1pt,y=0.65pt, scale = 0.65]
        \definecolor{fillColor}{RGB}{255,255,255}
        \path[use as bounding box,fill=fillColor,fill opacity=0.00] (0,0) rectangle (505.89,505.89);
        \begin{scope}
        \path[clip] (  8.25,  8.25) rectangle (500.39,500.39);
        \definecolor{drawColor}{RGB}{248,118,109}
        
        \path[draw=red,line width= 0.8pt,line join=round] (452.04, 44.48) --
        	(437.74, 60.09) --
        	(439.07, 83.79) --
        	(428.99, 94.53) --
        	(416.22,114.90) --
        	(406.78,131.27) --
        	(421.78,149.14) --
        	(427.00,157.15) --
        	(416.85,178.76) --
        	(400.45,192.15) --
        	(427.80,208.79) --
        	(420.65,236.44) --
        	(417.47,256.44) --
        	(415.29,269.55) --
        	(429.61,268.45) --
        	(417.28,284.90) --
        	(421.83,305.92) --
        	(419.00,320.12) --
        	(428.88,338.46) --
        	(437.84,345.07) --
        	(447.01,346.59) --
        	(446.01,366.49) --
        	(452.90,385.43) --
        	(433.87,399.83) --
        	(474.26,408.01) --
        	(476.09,417.71) --
        	(474.69,428.86) --
        	(478.02,446.77) --
        	(459.40,468.22) --
        	(472.22,478.02);
        \definecolor{drawColor}{RGB}{0,191,196}
        
        \path[draw=blue,line width= 0.8pt,line join=round] (440.85, 30.62) --
        	(434.72, 31.05) --
        	(420.30, 36.75) --
        	(406.55, 36.73) --
        	(387.28, 34.53) --
        	(368.55, 38.59) --
        	(368.99, 33.23) --
        	(359.01, 31.79) --
        	(341.25, 35.23) --
        	(321.55, 40.24) --
        	(295.16, 45.79) --
        	(284.21, 49.91) --
        	(260.97, 51.01) --
        	(240.23, 61.31) --
        	(221.90, 59.96) --
        	(223.88, 71.65) --
        	(223.44, 76.00) --
        	(233.01, 79.25) --
        	(218.20, 81.51) --
        	(227.86, 82.07) --
        	(206.65, 94.84) --
        	(201.95,102.73) --
        	(158.55,121.33) --
        	(150.06,133.21) --
        	(131.40,135.23) --
        	(108.74,147.17) --
        	( 67.60,155.49) --
        	( 62.28,165.24) --
        	( 36.07,171.12) --
        	( 30.62,180.27);
        \definecolor{drawColor}{RGB}{0,0,0}
        
        \path[draw=drawColor,line width= 0.8pt,line join=round] (452.04, 44.48) --
        	(440.85, 30.62) --
        	(359.01, 31.79) --
        	(341.25, 35.23) --
        	(260.97, 51.01) --
        	(221.90, 59.96) --
        	( 36.07,171.12) --
        	( 30.62,180.27) --
        	(472.22,478.02) --
        	(478.02,446.77) --
        	cycle;
        \end{scope}
    \end{tikzpicture}
    \caption{The convex hull of two independent planar random walks.}
    \label{fig:sketch}
\end{figure}
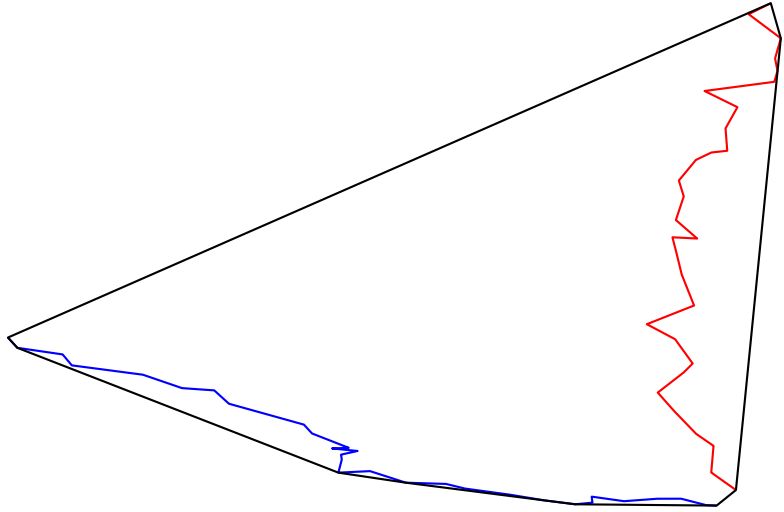

The central question we aim to address is whether we observe approximate normality of centered and scaled perimeter and diameter functionals of the random walks. Understanding these questions is of interest even from an applied point of view, not just theoretical. For instance, a question that ecologists often face, particularly in designing a conservation area to preserve a given animal population \cite{murphy1992integrating}, is how to estimate the home range of this animal population. Different methods are used to estimate this territory, based on monitoring the animals' positions \cite{giuggioli2006theory, worton1995convex}. One of these consists of simply the minimum convex polygon enclosing all monitored positions, that is, the convex hull. While this may seem simple-minded, it remains, under certain circumstances, the best way to proceed \cite{boyle2008home}. Even though the cited papers are all dealing with modeling the home-range of one roaming animal, having this particular application in mind, convex hulls of multiple random walks or Brownian motions have attracted more attention recently (see, e.g.~\cite{rfmc,dewenter2016convex,Majumdar,randon2021convex,randon2012convex}), and this is precisely the direction we take in this work.

Let $(Z_i^{(1)})_{i = 1}^{\infty}$ and $(Z_i^{(2)})_{i = 1}^{\infty}$ be two sequences of independent and identically distributed planar random vectors, which are mutually independent, but not necessarily identically distributed. Let also $(S_n^{(1)})_{n = 0}^{\infty}$ and $(S_n^{(2)})_{n = 0}^{\infty}$ be the corresponding random walks:
\begin{align*}
    S_0^{(k)} := 0, \qquad S_n^{(k)} := \sum_{i = 1}^n Z_i^{(k)},\qquad k=1,2. 
\end{align*}

The main objects we focus on in this paper are the perimeter process
\begin{align*}
    L_n := \Per \left( \text{chull} \left\{ S_j^{(k)} : 0 \leq j \leq n, \ k = 1, 2 \right\} \right), \quad n \ge 0,
\end{align*}
and the diameter process
\begin{align*}
    D_n := \diam \left( \text{chull} \left\{ S_j^{(k)} : 0 \leq j \leq n, \ k = 1, 2 \right\} \right), \quad n \ge 0.
\end{align*}
Here, $\operatorname{Per}(A)$, $\operatorname{diam}(A)$, and $\text{chull}(A)$ stand, respectively, for the perimeter, the diameter, and the convex hull of the set $A \subseteq \mathbb{R}^2$. Notice that the set $\operatorname{chull} \{ S_j^{(k)} : 0 \leq j \leq n, \ k = 1, 2 \}$ is a.s.\ a polygon.

This paper relies on the techniques and ideas developed in \cite{wade2015convex, mcredmond2018convex}, where the strong law of large numbers and the central limit theorem for the perimeter and the diameter of the convex hull spanned by the first $n$ steps of a single planar random walk has been considered.

The rest of the paper is organized as follows. In Section \ref{sec:main_results}, we present and comment the main results of this paper. Section \ref{sec:SLLN} is devoted to the proof of the strong law of large numbers for the convex hull spanned by the first $n$ steps of independent random walks. In Sections \ref{sec:perimeter} and \ref{sec:diameter}, we present the proofs of the main results concerning the perimeter and the diameter of the convex hull of two independent planar random walks. In the last two sections, Sections \ref{sec:final_discussions} and \ref{sec:simulations}, we discuss possible extensions of the main results, and pose several related open problems.

\section{The main results}\label{sec:main_results}

Our first main result is the strong law of large numbers for the set $\chull \{ S_j^{(k)} : 0 \leq j \leq n, \ k = 1, 2 \}$. Assuming that sequences $(Z_i^{(k)})_{i = 1}^{\infty}$, $k \in \{1,2\}$, have finite first moment, and denoting the drift vector of the $k$-th random walk by $\boldsymbol{\mu}^{(k)} = \E[Z_1^{(k)}]$, we have the following result.

\begin{theorem}\label{tm:SLLN} In the metric space of convex and compact planar sets endowed with the Hausdorff metric, it holds that
    \begin{equation*}
        n^{-1} \chull \left\{ S_j^{(k)} : 0 \leq j \leq n, \ k = 1, 2 \right\} \xrightarrow[n \to \infty]{a.s.} \chull \{ \boldsymbol{0}, \boldsymbol{\mu}^{(1)}, \boldsymbol{\mu}^{(2)}\}.
    \end{equation*}
\end{theorem}

Due to continuity of the perimeter and the diameter functionals (see \cite[Lemma 5.7 and Lemma 6.7]{lo2018functional}), from Theorem \ref{tm:SLLN} it immediately follows that the processes $(L_n)_{n = 0}^{\infty}$ and $(D_n)_{n = 0}^{\infty}$ converge a.s.\ to the perimeter, respectively, the diameter of the (possibly degenerate) triangle spanned by $\boldsymbol{\mu}^{(1)}$ and $\boldsymbol{\mu}^{(2)}$, that is,
\begin{equation}\label{eq:as_kvg_Ln_Dn}
    \frac{L_n}{n} \xrightarrow[n \to \infty]{a.s.} \Per\left( \chull \{\boldsymbol{0}, \boldsymbol{\mu}^{(1)}, \boldsymbol{\mu}^{(2)}\} \right), \qquad \textnormal{and} \qquad \frac{D_n}{n} \xrightarrow[n \to \infty]{a.s.} \diam\left( \chull \{\boldsymbol{0}, \boldsymbol{\mu}^{(1)}, \boldsymbol{\mu}^{(2)}\} \right).
\end{equation}
As a consequence of Pratt's lemma \cite[Theorem 5.5]{gut2006probability}, in Corollaries \ref{cor:L1_cvg_per} and \ref{cor:L1_cvg_diam} we show that the above convergences hold also in $L^1$.
In what follows, we investigate error terms of these approximations in terms of the central limit theorem for both processes.

We now introduce some notation that will be used throughout the paper. For $\theta \in [0, 2\pi)$, we let $\e_\theta = (\cos \theta, \sin \theta)$ be the unit vector pointing in the direction corresponding to this angle. When the  sequences $(Z_i^{(k)})_{i = 1}^{\infty}$, $k \in \{1,2\}$, have finite second moment,   the associated covariance matrices are denoted by $\boldsymbol{\Sigma}^{(k)} = \mathbb{E}[(Z_1^{(k)} - \boldsymbol{\mu}^{(k)})(Z_1^{(k)} - \boldsymbol{\mu}^{(k)})^T]$. Expressing drift vectors $\boldsymbol{\mu}^{(k)},$ $k \in \{1,2\}$, in polar coordinates, we have 
$$\boldsymbol{\mu}^{(k)} = \mu^{(k)} \e_{\theta^{(k)}},$$
where $\theta^{(k)} \in [0, 2\pi)$ represents the angle between the drift vector and the positive part of the $x$-axis, and $\mu^{(k)} \geq 0$ stands for the length of the vector $\boldsymbol{\mu}^{(k)}$. Let $\theta^{(0)} \in [0, 2\pi)$ be an angle satisfying the condition 
$$
\boldsymbol{\mu}^{(1)} \cdot \e_{\theta^{(0)}} =  \boldsymbol{\mu}^{(2)} \cdot \e_{\theta^{(0)}}.
$$ 
Here, $x\cdot y$ stands for the standard scalar product of $x,y\in\mathbb{R}^2$. In an intuitive sense, $\theta^{(0)}$ is the direction along which the projections of the drift vectors are equal. We also define $\e_{\theta^{(0)}}^\perp$, the unit vector perpendicular to this common projection line, subject to the constraint that $\e_{\theta^{(0)}}^\perp \cdot \mathbf{e}_{\theta^{(1)}} \geq 0$. 

Before stating our remaining main results, we introduce and discuss an assumption which we impose on the drift vectors $\boldsymbol{\mu}^{(1)}$ and $\boldsymbol{\mu}^{(2)}$:
\begin{equation} \tag{A1} \label{eq:per_assumption}
    \boldsymbol{0} \notin \{\boldsymbol{\mu}^{(1)}, \boldsymbol{\mu}^{(2)}, \boldsymbol{\mu}^{(1)} - \boldsymbol{\mu}^{(2)}\}.
\end{equation}
In the case of a single planar zero-drift random walk, in \cite{wade2015convex2} it has been shown that the process $(L_n)_{n = 0}^{\infty}$ (with the classical central limit theorem centering and scaling) has a non-Gaussian distributional limit. Analogously, in the case of two independent planar random walks, we conjecture that if the assumption \eqref{eq:per_assumption} is not satisfied, we can again expect a non-Gaussian distributional limit. Unfortunately, we have not been able to carry out a rigorous proof of this conjecture. See Section \ref{sec:simulations} for a  computer simulation study and discussion that support the conjecture. We now present our second main result, which provides the $L^2$ approximation of the perimeter process $(L_n)_{n = 0}^{\infty}$.

\begin{theorem} \label{theorem - L2 convergence - perimeter}
Assume \eqref{eq:per_assumption}. Then,
\begin{align*} 
    n^{-1 / 2}\left|L_n-\mathbb{E}\left[L_n\right]
    -\sum_{i=1}^n \left[ (Z_i^{(1)} 
    - \boldsymbol{\mu}^{(1)}) \cdot (\e_{\theta^{(0)}}^\perp + \e_{\theta^{(1)}}) + (Z_i^{(2)} - \boldsymbol{\mu}^{(2)}) \cdot (\e_{\theta^{(2)}} - \e_{\theta^{(0)}}^\perp) \right] \right| \xrightarrow[n \to \infty]{L^2} 0.
\end{align*}
\end{theorem}

Intuitively, according to Theorem \ref{tm:SLLN}, the set $\operatorname{chull} \{ S_j^{(k)} : 0 \leq j \leq n, \ k = 1, 2 \}$ can be approximated (with respect to the Hausdorff metric) by the scaled (possibly degenerate) triangle spanned by the drift vectors $\boldsymbol{\mu}^{(1)}$ and $\boldsymbol{\mu}^{(2)}$. Theorem \ref{theorem - L2 convergence - perimeter} analyses the error of this approximation, which is decomposed into three parts. The parts $\sum_{i=1}^n (Z_i^{(k)} - \boldsymbol{\mu}^{(k)}) \cdot \e_{\theta^{(k)}}$, $k\in\{1,2\},$  represent the deviation in the direction of the corresponding drift vectors $\boldsymbol{\mu}^{(k)}$, and the remaining expression, 
$$
\sum_{i=1}^n ((Z_i^{(1)} - \boldsymbol{\mu}^{(1)}) -( Z_i^{(2)} - \boldsymbol{\mu}^{(2)})) \cdot \e_{\theta^{(0)}}^\perp, 
$$ 
corresponds to the deviation along the third side of the triangle, the one connecting two drift vectors. In the case of the diameter functional, in addition to assumption \eqref{eq:per_assumption}, we assume the following:
\begin{equation} \tag{A2} \label{eq:diam_assumption}
    \textnormal{the set } \{\|\boldsymbol{\mu}^{(1)}\|, \|\boldsymbol{\mu}^{(2)}\|, \|\boldsymbol{\mu}^{(1)} - \boldsymbol{\mu}^{(2)}\|\} \textnormal{ has a unique maximal element}.
\end{equation}
Here, $\|x\|$ stands for the standard Euclidean  norm of $x\in\mathbb{R}^2.$ Assumption \eqref{eq:per_assumption} is crucial for the same reason as in the case of the perimeter process, while assumption \eqref{eq:diam_assumption} allows us to identify the direction of the diameter of the set. We conjecture that, in the case of the diameter process, we again have non-Gaussian distributional limit in the case when assumptions \eqref{eq:per_assumption} or \eqref{eq:diam_assumption} are not satisfied (see Section \ref{sec:simulations} for a  computer simulation study and discussion that support the conjecture). We now state our third main result.

\begin{theorem} \label{theorem - diameter L2 convergence}
Assume \eqref{eq:per_assumption}, \eqref{eq:diam_assumption}, and that the maximal element of the set from the assumption \eqref{eq:diam_assumption} is $\|\boldsymbol{\mu}^{(1)}\|$. Then,
$$
n^{-1 / 2}\left|D_n-\mathbb{E} D_n-\left(S_n^{(1)}-\mathbb{E} S_n^{(1)}\right) \cdot \e_{\theta^{(1)}}\right|  \xrightarrow[n \to \infty]{L^2} 0. 
$$
\end{theorem}

If the maximal element is $\|\boldsymbol{\mu}^{(2)}\|$, our proof follows in the same manner, only interchanging first and second random walk, while in the third scenario we consider the difference of the two walks and the angle $\theta^{(k)}$ is replaced by the angle corresponding to the direction of the vector $\boldsymbol{\mu}^{(1)} - \boldsymbol{\mu}^{(2)}$, see Section \ref{sec:final_discussions} for details.

In order to obtain the central limit theorem for the perimeter and diameter processes, we first have to determine the variance of the limiting normal law.

\begin{theorem} \label{theorem - variance asymptotic - perimeter}
Assume \eqref{eq:per_assumption}. Then,
$$
\lim _{n \rightarrow \infty} \frac{\V \left[L_n\right]}{n}=\sigma^2_L \in [0, \infty),
$$
where
\begin{align*}
    \sigma^2_L = &
    \E[((Z_1^{(1)} - \boldsymbol{\mu}^{(1)}) \cdot\e_{\theta^{(1)}})^2] + \E[((Z_1^{(1)} - \boldsymbol{\mu}^{(1)}) \cdot\e_{\theta^{(0)}}^\perp)^2] + 2 (\boldsymbol{\Sigma}^{(1)} \e_{\theta^{(1)}}) \cdot \e_{\theta^{(0)}}^\perp \\
     &+
    \E[((Z_1^{(2)} - \boldsymbol{\mu}^{(2)}) \cdot\e_{\theta^{(2)}})^2] + \E[((Z_1^{(2)} - \boldsymbol{\mu}^{(2)}) \cdot\e_{\theta^{(0)}}^\perp)^2] - 2(\boldsymbol{\Sigma}^{(2)} \e_{\theta^{(2)}}) \cdot \e_{\theta^{(0)}}^\perp.
\end{align*}
\end{theorem}

It may be observed that $\sigma_L^2$ represents the variance of an individual term in the approximating sum presented in Theorem \ref{theorem - L2 convergence - perimeter}.

\begin{theorem} \label{theorem - variance asymptotic - diameter}
Assume \eqref{eq:per_assumption}, \eqref{eq:diam_assumption}, and that the maximal element of the set from assumption \eqref{eq:diam_assumption} is $\|\boldsymbol{\mu}^{(1)}\|$. Then,
$$
\lim _{n \rightarrow \infty} \frac{\V \left[D_n\right]}{n}=\sigma^2_D \in [0, \infty),
$$
where
\begin{align*}
    \sigma^2_D = \E \left[ \left(\left(Z_1^{(1)}-\boldsymbol{\mu}^{(1)}\right) \cdot \mathbf{e}_{\theta^{(1)}} \right)^2 \right].  
\end{align*}
\end{theorem}
If the maximal element is not $\|\boldsymbol{\mu}^{(1)}\|$, $\sigma_D^2$ is 
modified as commented above. Finally, we state the central limit theorems for both processes.

\begin{theorem} \label{theorem - CLT - perimeter}
Assume \eqref{eq:per_assumption}, and $\sigma^2_L > 0$. Then, for any $x \in \mathbb{R}$,
\begin{align*}
\lim _{n \rightarrow \infty} \mathbb{P}\left[\frac{L_n-\mathbb{E}\left[L_n\right]}{\sqrt{\mathbb{V} \operatorname{ar}\left[L_n\right]}} \leq x\right]=\lim _{n \rightarrow \infty} \mathbb{P}\left[\frac{L_n-\mathbb{E}\left[L_n\right]}{\sqrt{\sigma^2_L n}} \leq x\right]=\Phi(x),
\end{align*}
where $\Phi$ stands for the cummulative distribution function of the standard normal distribution.
\end{theorem}
In the same manner, we establish the central limit theorem for the behavior of the diameter process.

\begin{theorem} \label{theorem - CLT - diameter}
Assume \eqref{eq:per_assumption}, \eqref{eq:diam_assumption}, and $\sigma^2_D > 0$. Then, for any $x \in \mathbb{R}$,
\begin{align*}
\lim _{n \rightarrow \infty} \mathbb{P}\left[\frac{D_n-\mathbb{E}\left[D_n\right]}{\sqrt{\V \left[D_n\right]}} \leq x\right]=\lim _{n \rightarrow \infty} \mathbb{P}\left[\frac{D_n-\mathbb{E}\left[D_n\right]}{\sqrt{\sigma^2_D n}} \leq x\right]=\Phi(x).
\end{align*}
\end{theorem}

Notice that for $\sigma_D^2$ to be strictly positive it is sufficient, and necessary, that the variance of the projection of the first random walk onto the vector $\e_{\theta^{(1)}}$ is non-zero. When this variance is zero the walk is characterized by deterministic (rather than random) behavior along this particular direction. On the other hand, $\sigma_L^2$ will be strictly positive if, and only if, either the variance of the projection of the first random walk onto the vector $\e_{\theta^{(0)}}^\perp + \e_{\theta^{(1)}}$ is non-zero, or the variance of the projection of the second walk onto the vector $\e_{\theta^{(2)}} - \e_{\theta^{(0)}}^\perp$ is non-zero.

\section{Strong Law Of Large Numbers}\label{sec:SLLN}

In this section, we show the strong law of large numbers stated in Theorem \ref{tm:SLLN}. We show the result in a slightly more general setting, namely for arbitrarily many, say $m \ge 1$, independent random walks $(S_n^{(k)})_{n = 0}^{\infty}$, with drift vectors $\boldsymbol{\mu}^{(k)}$, $k \in \{1, 2, \ldots, m\}$, in arbitrary dimension $d \ge 1$. In particular, we show that in the metric space of convex and compact $d$-dimensional sets endowed with the Hausdorff metric, it holds that
    \begin{equation*}
        n^{-1} \chull \left\{ S_j^{(k)} : 0 \leq j \leq n, \ k = 1, \dots,m \right\} \xrightarrow[n \to \infty]{a.s.} \chull 
\left\{
 \{\boldsymbol{0}\} \cup \left\{ \boldsymbol{\mu}^{(k)}: k=1, \ldots, m \right\} \right\}.
    \end{equation*}

We first recall the definition of the Hausdorff metric. Let $(X, \rho)$ be a metric space and let $A,B\subseteq X$. The Hausdorff distance of $A$ and $B$ is defined as 
$$
\rho_H(A, B) := \max \left\{ \sup_{x \in A} \rho(x, B),\, \sup_{y \in B} \rho(y, A) \right\},
$$
or, equivalently, 
$$
\rho_H(A, B) := \inf \left\{ \varepsilon \geq 0 : A \subseteq B^{\varepsilon} \text{ and } B \subseteq A^{\varepsilon} \right\},
$$
where $\rho(x,A):=\inf_{y\in A}\rho(x,y)$ and $A^{\varepsilon} = \left\{ x \in X : \rho(x, A) \leq \varepsilon \right\}$. One can demonstrate that these two definitions are indeed equivalent (see \cite[Section 1.8]{schneider2014convex}). In order to obtain a proper metric space, we restrict $\rho_H$ to closed and bounded subsets of $X$.  In the case when $X$ is $d$-dimensional Euclidean space $\mathbb{R}^d$ equipped with the standard Euclidean distance, the corresponding Hausdorff distance is denoted by $\rho_H^d$.

\begin{proof}[Proof of Theorem \ref{tm:SLLN}]
    Observe first that for sequences $(A_n^{(k)})_{n = 0}^{\infty}$, $k \in \{1, \ldots, m\}$, of (closed) subsets in $\mathbb{R}^d$ that  converge, respectively, to (closed) sets $A^{(k)}$ (for $k = 1, \dots, m$), their union converges to the union of the limiting subsets (with respect to $\rho_H^d$). Namely, it is sufficient to prove that
    \begin{align} \label{hausdroff unije}
    \rho_H^d \left( 
    \bigcup_{k = 1}^{m} A_n^{(k)}, 
    \bigcup_{k = 1}^{m} A^{(k)}
    \right) 
    \leq 
    \max_{k \in \{1, \dots, m\}} \rho_H^d \left( 
    A_n^{(k)}, 
    A^{(k)}
    \right) . 
    \end{align} 
Let
$$
\varepsilon = \max_{k \in \{1, \dots, m\}} \rho_H^d(A_n^{(k)}, A^{(k)}).
$$
From the definition of the Hausdorff metric, we have that $A_n^{(k)} \subseteq (A^{(k)})^\varepsilon$ and $A^{(k)} \subseteq (A_n^{(k)})^\varepsilon$ holds for all $k \in \{1, \ldots, m\}$. Consequently,
\begin{align*}
\left( \bigcup_{k = 1}^{m} A^{(k)} \right)^\varepsilon
&= (A^{(1)})^\varepsilon \cup (A^{(2)})^\varepsilon \cup \ldots \cup (A^{(m)})^\varepsilon \\
&\supseteq A_n^{(1)} \cup A_n^{(2)} \cup \ldots \cup A_n^{(m)} \\
&= \bigcup_{k = 1}^{m} A_n^{(k)}.
\end{align*}
Analogously, we deduce
\begin{align*}
\left( \bigcup_{k = 1}^{m} A_n^{(k)} \right)^\varepsilon
&\supseteq \bigcup_{k = 1}^{m} A^{(k)}.
\end{align*}
Thus,
$$
\rho_H^d \left( 
\bigcup_{k = 1}^{m} A_n^{(k)}, 
\bigcup_{k = 1}^{m} A^{(k)}
\right) 
\leq \varepsilon,
$$
which thereby proves (\ref{hausdroff unije}). Next, in \cite[Theorem 5.4]{lo2018functional} it is established that for a single random walk $(S_n)_{n = 0}^{\infty}$ with drift $\boldsymbol{\mu}$,
$$
n^{-1}\left\{ S_0, S_1, \ldots, S_n \right\}\xrightarrow[n \to \infty]{a.s.}\left\{t\boldsymbol{\mu} : t\in [0, 1]\right\}.
$$ 
Hence,
$$
n^{-1}\left\{ S_j^{(k)}: \ 0 \leq j \leq n \right\} \xrightarrow[n \to \infty]{a.s.}\left\{ t \boldsymbol{\mu}^{(k)}: \ t \in [0, 1] \right\},
$$
for $k \in \{1, \ldots, m\}$, and, by applying (\ref{hausdroff unije}), we establish
$$
n^{-1}\left\{ S_j^{(k)}: \ 0 \leq j \leq n, \ k = 1, \ldots, m \right\} \xrightarrow[n \to \infty]{a.s.}\left\{ t \boldsymbol{\mu}^{(k)}: \ t \in [0, 1], \ k = 1, \ldots, m \right\}.
$$
Finally, using \cite[Lemma 6.1]{lo2018functional} the result follows.
\end{proof} 

Recall that intrinsic volumes $V_1, \ldots, V_d$ are the classical geometric functionals of $d$-dimensional convex and compact sets. It is known that $V_1$ is proportional to the mean width of the set, $V_{d-1}$  equals   one half of the surface area of the set, while $V_{d}$ is the volume of the set.
Furthermore, it is known that all these functionals are continuous mappings (with respect to the Hausdorff metric), and the $\ell$-th intrinsic volume $V_{\ell}$ is homogeneus of degree $\ell$, that is, $V_{\ell}(cA) = c^{\ell}V_{\ell}(A)$, for any $c \ge 0$, see \cite[Theorem III.1.1]{baddeley2007random} and \cite[Lemma 1.8.14]{schneider2014convex}.
As a consequence of Theorem \ref{tm:SLLN} we now conclude
\begin{equation*}
    \frac{V_{\ell}\left( \chull
\left\lbrace
S_j^{(k)}: \ 0 \leq j \leq n, \ k = 1, \dots, m
\right\rbrace  \right)}{n^{\ell}} \xrightarrow[n \to \infty]{a.s.} V_{\ell}\left( \chull
\left\{
 \{\boldsymbol{0}\} \cup \left\{ \boldsymbol{\mu}^{(k)}: k=1, \ldots, m \right\} \right\} \right).
\end{equation*}
Observe that the above limit is non-trivial if, and only if, there are at least $\ell$ linearly independent vectors in the set $\{ \boldsymbol{\mu}^{(k)}: k=1, \ldots, m \}$. From this we conclude that in the planar case we cannot expect a non-trivial limit for the area functional of the convex hull of a single random walk under $n^2$ scaling. In \cite[Corollary 2.8]{wade2015convex2} the authors show that the appropriate scaling in this case is $n^{3/2}$, and establish convergence in distribution to a non-degenerate limit.

\section{Perimeter}\label{sec:perimeter}

In this section, we discuss the limiting behavior of the perimeter process. Our proofs rely on martingale difference sequences and the Cauchy formula for the perimeter.

\subsection{Martingale difference sequence}

Let $\mathcal{F}_0 = \{ \emptyset, \Omega \}$, and
\begin{align*}
    \mathcal{F}_n = \sigma \left( S_j^{(k)} : 0 \leq j \leq n, \ k = 1, 2 \right), \qquad n \ge 1,
\end{align*}
be the information about both random walks up to time $n$. Further, let $(\tilde{Z}_i^{(1)})_{i = 1}^{\infty}$ and $(\tilde{Z}_i^{(2)})_{i = 1}^{\infty}$ be independent copies of $(Z_i^{(1)})_{i = 1}^{\infty}$ and $(Z_i^{(2)})_{i = 1}^{\infty}$, which are also mutually independent. For a fixed $i \ge 1$ the resampled random walk at time $i$ is defined by
\begin{align} \label{definicija resamplanih šetnji}
    S_j^{(k, i)} := \left\{ \begin{array}{cc}
       S_j^{(k)},& j < i,  \\
       S_j^{(k)}-Z_i^{(k)} + \tilde{Z}_i^{(k)},& j \geq i.
    \end{array} \right. 
\end{align}
The corresponding perimeter processes are given as before, 
\begin{align*}
    L_n^{(i)} := \operatorname{Per} \left( \text{chull} \left\{ S_j^{(k, i)} : 0 \leq j \leq n, \ k = 1, 2 \right\} \right). 
\end{align*}
In the following lemma we show that
\begin{align*}
    \mathcal{L}_{n, i}:=\mathbb{E}\left[L_{n}-L_{n}^{(i)} \mid \mathcal{F}_{i}\right], \qquad 1\le i \le n,
\end{align*}
is a martingale difference sequence (see \cite[p. 230]{davidson1994stochastic}).

\begin{lemma}\label{lemma - MDS - perimeter}
Let $n \in \mathbb{N}$. Then, 
\begin{itemize}
    \item[(i)] $L_{n}-\mathbb{E}\left[L_{n}\right]=\sum_{i=1}^{n} \mathcal{L}_{n, i}$, 
    \item[(ii)] $\operatorname{Var}\left[L_{n}\right]=\sum_{i=1}^{n} \mathbb{E}\left[\mathcal{L}_{n, i}^{2}\right]$, whenever the latter sum is finite.
\end{itemize}
\end{lemma}

\begin{proof}
Observe that $L_{n}^{(i)}$ is independent of $Z_{i}^{(k)}$ for both $k\in\{1,2\}$, so that 
$$
\mathbb{E}[L_{n}^{(i)} \mid \mathcal{F}_{i}]=\mathbb{E}[L_{n}^{(i)} \mid \mathcal{F}_{i-1}]=\mathbb{E}\left[L_{n} \mid \mathcal{F}_{i-1}\right]. 
$$
Hence, $\mathcal{L}_{n,i}$ can be expressed as
\begin{align*}
\mathcal{L}_{n, i}=\mathbb{E}\left[L_{n} \mid \mathcal{F}_{i}\right]-\mathbb{E}\left[L_{n} \mid \mathcal{F}_{i-1}\right] .
\end{align*}
Summing over $1 \leq i \leq n$, we conclude $\sum_{i=1}^{n} \mathcal{L}_{n, i}=\mathbb{E}\left[L_{n} \mid \mathcal{F}_{n}\right]-\mathbb{E}\left[L_{n} \mid \mathcal{F}_{0}\right] = L_n - \mathbb{E}[L_n]$, which gives $(i)$. The claim in $(ii)$ follows from the martingale difference property of the sequence $(\mathcal{L}_{n, i})_{i = 1}^n$.
\end{proof}

\subsection{Cauchy formula for the perimeter}

One of the most important contributions to convex analysis is the Cauchy formula for the perimeter (see \cite[Theorem 6.15.]{gruber2007convex}). For $\theta \in [0, \pi]$, let us define 
\begin{align*}
    M_{n}(\theta):=\max _{\substack{0 \leq j \leq n \\ k = 1, 2}}\left(S_{j}^{(k)} \cdot \mathbf{e}_{\theta}\right), \qquad m_{n}(\theta):=\min _{\substack{0 \leq j \leq n \\ k = 1, 2}}\left(S_{j}^{(k)} \cdot \mathbf{e}_{\theta}\right). 
\end{align*}
For a given angle $\theta$, the terms $M_n(\theta)$ and $m_n(\theta)$ denote the maximal and minimal projections, respectively, of the convex hull onto a line passing through the origin and directed by the unit vector $\mathbf{e}_\theta$. Since $S_0^{(k)} = 0$, it is clear that $M_n(\theta) \geq 0$ and $m_n(\theta) \leq 0$ a.s. The Cauchy formula expresses the perimeter of the convex set in terms of $M_n(\theta)$ and $m_n(\theta)$:
\begin{align*}
    L_{n}=\int_{0}^{\pi}\left(M_{n}(\theta)-m_{n}(\theta)\right) \mathrm{d} \theta=\int_{0}^{\pi} R_{n}(\theta) \mathrm{d} \theta, 
\end{align*}
where $R_{n}(\theta):=M_{n}(\theta)-m_{n}(\theta) \geq 0$ is called the parametrized range function. Notice that the Cauchy formula for the perimeter can be equivalently stated as
\begin{equation}\label{eq:cauchy_formula_0-2pi}
    L_n = \int_0^{2\pi} M_n(\theta) d\theta.
\end{equation}
We similarly have that
\begin{align*}
    L_{n}^{(i)}=\int_{0}^{\pi}\left(M_{n}^{(i)}(\theta)-m_{n}^{(i)}(\theta)\right) \mathrm{d} \theta=\int_{0}^{\pi} R_{n}^{(i)}(\theta) \mathrm{d} \theta
\end{align*}
with $R_{n}^{(i)}(\theta)=M_{n}^{(i)}(\theta)-m_{n}^{(i)}(\theta)$ and
\begin{align*}
    M_{n}^{(i)}(\theta):=\max _{\substack{0 \leq j \leq n \\ k = 1, 2}}\left(S_{j}^{(k, i)} \cdot \mathbf{e}_{\theta}\right), \qquad m_{n}^{(i)}(\theta):=\min _{\substack{0 \leq j \leq n \\ k = 1, 2}}\left(S_{j}^{(k, i)} \cdot \mathbf{e}_{\theta}\right). 
\end{align*}
We consider the following difference
\begin{align*}
    L_{n}-L_{n}^{(i)}=\int_{0}^{\pi}\left(R_{n}(\theta)-R_{n}^{(i)}(\theta)\right) \mathrm{d} \theta=\int_{0}^{\pi} \Delta_{n}^{(i)}(\theta) \mathrm{d} \theta,
\end{align*}
where $\Delta_{n}^{(i)}(\theta) := R_{n}(\theta) - R_{n}^{(i)}(\theta)$. We define two random variables for an angle $\theta \in [0, \pi]$. The first random variable represents the last time at which the minimal projections of both the first and the second random walk are achieved. Conversely, the second random variable denotes the first time at which the maximal projections of both random walks are attained. Formally: 
\begin{align*}
    \underline{J}_{n, k}(\theta):= \max \left\{ \underset{0 \leq j \leq n}{\arg \min }\left(S_{j}^{(k)} \cdot \mathbf{e}_{\theta}\right) \right\}, \qquad \text { and } \qquad \overline{J}_{n, k}(\theta):= \min \left\{ \underset{0 \leq j \leq n}{\arg \max }\left(S_{j}^{(k)} \cdot \mathbf{e}_{\theta}\right) \right\}. 
\end{align*}
Notice that we record these time instances for each walk individually. For the resampled walks, we analogously define variables $ \underline{J}_{n, k}^{(i)}(\theta)$ and $ \overline{J}_{n, k}^{(i)}(\theta)$. We further introduce the random variables $\underline{\mathcal{I}}_n(\theta)$ and $\overline{\mathcal{I}}_n(\theta)$, which denote the indices of the random walks ($k = 1$, or $k = 2$) where the minimum and maximum projections are reached, respectively. In the event of a tie, the default choice is $k = 1$. Analogously, we define the variables $\underline{\mathcal{I}}^{(i)}_n(\theta)$ and $\overline{\mathcal{I}}^{(i)}_n(\theta)$.

Throughout the subsequent proofs, we frequently require that the variable $\Delta_n^{(i)}(\theta)$ is dominated by an integrable random variable.

\begin{lemma} \label{lemma - omedjenost integrabilnom varijablom}
For any $1 \le i \le n$, we have that 
\begin{align*}
    \sup\limits_{\theta\in[0,\pi]}|\Delta_{n}^{(i)}(\theta)| \leq 2 \left( \|Z_{i}^{(1)}\|+ \|\tilde{Z}_{i}^{(1)}\|
    + \|Z_{i}^{(2)}\|+ \|\tilde{Z}_{i}^{(2)}\|\right) . 
\end{align*}
\end{lemma}

\begin{proof}
Take an arbitrary $\theta\in[0,\pi]$. By definition, we have that
\begin{align*}
    M_{n}(\theta) = S_{\overline{J}_{n, \overline{\I}_n}(\theta)}^{(\overline{\I}_n)} \cdot \e_\theta. 
\end{align*}
Thus,
\begin{align*}
    M^{(i)}_{n}(\theta) \geq S_{\overline{J}_{n,\overline{\I}_n }(\theta)}^{(\overline{\I}_n, i)} \cdot \e_\theta. 
\end{align*}
If $\overline{J}_{n, \overline{\I}_n}(\theta) < i$, then $S_{\overline{J}_{n,\overline{\I}_n}(\theta)}^{(\overline{\I}_n, i)}=S_{\overline{J}_{n,\overline{\I}_n}(\theta)}^{(\overline{\I}_n)}$, so $M^{(i)}_{n}(\theta) \geq M_{n}(\theta)$. On the other hand, if $\overline{J}_{n, \overline{\I}_n}(\theta) \geq i$, we have that
\begin{align*}
    S_{\overline{J}_{n, \overline{\I}_n}(\theta)}^{(\overline{\I}_n, i)} 
    = 
    S_{\overline{J}_{n, \overline{\I}_n}(\theta)}^{(\overline{\I}_n)} 
    - 
    (Z_i^{(\overline{\I}_n)} - \tilde{Z}_i^{(\overline{\I}_n)}),
\end{align*}
so taking a projection in the direction of $\theta$ gives us
\begin{align*}
    M_{n}^{(i)}(\theta) &\geq 
    S_{\overline{J}_{n, \overline{\I}_n}(\theta)}^{(\overline{\I}_n)} \cdot \e_\theta 
    -
    (Z_i^{(\overline{\I}_n)}- \tilde{Z}_i^{(\overline{\I}_n)})\cdot\e_\theta  \\
    &\geq M_n(\theta) - \left( \|Z_{i}^{(1)}\|+ \|\tilde{Z}_{i}^{(1)}\|
    + \|Z_{i}^{(2)}\|+ \|\tilde{Z}_{i}^{(2)}\|\right). 
\end{align*}
In both cases, we have the lower bound on $M_n^{(i)}(\theta)$ as follows
\begin{align*}
    M_{n}^{(i)}(\theta) \geq M_n (\theta) - \left( \|Z_{i}^{(1)}\|+ \|\tilde{Z}_{i}^{(1)}\|
    + \|Z_{i}^{(2)}\|+ \|\tilde{Z}_{i}^{(2)}\|\right). 
\end{align*}
Similar arguments can be applied when the original and resampled maximal projections are interchanged, thereby demonstrating that
\begin{align*}
    |M_{n}^{(i)}(\theta)-M_{n}(\theta)| \leq
      \|Z_{i}^{(1)}\|+ \|\tilde{Z}_{i}^{(1)}\|
    + \|Z_{i}^{(2)}\|+ \|\tilde{Z}_{i}^{(2)}\| .
\end{align*}
The same approach can be employed to establish an analogous upper bound on $|m_{n}^{(i)}(\theta) - m_{n}(\theta)|$. With this, the assertion of the lemma is verified.
\end{proof}

Before moving onto the next subsection, we show that the convergence in the strong law of large numbers for the perimeter process, presented in \eqref{eq:as_kvg_Ln_Dn}, holds also in $L^1$ sense.

\begin{corollary}\label{cor:L1_cvg_per}
    Under the assumptions of Theorem \ref{tm:SLLN}, we have
    \begin{equation*}
        \frac{L_n}{n} \xrightarrow[n \to \infty]{L^1} \Per \left( \chull \{\boldsymbol{0}, \boldsymbol{\mu}^{(1)}, \boldsymbol{\mu}^{(2)}\} \right).
    \end{equation*}    
\end{corollary}
\begin{proof}
    Using the Cauchy formula from \eqref{eq:cauchy_formula_0-2pi} we have
    \begin{align*}
        L_n = \int_0^{2\pi} M_n(\theta) d\theta \le 2\pi \max_{\substack{0 \leq j \leq n \\ k = 1, 2}} \|S_j^{(k)}\| \le 2\pi \max_{k = 1, 2} \sum_{j = 0}^{n} \|Z_j^{(k)}\| \le 2\pi \sum_{j = 0}^{n} \left( \|Z_j^{(1)}\| + \|Z_j^{(2)}\| \right).
    \end{align*}
    Since $(Z_j^{(1)})_{j = 1}^{\infty}$ and $(Z_j^{(2)})_{j = 1}^{\infty}$ are sequences of i.i.d.\ random variables, from strong law we have that a.s., $n^{-1}\sum_{j = 1}^{n}(\|Z_j^{(1)}\| + \|Z_j^{(2)}\|) \to \mathbb{E}[\|Z_1^{(1)}\| + \|Z_1^{(2)}\|] < \infty$, and clearly $\mathbb{E}[n^{-1}\sum_{j = 1}^{n}(\|Z_j^{(1)}\| + \|Z_j^{(2)}\|)] = \mathbb{E}[\|Z_1^{(1)}\| + \|Z_1^{(2)}\|]$. Hence, Pratt's lemma \cite[Theorem 5.5]{gut2006probability} implies the claim.
\end{proof}

\subsection{Control of extrema}

To make the geometric analysis of the problem a little bit more convenient, we may restrict our attention to $\theta^{(1)}, \theta^{(2)} \in [0, \pi]$ such that the projections of their corresponding drift vectors onto the $y$-axis are equal. This simplification is justifiable due to the geometric properties of the convex hull, which remain unchanged under rotation and reflection operations. After performing these coordinate transformations, we find that we are left with two mutually exclusive scenarios:
\begin{itemize} \label{GA}
    \item [(i)] The first drift vector lies in the first quadrant, while the second is in the second quadrant. The $y$-axis effectively separates the two vectors.
    \item [(ii)] Both drift vectors lie in the first quadrant, with the first vector displaying a smaller angular displacement from the $x$-axis than the second one.
\end{itemize}

The described scenarios are illustrated in Figure \ref{fig:position of the drift vectors}. It should be emphasized that while our mathematical manipulations are made to address the first scenario, they are not restrained to it. Transitioning to the second scenario does not demand substantially altering the framework.

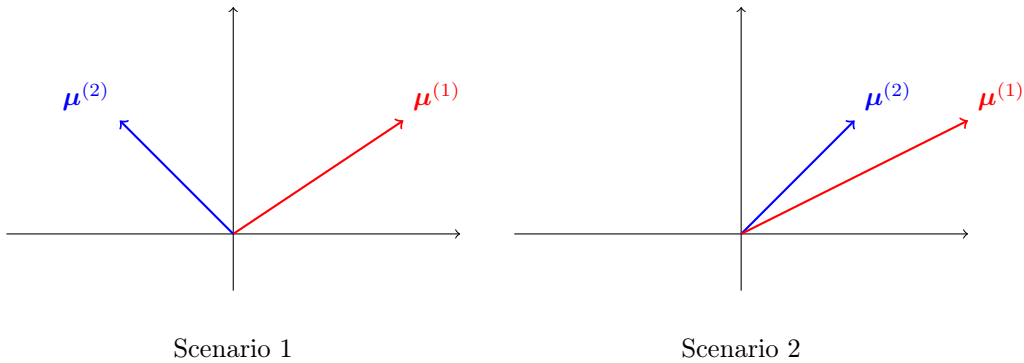
\begin{figure}[h]
    \centering
    \begin{tikzpicture}[scale=1.5]
    \draw[->] (-2,0) -- (2,0) node[right] {};
    \draw[->] (0,-0.5) -- (0,2) node[above] {};
    
    \draw[->, red, thick] (0,0) -- (1.5,1) node[above right] {$\boldsymbol{\mu}^{(1)}$};
    
    \draw[->, blue, thick] (0,0) -- (-1,1) node[above left] {$\boldsymbol{\mu}^{(2)}$};

    \node at (0,-1) {Scenario 1};
\end{tikzpicture}
\quad  
\begin{tikzpicture}[scale=1.5]
    \draw[->] (-2,0) -- (2,0) node[right] {};
    \draw[->] (0,-0.5) -- (0,2) node[above] {};
    
    \draw[->, blue, thick] (0,0) -- (1,1) node[above right] {$\boldsymbol{\mu}^{(2)}$};
    
    \draw[->, red, thick] (0,0) -- (2,1) node[above right] {$\boldsymbol{\mu}^{(1)}$};

    \node at (0,-1) {Scenario 2};
\end{tikzpicture}
    \caption{Possible positions of the drift vectors.}
    \label{fig:position of the drift vectors}
\end{figure}

Observe that $(S_{j}^{(k)} \cdot \mathbf{e}_{\theta})_{j = 0}^n$, $k\in \{1, 2\}$, are one-dimensional random walks with means 
$$
\boldsymbol{\mu}^{(k)} \cdot \e_\theta = \mu^{(k)} \cos(\theta^{(k)}-\theta). 
$$
For an arbitrary $\varepsilon > 0$, define the following subset of angles $\theta \in [0, \pi]$:
\begin{equation}\label{eq:def_of_Theta_sets}
    \begin{aligned}
        \Theta_{(1>2>0)}^\varepsilon & = \left\{ \theta \in [0, \pi]: \boldsymbol{\mu}^{(1)} \cdot \e_\theta - \boldsymbol{\mu}^{(2)} \cdot \e_\theta > \varepsilon, \quad \boldsymbol{\mu}^{(2)} \cdot \e_\theta > \varepsilon \right\},
        \\
        \Theta_{(2>1>0)}^\varepsilon & = \left\{ \theta \in [0, \pi]: \boldsymbol{\mu}^{(2)} \cdot \e_\theta  - \boldsymbol{\mu}^{(1)} \cdot \e_\theta > \varepsilon, \quad
        \boldsymbol{\mu}^{(1)} \cdot \e_\theta > \varepsilon \right\},
        \\
        \Theta_{(1>0>2)}^\varepsilon & = \left\{ \theta \in [0, \pi] : \boldsymbol{\mu}^{(1)} \cdot \e_\theta > \varepsilon,  \quad \boldsymbol{\mu}^{(2)} \cdot \e_\theta  < - \varepsilon  \right\},
        \\
        \Theta_{(2>0>1)}^\varepsilon & = \left\{ \theta \in [0, \pi] : \boldsymbol{\mu}^{(2)} \cdot \e_\theta > \varepsilon,  \quad \boldsymbol{\mu}^{(1)} \cdot \e_\theta  < - \varepsilon \right\},
        \\
        \Theta_{(0>2>1)}^\varepsilon &= \left\{ \theta \in [0, \pi]: \boldsymbol{\mu}^{(2)} \cdot \e_\theta  - \boldsymbol{\mu}^{(1)} \cdot \e_\theta > \varepsilon,  \ \boldsymbol{\mu}^{(2)} \cdot \e_\theta  < - \varepsilon \right\}.
    \end{aligned}
\end{equation}
We define these sets in order to divide our domain into segments where we have an information about the dominating drift vector, and positivity or negativity of the projections. In other words, we determine whether we contribute to the minimum or maximum of the projected line with each walk in each region. The subscripts in the set notations indicate what happens in each specific region.  For example, $\Theta_{(1>2>0)}^\varepsilon$ is the set of angles on which both drift vectors have a strictly positive projection (greater than some chosen $\varepsilon>0$), and the first vector has a projection that is larger for at least $\varepsilon$ than the projection of the second drift vector. On this set, with high probability, the first walk will contribute to the maximum, and the minimum will be achieved early enough. Similar reasoning can be applied to the rest of the subsets. The Figure \ref{fig:Theta areas} illustrates this division.

Because of the earlier discussion about rotations and reflections, we do not need to consider the set of angles in $[0, \pi]$ such that the projection of both walks have sufficiently negatively oriented drifts, and the projection of the first walk is sufficiently greater than the projection of the second walk.  We write 
$$\Theta^\varepsilon_{\text{I}} := \Theta_{(1>2>0)}^\varepsilon \cup \Theta_{(2>1>0)}^\varepsilon, 
\qquad  
\Theta^\varepsilon_{\text{II}} := \Theta_{(1>0>2)}^\varepsilon \cup \Theta_{(2>0>1)}^\varepsilon, 
\qquad 
\Theta^\varepsilon_{\text{III}}: = \Theta_{(0>2>1)}^\varepsilon,$$
and with $\Theta^\varepsilon$ we denote the union of these three sets.

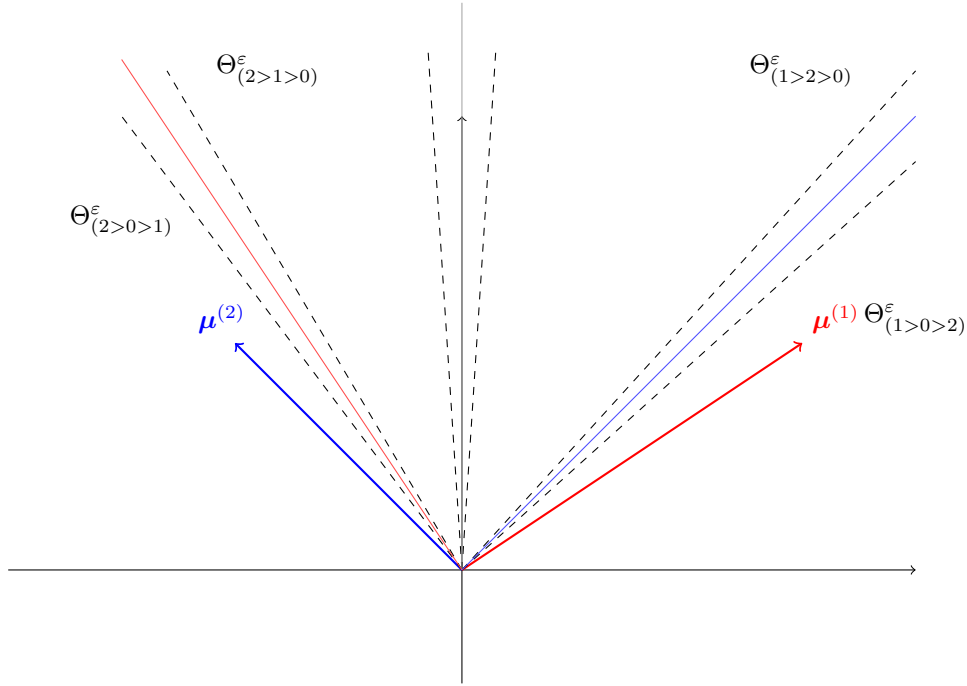
\begin{figure}[h!]
        \centering
        \begin{tikzpicture}[scale=3]
        \draw[->] (-2,0) -- (2,0) node[right] {};
        \draw[->] (0,-0.5) -- (0,2) node[above] {};
        
        \draw[->, red, thick] (0,0) -- (1.5,1) node[above right] {$\boldsymbol{\mu}^{(1)}$};
        
        \draw[->, blue, thick] (0,0) -- (-1,1) node[above, xshift=-5pt] {$\boldsymbol{\mu}^{(2)}$};
    
        \draw[-, dashed] (0,0) -- (2,2.2) node[left, xshift=-20pt] {$\Theta_{(1>2>0)}^\varepsilon$};
        \draw[-, dashed] (0,0) -- (0.15,2.3);
    
        \draw[-, dashed] (0,0) -- (-1.3,2.2) node[right, xshift=15pt] {$\Theta_{(2>1>0)}^\varepsilon$};
        \draw[-, dashed] (0,0) -- (-0.15,2.3);
    
        \draw[-, dashed] (0,0) -- (-1.5,2) node[below, yshift=-30pt] {$\Theta_{(2>0>1)}^\varepsilon$};
    
        \draw[-, dashed] (0,0) -- (2,1.8) node[below, yshift=-50pt] 
        {$\Theta_{(1>0>2)}^\varepsilon$};

        \draw[-, color=gray!70] (0,0) -- (0, 2.5); 
        
        \draw[-, color=blue!70] (0,0) -- (2, 2); 

        \draw[-, color=red!70] (0, 0) -- (-1.5, 2.25); 
        
        \end{tikzpicture}
        \caption{Division of angles from $[0,\pi]$.}
        \label{fig:Theta areas}
    \end{figure}

For $\gamma \in (0, 1/2)$ and $\varepsilon > 0$ define the event $E_{n, i}(\varepsilon, \gamma)$ with the following:
\begin{itemize}
    \item for all $\theta \in \Theta^\varepsilon_\text{I}$, \ $\underline{J}_{n, \underline{\I}_n(\theta)}(\theta)<\gamma n$,  \ $\overline{J}_{n, \overline{\I}_n(\theta)}(\theta)>(1-\gamma) n$, \
     $\underline{J}^{(i)}_{n, \underline{\I}^{(i)}_n(\theta)}(\theta)< \gamma n$, \ and $\overline{J}_{n, \overline{\I}_n^{(i)}(\theta)}^{(i)}(\theta)>(1-\gamma) n$,
    \item for all $\theta \in \Theta^\varepsilon_\text{II}$, \
        $\underline{J}_{n, \underline{\I}_n(\theta)}(\theta)>(1-\gamma) n$, 
          \
          $\overline{J}_{n, \overline{\I}_n(\theta)}(\theta)>(1-\gamma) n$, 
          \
          $\underline{J}^{(i)}_{n, \underline{\I}^{(i)}_n(\theta)}(\theta)>(1-\gamma) n$, \ and $\overline{J}_{n, \overline{\I}_n^{(i)}(\theta)}^{(i)}(\theta)>(1-\gamma) n$,
    
    \item for all $\theta \in \Theta^\varepsilon_\text{III}$, \
        $\underline{J}_{n, \underline{\I}_n(\theta)}(\theta)>(1-\gamma) n$, \ 
        $\overline{J}_{n, \overline{\I}_n(\theta)}(\theta)<\gamma n$, \ $\underline{J}^{(i)}_{n, \underline{\I}^{(i)}_n(\theta)}(\theta)>(1-\gamma) n$, \ 
        and $\overline{J}_{n, \overline{\I}_n^{(i)}(\theta)}^{(i)}(\theta)< \gamma n$, 
    
    \item for all $\theta \in \Theta^\varepsilon_{(1>2>0)}$,  \ $\overline{\I}_n(\theta) = \overline{\I}_n^{(i)}(\theta) = 1$,

    \item for all $\theta \in \Theta^\varepsilon_{(2>1>0)}$,  \ $\overline{\I}_n(\theta) = \overline{\I}_n^{(i)}(\theta) = 2$,

    \item for all $\theta \in \Theta^\varepsilon_{(1>0>2)}$,  \ $\overline{\I}_n(\theta) = \overline{\I}_n^{(i)}(\theta) = 1$, and $\underline{\I}_n(\theta) = \underline{\I}_n^{(i)}(\theta) = 2$,

    \item for all $\theta \in \Theta^\varepsilon_{(2>0>1)}$,  \ $\overline{\I}_n(\theta) = \overline{\I}_n^{(i)}(\theta) = 2$, and $\underline{\I}_n(\theta) = \underline{\I}_n^{(i)}(\theta) = 1$,

    \item for all $\theta \in \Theta^\varepsilon_{(0>2>1)}$,  \ $\underline{\I}_n(\theta) = \underline{\I}_n^{(i)}(\theta) = 1$. 
    
\end{itemize}

The idea behind event $E_{n, i}(\varepsilon, \gamma)$ is that it occurs with very high probability and that we have a good control of $\Delta_n^{(i)}(\theta)$ on that event, namely, for each region, we condition how early or late and on which of the walks the minima and maxima of projections will be achieved. The following proposition establishes our assertion. 

\begin{proposition} \label{proposition - control of delta_n_i - perimeter}
For any $\gamma \in(0,1 / 2)$, and any  $\varepsilon > 0$, the following hold:
\begin{itemize}
    \item[(i)] if $i \in I_{n, \gamma} = \{1, \ldots, n\} \cap[\gamma n,(1-\gamma) n]$, then, a.s., for any $\theta \in \Theta^\varepsilon_{\operatorname{I}}$,
    \begin{align*}
       \Delta_{n}^{(i)}(\theta) \mathbf{1}\left(E_{n, i}(\varepsilon, \gamma)\right)= \left( Z_{i}^{(\overline{\I}_n(\theta))}-\tilde{Z}^{(\overline{\I}_n(\theta))}_{i} \right)  \cdot \mathbf{e}_{\theta} \mathbf{1}\left(E_{n, i}(\varepsilon, \gamma)\right);
    \end{align*}
    for any $\theta \in \Theta^\varepsilon_{\operatorname{II}}$ we have
    \begin{align*}
       \Delta_{n}^{(i)}(\theta) \mathbf{1}\left(E_{n, i}(\varepsilon, \gamma)\right)=\left( \left( Z_{i}^{(\overline{\I}_n(\theta))}-\tilde{Z}^{(\overline{\I}_n(\theta))}_{i} \right) - \left( Z_{i}^{(\underline{\I}_n(\theta))}-\tilde{Z}^{(\underline{\I}_n(\theta))}_{i} \right) \right) \cdot \mathbf{e}_{\theta} \mathbf{1}\left(E_{n, i}(\varepsilon, \gamma)\right);
    \end{align*}
    while for any $\theta \in \Theta^\varepsilon_{\operatorname{III}}$ we have
    \begin{align*}
       \Delta_{n}^{(i)}(\theta) \mathbf{1}\left(E_{n, i}(\varepsilon, \gamma)\right)= - \left( Z_{i}^{(\underline{\I}_n(\theta))}-\tilde{Z}^{(\underline{\I}_n(\theta))}_{i} \right)  \cdot \mathbf{e}_{\theta} \mathbf{1}\left(E_{n, i}(\varepsilon, \gamma)\right). 
    \end{align*}
    \item[(ii)] if $\mathbb{E}\left[\|Z_{1}^{(k)}\|\right]<\infty$ for both $k\in\{1,2\}$, and \eqref{eq:per_assumption} holds, then 
    $$
    \lim_{n \to \infty} \min _{1 \leq i \leq n} \mathbb{P}\left[E_{n, i}(\varepsilon, \gamma)\right] = 1.
    $$
\end{itemize}
\end{proposition}

\begin{proof}
$(i)$ Suppose that $i \in I_{n, \gamma}$, so $\gamma n \leq i \leq(1-\gamma) n$. Also, suppose that $\theta \in \Theta^\varepsilon_{\operatorname{I}}$. On $E_{n, i}(\varepsilon, \gamma)$, we have that $\underline{J}_{n, \underline{\I}_n(\theta)}(\theta)<i<\overline{J}_{n, \overline{\I}_n(\theta)}(\theta)$ and $\underline{J}^{(i)}_{n, \underline{\I}_n^{(i)}(\theta)}(\theta)<i<\overline{J}^{(i)}_{n, \overline{\I}_n^{(i)}(\theta)}(\theta)$. Therefore, from the definition of the resampled processes (\ref{definicija resamplanih šetnji}), we can see that it has to be $\underline{J}_{n, \underline{\I}_n(\theta)}(\theta)=\underline{J}^{(i)}_{n, \underline{\I}_n^{(i)}(\theta)}(\theta)$ and moreover $\underline{\I}_n(\theta) = \underline{\I}_n^{(i)}(\theta)$. Thus, it implies that $\underline{J}_{n, \underline{\I}_n(\theta)}(\theta)=\underline{J}^{(i)}_{n, \underline{\I}_n(\theta)}(\theta)$. Hence $m_n(\theta) = m_n^{(i)}(\theta)$.  Further, on the event $E_{n, i}(\varepsilon, \gamma)$, it holds that $\overline{\mathcal{I}}_n(\theta) = \overline{\mathcal{I}}_n^{(i)}(\theta)$. Thus, we can express 
\begin{align*}
    M_n^{(i)}(\theta) = M_n(\theta) + (\tilde{Z}_i^{(\overline{\I}_n(\theta))}-Z_i^{(\overline{\I}_n(\theta)) })\cdot \e_\theta. 
\end{align*}
Therefore, the first equality of $(i)$ follows.  For the second equality, take an angle $\theta \in \Theta^\varepsilon_{\operatorname{II}}$. On $E_{n, i}(\varepsilon, \gamma)$, we have that $\overline{\I}_n(\theta) = \overline{\I}_n^{(i)}(\theta)$ and $\underline{\I}_n(\theta) = \underline{\I}_n^{(i)}(\theta)$. Hence, similarly as earlier, we obtain that
\begin{align*}
    M_n^{(i)}(\theta) = M_n(\theta) + (\tilde{Z}_i^{(\overline{\I}_n(\theta))}-Z_i^{(\overline{\I}_n(\theta)) })\cdot \e_\theta, 
\end{align*}
and similarly
\begin{align*}
    m_n^{(i)}(\theta) = m_n(\theta) + (\tilde{Z}_i^{(\underline{\I}_n(\theta))}-Z_i^{(\underline{\I}_n(\theta)) })\cdot \e_\theta, 
\end{align*}
so the claim follows. The third equality (for $\theta \in \Theta_{\operatorname{III}}^\varepsilon$) is shown similarly. 

$(ii)$ The idea behind the proof of this claim is to show that the probabilities for all eight items in the definition of $E_{n, i}(\gamma, \varepsilon)$ tend to $1$ as $n \rightarrow \infty$, no matter which $i \in I_{n, \gamma}$ we choose. Let us prove the claim for the first item. The key idea is to simultaneously use the strong law of large numbers (see \cite[Theorem 2.4.1]{durrett2019probability}) for both walks. Take an arbitrary $\varepsilon_1$ such that $0 < \varepsilon_1 < \varepsilon$. There exists a random variable $N := N(\varepsilon_1)$ such that $N$ is finite almost surely and 
\begin{align*}
   n \geq N \implies \left\| \frac{S_n^{(k)}}{n} - \boldsymbol{\mu}^{(k)} \right\| < \varepsilon_1
\end{align*}
for both $k \in\{1, 2\}$. This implies that, for $\theta\in\Theta^\varepsilon_{\operatorname{I}}$,
\begin{align}\label{jzvb}
     n \geq N \implies \left| \frac{S_n^{(k)}}{n} \cdot\e_\theta  - \boldsymbol{\mu}^{(k)}\cdot \e_\theta \right| =  \left| \frac{S_n^{(k)}}{n}\cdot \e_\theta  - \mu^{(k)} \cos(\theta^{(k)} - \theta) \right| \leq \left\| \frac{S_n^{(k)}}{n} - \boldsymbol{\mu}^{(k)} \right\| < \varepsilon_1. 
\end{align}
For $n \geq N$, we have that
\begin{align*}
    S_n^{(k)}\cdot \e_\theta > (\mu^{(k)} \cos(\theta^{(k)} - \theta) - \varepsilon_1) n > (\varepsilon - \varepsilon_1)n. 
\end{align*}
The last term is strictly positive because of the choice of $\varepsilon_1$. Therefore, for any $\theta \in \Theta_{\operatorname{I}}^\varepsilon$, we have that $S_n^{(k)}\cdot\e_\theta > 0$ for both $k \in \{1, 2\}$ and $n \geq N$. But, recall that $S_0^{(k)} \cdot \e_\theta = 0$, so it gives us that 
$\underline{J}_{n, \underline{\I}_n(\theta)}(\theta) < N$ for all $\theta \in \Theta_{\operatorname{I}}^\varepsilon$. Hence,
\begin{align*}
    1 \geq \lim_n \mathbb{P} \left( \bigcap_{\theta \in \Theta_{\operatorname{I}}^\varepsilon} \{ \underline{J}_{n, \underline{\I}_n(\theta)}(\theta) < \gamma n \} \right) \geq \lim_n \mathbb{P} \left( N \leq \gamma n \right) = 1, 
\end{align*}
since $N$ is a.s. finite. Considering the second event, $\overline{J}_{n, \overline{\mathcal{I}}_n(\theta)}(\theta)>(1-\gamma) n$, we have that
\begin{align}\label{bigger}
    \max _{0 \leq j \leq(1-\gamma) n} S_j^{(\overline{\I}_n(\theta))} \cdot \mathbf{e}_\theta \leq \max \left\{\max _{0 \leq j \leq N} S_j^{(\overline{\I}_n(\theta))} \cdot \mathbf{e}_\theta, \max _{N \leq j \leq(1-\gamma) n} S_j^{(\overline{\I}_n(\theta))} \cdot \mathbf{e}_\theta\right\}. 
\end{align}
For the last term, $(\ref{jzvb})$ yields
\begin{align*}
    \max _{N \leq j \leq(1-\gamma) n} S_j^{(\overline{\I}_n(\theta))} \cdot \mathbf{e}_\theta 
    &\leq 
    \max _{0 \leq j \leq(1-\gamma) n}
    \left( 
    \mu^{(\overline{\I}_n(\theta))} \cos(\theta^{(\overline{\I}_n(\theta))} - \theta)+\varepsilon_1\right) j \\ &\leq \left( 
    \mu^{(\overline{\I}_n(\theta))} \cos(\theta^{(\overline{\I}_n(\theta))} - \theta)+\varepsilon_1\right) (1-\gamma) n .
\end{align*}
Once again, if $n \geq N$, the inequality $(\ref{jzvb})$ gives us 
\begin{align*}
     S_n^{(\overline{\I}_n(\theta))} \cdot\e_\theta > \left( 
    \mu^{(\overline{\I}_n(\theta))} \cos(\theta^{(\overline{\I}_n(\theta))} - \theta)- \varepsilon_1\right)  n . 
\end{align*}
The inequality 
\begin{align*}
    \left( 
    \mu^{(\overline{\I}_n(\theta))} \cos(\theta^{(\overline{\I}_n(\theta))} - \theta)-\varepsilon_1\right) 
    \geq
    \left( 
    \mu^{(\overline{\I}_n(\theta))} \cos(\theta^{(\overline{\I}_n(\theta))} - \theta)+\varepsilon_1\right)  (1-\gamma)
\end{align*}
holds a.s.\ if, and only if, 
\begin{align*}
    \varepsilon_1 \leq \frac{\gamma \mu^{(\overline{\I}_n(\theta))} \cos(\theta^{(\overline{\I}_n(\theta))} - \theta)}{2 - \gamma}
\end{align*}
holds a.s. Therefore, we can additionally require that $\varepsilon_1 > 0$ has been taken such that
\begin{align*}
    \varepsilon_1 < \frac{\gamma \varepsilon}{2}
\end{align*}
for the preceding inequality to hold. In that case, for any $\theta \in \Theta_{\operatorname{I}}^\varepsilon$, we have that
\begin{align*}
    S_n^{(\overline{\I}_n(\theta))} \cdot \mathbf{e}_\theta>\max _{N \leq j \leq(1-\gamma) n} S_j^{(\overline{\I}_n(\theta))} \cdot \mathbf{e}_\theta \quad \textnormal{a.s.}
\end{align*}
Hence, by $(\ref{bigger})$, we have that
\begin{align*}
    \mathbb{P}\left[\bigcap_{\theta \in \Theta_{\operatorname{I}}^\varepsilon}\left\{\overline{J}_{n, \overline{\I}_n(\theta)}(\theta)>(1-\gamma) n\right\}\right] &\geq \mathbb{P}\left[\bigcap_{\theta \in \Theta_{\operatorname{I}}^\varepsilon}\left\{S_n^{(\overline{\I}_n(\theta))} \cdot \mathbf{e}_\theta>\max _{0 \leq j \leq(1-\gamma) n} S_j^{(\overline{\I}_n(\theta))} \cdot \mathbf{e}_\theta\right\}\right] \\
    &\geq \mathbb{P}\left[ N \leq n, \bigcap_{\theta \in \Theta_{\operatorname{I}}^\varepsilon}\left\{S_n^{(\overline{\I}_n(\theta))} \cdot \mathbf{e}_\theta>\max _{0 \leq j \leq N} S_j^{(\overline{\I}_n(\theta))} \cdot \mathbf{e}_\theta\right\}\right]. 
\end{align*}
Additionally, for $n \geq N$, we have that 
\begin{align*}
    S_n^{(\overline{\I}_n(\theta))} \e_\theta 
    > (\mu^{(\overline{\I}_n(\theta))} \cos(\theta^{(\overline{\I}_n(\theta))} - \theta) - \varepsilon_1) n 
    > (\varepsilon - \frac{\gamma \varepsilon}{2})n = \varepsilon n \left( 1 - \frac{\gamma}{2} \right), 
\end{align*}
so we get
\begin{align*}
    \mathbb{P}\left[\bigcap_{\theta \in                     \Theta_{\operatorname{I}}^\varepsilon}\left\{\overline{J}_{n, \overline{\I}_n(\theta)}(\theta)>(1-\gamma) n\right\}\right] 
    \geq 
    \mathbb{P}\left[ N \leq n, \max _{\substack{0 \leq j \leq N \\ k = 1, 2}} \| S_j^{(k)} \|  \leq \varepsilon n \left( 1 - \frac{\gamma}{2} \right) \right]. 
\end{align*}
But, since $N$ is finite a.s., we have that $\mathbb{P}(N >n) \rightarrow 0$ and
\begin{align*}
    \lim_{n \rightarrow \infty} \mathbb{P}\left[ \max _{\substack{0 \leq j \leq N \\ k = 1, 2}} \| S_j^{(k)} \|  > \varepsilon n \left( 1 - \frac{\gamma}{2} \right) \right] = 0.  
\end{align*}
Therefore, we have that 
\begin{align*}
    \lim_{n \rightarrow \infty} \mathbb{P}\left[ N \leq n, \max _{\substack{0 \leq j \leq N \\ k = 1, 2}} \| S_j^{(k)} \|  \leq \varepsilon n \left( 1 - \frac{\gamma}{2} \right) \right] = 1, 
\end{align*}
so it gives us
\begin{align*}
    \lim_{n \rightarrow \infty} \mathbb{P}\left[\bigcap_{\theta \in \Theta_{\operatorname{I}}^\varepsilon}\left\{\overline{J}_{n, \overline{\I}_n(\theta)}(\theta)>(1-\gamma) n\right\}\right] = 1. 
\end{align*}
This shows the asymptotic probability of the first statement in the first item point of the definition of the set $E_{n,i}(\varepsilon, \gamma)$. Note that the statement of the first item corresponding to the resampled walks can be shown in the same way, given that resampling preserves the underlying distribution. The proofs for the second and third item points are omitted, as they proceed in a completely analogous way as the first item point.

We now proceed with the the fourth item. We focus on the angles belonging to $\Theta_{(1>2>0)}^\varepsilon$. It should be noted that the reasoning deployed here can be easily adapted to other cases. We aim to establish that
\begin{align*}
    \lim_{n \rightarrow \infty} \mathbb{P} \left( \bigcap_{\theta \in \Theta_{(1>2>0)}^\varepsilon} \{ \overline{\I}_n(\theta) = \overline{\I}_n^{(i)}(\theta) = 1 \} \right)  = 1. 
\end{align*}
Since $\overline{\I}_n^{(i)}(\theta)$ is identically distributed as $\overline{\I}_n(\theta)$, it is sufficient to prove that 
$$
\lim_{n \to \infty} \mathbb{P} \left( \bigcap_{\theta \in\Theta_{(1>2>0)}^\varepsilon} \{ \overline{\I}_n(\theta) = 1 \} \right) = 1.
$$ 
Choose $\varepsilon_1 > 0$ such that $2 \varepsilon_1 < \boldsymbol{\mu}^{(1)} \cdot\e_\theta - \boldsymbol{\mu}^{(2)} \cdot\e_\theta $, and $\boldsymbol{\mu}^{(2)} \cdot\e_\theta  - \varepsilon_1 > 0$ for all $\theta \in \Theta_{(1>2>0)}^{\varepsilon}$, which is possible because of the definition of $\Theta_{(1>2>0)}^{\varepsilon}$. For this collection of one-dimensional walks, we have that 
\begin{align*}
    n \geq N \implies \left| \frac{S_n^{(k)} }{n} \cdot\e_\theta- \boldsymbol{\mu}^{(k)} \cdot\e_\theta  \right| \leq  \left\| \frac{S_n^{(k)}}{n} - \boldsymbol{\mu}^{(k)}\right\| < \varepsilon_1
\end{align*}
for $k \in \{1, 2\}$. Hence, we have the following 
\begin{align*}
    \mathbb{P} \left( \bigcap_{\theta \in \Theta_{(1>2>0)}^{\varepsilon}} \{ \overline{\I}_n(\theta) = 1 \} \right) 
    &\geq 
    \mathbb{P} \left(N \leq n, \ \bigcap_{\theta \in \Theta_{(1>2>0)}^{\varepsilon}} \{ \overline{\I}_n(\theta) = 1 \} \right) \\
    &\geq 
    \mathbb{P}\left(
        N \leq n, \ \bigcap_{\theta \in \Theta_{(1>2>0)}^{\varepsilon}} \max_{0 \leq j \leq n}S_j^{(1)} \cdot \e_\theta \geq 
        \max_{0 \leq j \leq n}S_j^{(2)} \cdot\e_\theta 
    \right) \\
    &= 
    \mathbb{P}\left(
        N \leq n, \ \bigcap_{\theta \in \Theta_{(1>2>0)}^{\varepsilon}} \max_{0 \leq j \leq n}\frac{S_j^{(1)} \cdot\e_\theta }{n} 
        \geq 
        \max_{0 \leq j \leq n} \frac{S_j^{(2)} \cdot\e_\theta }{n}  
    \right) \\
    &= 
    \mathbb{P}\left(
        N \leq n, \  \bigcap_{\theta \in \Theta_{(1>2>0)}^{\varepsilon}} \max \left\{ \max_{0 \leq j \leq N}\frac{S_j^{(1)} \cdot\e_\theta}{n}, \max_{N < j \leq n}\frac{S_j^{(1)} \cdot\e_\theta}{n} \right\} \right. 
        \geq 
        \\
        & \quad \quad  \quad \left. 
        \max \left\{ \max_{0 \leq j \leq N}\frac{S_j^{(2)} \cdot\e_\theta}{n}, \max_{N < j \leq n}\frac{S_j^{(2)} \cdot\e_\theta}{n} \right\} 
    \right) \\
    &\geq 
    \mathbb{P}\left(
        N \leq n, \  \bigcap_{\theta \in \Theta_{(1>2>0)}^{\varepsilon}} \max \left\{ \max_{0 \leq j \leq N}\frac{S_j^{(1)} \cdot\e_\theta}{n}, \max_{N < j \leq n} \frac{j}{n}(\boldsymbol{\mu}^{(1)} \cdot\e_\theta- \varepsilon_1) \right\} \right. 
        \geq \\
        & \quad \quad \quad \left. 
         \max \left\{ \max_{0 \leq j \leq N}\frac{S_j^{(2)} \cdot\e_\theta}{n}, \max_{N < j \leq n} \frac{j}{n}(\boldsymbol{\mu}^{(2)} \cdot\e_\theta + \varepsilon_1) \right\} \right), \\
\end{align*}
where our selection of $\varepsilon_1$ justifies the last inequality. By the dominated convergence theorem, it becomes clear that the final term converges to $1$ as $n\to\infty$. Thus, we have proved the asymptotic probability for the fourth item point in the definition of $E_{n, i}(\varepsilon, \gamma)$. The same idea is applied to prove the statements for the remaining four items. Combining all the results, we get $(ii)$. 
\end{proof}

\begin{remark}
    Notice that in the case when the drift vectors are co-linear, see Figure \ref{fig:colinear_drifts}, the situation is slightly different then the one shown in Figure \ref{fig:position of the drift vectors}. The proof of Proposition \ref{proposition - control of delta_n_i - perimeter} remains the same, the only difference being that some of the subsets of angles $\theta \in [0, \pi]$ that were introduced in \eqref{eq:def_of_Theta_sets} are empty. It the case when both drift vectors have the same direction, the sets $\Theta_{(1 > 0 > 2)}^{\varepsilon}$ and $\Theta_{(2 > 0 > 1)}^{\varepsilon}$ are empty, while in the case when drift vectors have opposite directions the sets $\Theta_{(1 > 2 > 0)}^{\varepsilon}$, $\Theta_{(2 > 1 > 0)}^{\varepsilon}$ and $\Theta_{(0 > 2 > 1)}^{\varepsilon}$ are empty. The case when both vectors have the same magnitude and the same orientation is excluded by the assumption \eqref{eq:per_assumption}. This case is discussed further in Section \ref{sec:simulations}. Somewhat surprisingly, simulation results suggest that we lose the normality of the distributional limit in this case. We offer a possible explanation for this phenomena, but the formal proof remained out of our reach. The efforts to extend our results in this direction are currently underway.

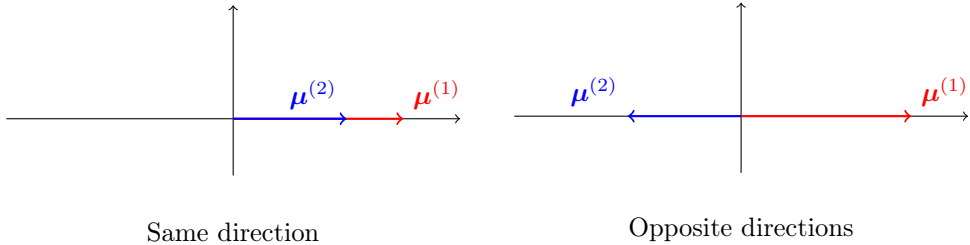
\begin{figure}[h]
    \centering
    \begin{tikzpicture}[scale=1.5]
    \draw[->] (-2,0) -- (2,0) node[right] {};
    \draw[->] (0,-0.5) -- (0,1) node[above] {};
    
    \draw[->, red, thick] (0,0) -- (1.5,0) node[above right] {$\boldsymbol{\mu}^{(1)}$};
    
    \draw[->, blue, thick] (0,0) -- (1,0) node[above left] {$\boldsymbol{\mu}^{(2)}$};

    \node at (0,-1) {Same direction};
\end{tikzpicture}
\quad  
\begin{tikzpicture}[scale=1.5]
    \draw[->] (-2,0) -- (2,0) node[right] {};
    \draw[->] (0,-0.5) -- (0,1) node[above] {};
    
    \draw[->, red, thick] (0,0) -- (1.5,0) node[above right] {$\boldsymbol{\mu}^{(1)}$};
    
    \draw[->, blue, thick] (0,0) -- (-1,0) node[above left] {$\boldsymbol{\mu}^{(2)}$};

    \node at (0,-1) {Opposite directions};
\end{tikzpicture}
    \caption{Allowed positions of the co-linear drift vectors.}
    \label{fig:colinear_drifts}
\end{figure}

\end{remark}

\subsection{Approximation lemma for the perimeter}

In the following lemma, we compare $\mathcal{L}_{n, i}$ with appropriately centered and projected $i$-th steps of the walks. The proof of the lemma depends on our earlier assumptions regarding the spatial orientation of the drift vectors relative to the $y$-axis. Irrespective of the scenario chosen, the outcome remains the same - the only distinction lies in the specific angular regions under consideration and the sequencing of integrals. For clarity, we will explain only the case where the $y$-axis is situated between the drift vectors.

\begin{lemma} \label{approximation lemma}
    Assume \eqref{eq:per_assumption} and $E[\|Z_1^{(k)}\|] < \infty$ for both $k \in \{1,2\}$. Then, for any $\gamma \in (0,1/2)$, $\varepsilon > 0$, and $i \in I_{n, \gamma}$, we have
    \begin{align*}
        &\left| \mathcal{L}_{n, i}- \left( \left(Z_i^{(1)}-\boldsymbol{\mu}^{(1)}\right) \cdot (\e_{\theta^{(0)}}^\perp + \e_{\theta^{(1)}})
        + \left(Z_i^{(2)}-\boldsymbol{\mu}^{(2)}\right) \cdot (\e_{\theta^{(2)}} - \e_{\theta^{(0)}}^\perp)\right) \right| 
        \\
        &\leq
        3 \pi ( \|Z_i^{(1)}\| + \|Z_i^{(2)}\|)\mathbb{P}\left[E_{n, i}^c(\varepsilon, \gamma) \mid \mathcal{F}_i\right]
         + 3 \pi \mathbb{E}\left[ (  \|\tilde{Z}_i^{(1)}\| + \|\tilde{Z}_i^{(2)}\|)  \mathbf{1}\left(E_{n, i}^c(\varepsilon, \gamma)\right) \mid \mathcal{F}_i\right] 
         \\
         &\ \ \ \ + 3 \lambda((\Theta^\varepsilon)^c) \left( \|Z_i^{(1)}\| + \|Z_i^{(2)}\| + \E \|Z_1^{(1)}\| + \E \|Z_1^{(2)}\| \right), 
    \end{align*}
where $\lambda$ is Lebesgue measure on $[0, \pi]$ and $(\Theta^\varepsilon)^c$ is the complement of the set ${\Theta}^\varepsilon$ in $[0, \pi]$. 
\end{lemma}

\begin{proof} 

Start from $L_n - L_n^{(i)}$ and take the conditional expectation with respect to $\mathcal{F}_i$. We get
\begin{align}\label{al1}
    \mathcal{L}_{n, i}=\int_0^\pi \mathbb{E}\left[\Delta_n^{(i)}(\theta) \mathbf{1}\left(E_{n, i}(\varepsilon, \gamma)\right) \mid \mathcal{F}_i\right] \mathrm{d} \theta
    +
    \int_0^\pi \mathbb{E}\left[\Delta_n^{(i)}(\theta) \mathbf{1}\left(E_{n, i}^c(\varepsilon, \gamma)\right) \mid \mathcal{F}_i\right] \mathrm{d} \theta. 
\end{align}
For the second term in $(\ref{al1})$, we have
\begin{align}\label{al2}
    \left|\int_0^\pi \mathbb{E}\left[\Delta_n^{(i)}(\theta) \mathbf{1}\left(E_{n, i}^{\mathrm{c}}(\varepsilon, \gamma)\right) \mid \mathcal{F}_i\right] \mathrm{d} \theta\right| \leq \int_0^\pi \mathbb{E}\left[\left|\Delta_n^{(i)}(\theta)\right| \mathbf{1}\left(E_{n, i}^{\mathrm{c}}(\varepsilon, \gamma)\right) \mid \mathcal{F}_i\right] \mathrm{d} \theta, 
\end{align}
and apply the upper bound obtained in Lemma $\ref{lemma - omedjenost integrabilnom varijablom}$ to get
\begin{small}
\begin{align*}
    &\int_0^\pi \mathbb{E}\left[\left|\Delta_n^{(i)}(\theta)\right| \mathbf{1}\left(E_{n, i}^c(\varepsilon, \gamma)\right) \mid \mathcal{F}_i\right] \mathrm{d} \theta\\
    &
    \leq 
    2 \pi \mathbb{E}\left[\left(\left\|Z_i^{(1)}\right\|+\left\|\tilde{Z}_i^{(1)}\right\| + \left\|Z_i^{(2)}\right\|+\left\|\tilde{Z}_i^{(2)}\right\|\right) \mathbf{1}\left(E_{n, i}^c(\varepsilon, \gamma)\right) \mid \mathcal{F}_i\right] \\
    &= 2 \pi ( \|Z_i^{(1)}\| + \|Z_i^{(2)}\|)\mathbb{P}\left[E_{n, i}^c(\varepsilon, \gamma) \mid \mathcal{F}_i\right]+2 \pi \mathbb{E}\left[ (  \|\tilde{Z}_i^{(1)}\| + \|\tilde{Z}_i^{(2)}\|)  \mathbf{1}\left(E_{n, i}^c(\varepsilon, \gamma)\right) \mid \mathcal{F}_i\right],
\end{align*}
\end{small}
where we used that $Z_i^{(k)}$ are $\mathcal{F}_i$-measurable with $\E \|Z_i^{(k)}\| < \infty$. Thus, it suffices to show
\begin{align*}
        &\left| \int_0^\pi \mathbb{E}\left[\Delta_n^{(i)}(\theta) \mathbf{1}\left(E_{n, i}(\varepsilon, \gamma)\right) \mid \mathcal{F}_i\right] \mathrm{d} \theta - \left( \left(Z_i^{(1)}-\boldsymbol{\mu}^{(1)}\right) \cdot (\e_{\theta^{(0)}}^\perp + \e_{\theta^{(1)}})
        + \left(Z_i^{(2)}-\boldsymbol{\mu}^{(2)}\right) \cdot (\e_{\theta^{(2)}} - \e_{\theta^{(0)}}^\perp)\right) \right| 
        \\
        &\leq
        \pi ( \|Z_i^{(1)}\| + \|Z_i^{(2)}\|)\mathbb{P}\left[E_{n, i}^c(\varepsilon, \gamma) \mid \mathcal{F}_i\right]
         + \pi \mathbb{E}\left[ (  \|\tilde{Z}_i^{(1)}\| + \|\tilde{Z}_i^{(2)}\|)  \mathbf{1}\left(E_{n, i}^c(\varepsilon, \gamma)\right) \mid \mathcal{F}_i\right] 
         \\
         &\ \ \ \ + 3 \lambda((\Theta^\varepsilon)^c) \left( \|Z_i^{(1)}\| + \|Z_i^{(2)}\| + \E \|Z_1^{(1)}\| + \E \|Z_1^{(2)}\| \right), 
    \end{align*}
Let us decompose the first integral in $(\ref{al1})$. It can be written as the following sum
\begin{equation}\label{eq:integral_breakdown}
\begin{aligned}
    \int_{[0, \pi]} \mathbb{E}\left[\Delta_n^{(i)}(\theta) \mathbf{1}\left(E_{n, i}(\varepsilon, \gamma)\right) \mid \mathcal{F}_i\right]  \mathrm{d} \theta 
    =& 
    \int_{[0, \theta^{(2)}-\pi/2]} \mathbb{E}\left[\Delta_n^{(i)}(\theta) \mathbf{1}\left(E_{n, i}(\varepsilon, \gamma)\right) \mid \mathcal{F}_i\right] \mathrm{d} \theta 
    \\ 
    &+ 
    \int_{[\theta^{(2)} - \pi/2, \pi/2]} \mathbb{E} \left[\Delta_n^{(i)}(\theta) \mathbf{1}\left(E_{n, i}(\varepsilon, \gamma)\right) \mid \mathcal{F}_i\right] \mathrm{d} \theta \\
    &+
    \int_{[\pi/2, \theta^{(1)} + \pi/2]} \mathbb{E} \left[\Delta_n^{(i)}(\theta) \mathbf{1}\left(E_{n, i}(\varepsilon, \gamma)\right) \mid \mathcal{F}_i\right] \mathrm{d} \theta \\
    &+ 
    \int_{[\theta^{(1)} + \pi/2, \pi]} \mathbb{E} \left[\Delta_n^{(i)}(\theta) \mathbf{1}\left(E_{n, i}(\varepsilon, \gamma)\right) \mid \mathcal{F}_i\right] \mathrm{d} \theta. 
\end{aligned}
\end{equation}
Denote these integrals with $I_1, I_2, I_3$ and $I_4$, respectively. Let us calculate the first integral. For $I_1$, we have that
\begin{equation}\label{eq:I_11}
\begin{aligned}
    \int_{[0, \theta^{(2)} - \pi/2]} \mathbb{E} \left[\Delta_n^{(i)}(\theta) \mathbf{1}\left(E_{n, i}(\varepsilon, \gamma)\right) \mid \mathcal{F}_i\right] \mathrm{d} \theta =& \int_{\Theta^\varepsilon_{(1>0>2)}} \mathbb{E} \left[\Delta_n^{(i)}(\theta) \mathbf{1}\left(E_{n, i}(\varepsilon, \gamma)\right) \mid \mathcal{F}_i\right] \mathrm{d} \theta 
    \\
    &+ \int_{(\Theta^\varepsilon_{(1>0>2)})^c} \mathbb{E} \left[\Delta_n^{(i)}(\theta) \mathbf{1}\left(E_{n, i}(\varepsilon, \gamma)\right) \mid \mathcal{F}_i\right] \mathrm{d} \theta , 
\end{aligned}
\end{equation}
where the complement $(\Theta^\varepsilon_{(1>0>2)})^c$ of $\Theta^\varepsilon_{(1>0>2)}$  has been taken with respect to $[0, \theta^{(2)} - \pi/2]$. The second integral is bounded with
\begin{align*}
    \left| \int_{(\Theta^\varepsilon_{(1>0>2)})^c} \mathbb{E} \left[\Delta_n^{(i)}(\theta) \mathbf{1}\left(E_{n, i}(\varepsilon, \gamma)\right) \mid \mathcal{F}_i\right] \mathrm{d} \theta \right|
    \leq
    2 \lambda ((\Theta^\varepsilon_{(1>0>2)})^c) \left( \|Z_i^{(1)}\| + \E \|Z_i^{(1)}\| +  \|Z_i^{(2)}\| + \E \|Z_i^{(2)}\| \right) 
\end{align*}
by Lemma \ref{lemma - omedjenost integrabilnom varijablom}. An analog bound is obtained for $I_2$, $I_3$, and $I_4$, by replacing $\Theta^{\varepsilon}_{(1 > 0 > 2)}$ with $\Theta^{\varepsilon}_{(1 > 2 > 0)}$, $\Theta^{\varepsilon}_{(2 > 1 > 0)}$, and $\Theta^{\varepsilon}_{(2 > 0 > 1)}$ respectively, and by taking complements with respect to $[\theta^{(2)} - \pi/2, \pi/2]$, $[\pi/2, \theta^{(1)} + \pi/2]$, and $[\theta^{(1)} + \pi/2, \pi]$ respectively. Combining these bounds, our problem reduces to showing that
\begin{align*}
        &\left| \int_{\Theta^{\varepsilon}} \mathbb{E}\left[\Delta_n^{(i)}(\theta) \mathbf{1}\left(E_{n, i}(\varepsilon, \gamma)\right) \mid \mathcal{F}_i\right] \mathrm{d} \theta - \left( \left(Z_i^{(1)}-\boldsymbol{\mu}^{(1)}\right) \cdot (\e_{\theta^{(0)}}^\perp + \e_{\theta^{(1)}})
        + \left(Z_i^{(2)}-\boldsymbol{\mu}^{(2)}\right) \cdot (\e_{\theta^{(2)}} - \e_{\theta^{(0)}}^\perp)\right) \right| 
        \\
        &\leq
        \pi ( \|Z_i^{(1)}\| + \|Z_i^{(2)}\|)\mathbb{P}\left[E_{n, i}^c(\varepsilon, \gamma) \mid \mathcal{F}_i\right]
         + \pi \mathbb{E}\left[ (  \|\tilde{Z}_i^{(1)}\| + \|\tilde{Z}_i^{(2)}\|)  \mathbf{1}\left(E_{n, i}^c(\varepsilon, \gamma)\right) \mid \mathcal{F}_i\right] 
         \\
         &\ \ \ \ +\lambda((\Theta^\varepsilon)^c) \left( \|Z_i^{(1)}\| + \|Z_i^{(2)}\| + \E \|Z_1^{(1)}\| + \E \|Z_1^{(2)}\| \right). 
    \end{align*}
By Proposition $\ref{proposition - control of delta_n_i - perimeter}$, we have that the first integral in \eqref{eq:I_11} can be rewritten as
\begin{align*}
    &\int_{\Theta^\varepsilon_{(1>0>2)}} \mathbb{E} \left[\Delta_n^{(i)}(\theta) \mathbf{1}\left(E_{n, i}(\varepsilon, \gamma)\right) \mid \mathcal{F}_i\right] \mathrm{d} \theta  
    \\
    &= \int_{\Theta^\varepsilon_{(1>0>2)}} \mathbb{E} \left[ \left( \left( Z_{i}^{(\overline{\I}_n(\theta))}-\tilde{Z}^{(\overline{\I}_n(\theta))}_{i} \right) - \left( Z_{i}^{(\underline{\I}_n(\theta))}-\tilde{Z}^{(\underline{\I}_n(\theta))}_{i} \right) \right) \cdot \mathbf{e}_{\theta}\mathbf{1}\left(E_{n, i}(\varepsilon, \gamma)\right) \mid \mathcal{F}_i\right] \mathrm{d} \theta, 
\end{align*}
while on $E_{n, i}(\varepsilon, \gamma)$ we have that $\overline{\I}_n(\theta) = 1$ and $\underline{\I}_n(\theta) = 2$ for these choices of angles. Thus, the previous integral is equal to
\begin{align*}
    &\int_{\Theta^\varepsilon_{(1>0>2)}} \mathbb{E} \left[ \left( \left( Z_{i}^{(\overline{\I}_n(\theta))}-\tilde{Z}^{(\overline{\I}_n(\theta))}_{i} \right) - \left( Z_{i}^{(\underline{\I}_n(\theta))}-\tilde{Z}^{(\underline{\I}_n(\theta))}_{i} \right) \right) \cdot \mathbf{e}_{\theta} \mathbf{1}\left(E_{n, i}(\varepsilon, \gamma)\right) \mid \mathcal{F}_i\right]  \mathrm{d} \theta 
    \\
    &= \int_{\Theta^\varepsilon_{(1>0>2)}} \mathbb{E}  \left[ \left( \left( Z_{i}^{(1)}-\tilde{Z}^{(1)}_{i} \right) - \left( Z_{i}^{(2)}-\tilde{Z}^{(2)}_{i} \right) \right) \cdot \mathbf{e}_{\theta} \mathbf{1}\left(E_{n, i}(\varepsilon, \gamma)\right) \mid \mathcal{F}_i\right] \mathrm{d} \theta \text{.}
\end{align*}
Now, we can rewrite it as follows
\begin{equation*}
\begin{aligned}
    &\int_{\Theta^\varepsilon_{(1>0>2)}} \mathbb{E}  \left[ \left( \left( Z_{i}^{(1)}-\tilde{Z}^{(1)}_{i} \right) - \left( Z_{i}^{(2)}-\tilde{Z}^{(2)}_{i} \right) \right) \cdot\e_\theta \mathbf{1}\left(E_{n, i}(\varepsilon, \gamma)\right) \mid \mathcal{F}_i\right] \mathrm{d} \theta 
    \\
    & = \int_{\Theta^\varepsilon_{(1>0>2)}} \mathbb{E}  \left[ \left( \left( Z_{i}^{(1)}-\tilde{Z}^{(1)}_{i} \right) - \left( Z_{i}^{(2)}-\tilde{Z}^{(2)}_{i} \right) \right) \cdot\e_\theta \mid \mathcal{F}_i\right] \mathrm{d} \theta \\
    & \ \ \ \ -
    \int_{\Theta^\varepsilon_{(1>0>2)}} \mathbb{E} \left[ \left( \left( Z_{i}^{(1)}-\tilde{Z}^{(1)}_{i} \right) - \left( Z_{i}^{(2)}-\tilde{Z}^{(2)}_{i} \right) \right) \cdot\e_\theta \mathbf{1}\left(E_{n, i}^c(\varepsilon, \gamma)\right) \mid \mathcal{F}_i\right] \mathrm{d} \theta. 
\end{aligned}
\end{equation*}
The second integral is negligible since
\begin{align*}
    &\left| \int_{\Theta^\varepsilon_{(1>0>2)}} \mathbb{E} \left[ \left( \left( Z_{i}^{(1)}-\tilde{Z}^{(1)}_{i} \right) - \left( Z_{i}^{(2)}-\tilde{Z}^{(2)}_{i} \right) \right) \cdot\e_\theta \mathbf{1}\left(E_{n, i}^c(\varepsilon, \gamma)\right) \mid \mathcal{F}_i\right] \mathrm{d} \theta \right| 
    \\
    &\leq 
    \lambda(\Theta^\varepsilon_{(1>0>2)}) \left( ( \|Z_i^{(1)}\| + \|Z_i^{(2)}\|) \mathbb{P}[E_{n, i}^c(\varepsilon, \gamma) \mid \mathcal{F}_i]
    + 
    \E[ ( \|\tilde{Z}_i^{(1)}\| + \|\tilde{Z}_i^{(2)}\|) \mathbf{1}\left(E_{n, i}^c(\varepsilon, \gamma)\right) \mid \mathcal{F}_i ]
    \right),
\end{align*}
while the first one appears as
\begin{align*}
    &\int_{\Theta^\varepsilon_{(1>0>2)}} \mathbb{E}  \left[ \left( \left( Z_{i}^{(1)}-\tilde{Z}^{(1)}_{i} \right) - \left( Z_{i}^{(2)}-\tilde{Z}^{(2)}_{i} \right) \right) \cdot\e_\theta \mid \mathcal{F}_i\right] \mathrm{d} \theta \\
    &= \int_{[0, \theta^{(2)} - \pi/2]} \mathbb{E}  \left[ \left( \left( Z_{i}^{(1)}-\tilde{Z}^{(1)}_{i} \right) - \left( Z_{i}^{(2)}-\tilde{Z}^{(2)}_{i} \right) \right) \cdot\e_\theta \mid \mathcal{F}_i\right] \mathrm{d} \theta \\
    &\ \ \ \ - \int_{(\Theta^\varepsilon_{(1>0>2)})^c} \mathbb{E}  \left[ \left( \left( Z_{i}^{(1)}-\tilde{Z}^{(1)}_{i} \right) - \left( Z_{i}^{(2)}-\tilde{Z}^{(2)}_{i} \right) \right) \cdot\e_\theta  \mid \mathcal{F}_i\right] \mathrm{d} \theta. 
\end{align*}
The contribution of the latter integral is small since
\begin{align*}
    &\left| \int_{(\Theta^\varepsilon_{(1>0>2)})^c} \mathbb{E}  \left[\left( \left( Z_{i}^{(1)}-\tilde{Z}^{(1)}_{i} \right) - \left( Z_{i}^{(2)}-\tilde{Z}^{(2)}_{i} \right) \right) \cdot\e_\theta \mid \mathcal{F}_i\right] \mathrm{d} \theta \right| 
    \\
    &\leq 
    \lambda ((\Theta^\varepsilon_{(1>0>2)})^c) \left( \|Z_i^{(1)}\| + \|Z_i^{(2)}\| + \E \|Z_1^{(1)}\| + \E \|Z_1^{(2)}\| \right).  
\end{align*}
Similarly as before, we conclude analogous bounds for $I_2$, $I_3$ and $I_4$, again by replacing $\Theta^{\varepsilon}_{(1 > 0 > 2)}$ with $\Theta^{\varepsilon}_{(1 > 2 > 0)}$, $\Theta^{\varepsilon}_{(2 > 1 > 0)}$, and $\Theta^{\varepsilon}_{(2 > 0 > 1)}$ respectively, and $[0, \theta^{(2)} - \pi/2]$ with $[\theta^{(2)} - \pi/2, \pi/2]$, $[\pi/2, \theta^{(1)} + \pi/2]$, and $[\theta^{(1)} + \pi/2, \pi]$ respectively. Putting all this together, and using Proposition \ref{proposition - control of delta_n_i - perimeter}, it remains to show that
\begin{equation}\label{eqref:key_equality}
\begin{aligned}
    & \int_{[0, \theta^{(2)} - \pi/2]} \mathbb{E}  \left[ \left( \left( Z_{i}^{(1)}-\tilde{Z}^{(1)}_{i} \right) - \left( Z_{i}^{(2)}-\tilde{Z}^{(2)}_{i} \right) \right) \cdot\e_\theta \mid \mathcal{F}_i\right] \mathrm{d} \theta \\
    & + \int_{[\theta^{(2)} - \pi/2, \pi/2]} \mathbb{E}  \left[ \left( \left( Z_{i}^{(1)}-\tilde{Z}^{(1)}_{i} \right) \right) \cdot\e_\theta \mid \mathcal{F}_i\right] \mathrm{d} \theta \\
    & + \int_{[\pi/2, \theta^{(1)} + \pi/2]} \mathbb{E}  \left[ \left( \left( Z_{i}^{(2)}-\tilde{Z}^{(2)}_{i} \right) \right) \cdot\e_\theta \mid \mathcal{F}_i\right] \mathrm{d} \theta \\
    & + \int_{[\theta^{(1)} + \pi/2,\pi]} \mathbb{E}  \left[ \left( \left( Z_{i}^{(2)}-\tilde{Z}^{(2)}_{i} \right) - \left( Z_{i}^{(1)}-\tilde{Z}^{(1)}_{i} \right) \right) \cdot\e_\theta \mid \mathcal{F}_i\right] \mathrm{d} \theta \\
    &  = \left(Z_i^{(1)}-\boldsymbol{\mu}^{(1)}\right) \cdot (\e_{\theta^{(0)}}^\perp + \e_{\theta^{(1)}})
        + \left(Z_i^{(2)}-\boldsymbol{\mu}^{(2)}\right) \cdot (\e_{\theta^{(2)}} - \e_{\theta^{(0)}}^\perp).
\end{aligned}
\end{equation}
Recall that $\theta^{(0)}$ is $\pi/2$. To calculate the former integrals, use the following notation
\begin{align*}
    \mathbb{E}  \left[ (Z_i^{(k)} - \tilde{Z}_i^{(k)}) \cdot\e_\theta \mid \mathcal{F}_i\right] = (Z_i^{(k)} - \E [Z_i^{(k)}]) \cdot\e_\theta = R_i^{(k)}  \e_{\varphi_i^{(k)}} \cdot \e_\theta, 
\end{align*}
where $R_i^{(k)}$ is the module, and $\e_{\varphi_i^{(k)}}$ is the unit vector of the centered $i$-th step of the $k$-th random walk. Therefore, we get
\begin{align*}
   & \int_{[0, \theta^{(2)} - \pi/2]}   \mathbb{E}  \left[ \left( \left( Z_{i}^{(1)}-\tilde{Z}^{(1)}_{i} \right) - \left( Z_{i}^{(2)}-\tilde{Z}^{(2)}_{i} \right) \right) \e_\theta \mid \mathcal{F}_i\right] \mathrm{d} \theta \\
    & = \int_{[0, \theta^{(2)} - \pi/2]}   \left( R_i^{(1)}  \e_{\varphi_i^{(1)}}  \cdot \e_\theta  -  R_i^{(2)}  \e_{\varphi_i^{(2)}}  \cdot \e_\theta \right) \,  \mathrm{d} \theta \\
    & = R_i^{(1)} \left[ - \cos(\theta^{(2)} - \varphi_i^{(1)}) + \sin(\varphi_i^{(1)}) \right]
     - R_i^{(2)} \left[ - \cos(\theta^{(2)} - \varphi_i^{(2)}) + \sin(\varphi_i^{(2)}) \right].
\end{align*}

The remaining integrals can be expressed in the same way, and they are equal to
$$
R_i^{(1)}\left[\cos \left(\varphi_i^{(1)}\right)+\cos \left(\theta^{(2)}-\varphi_i^{(1)}\right)\right], \qquad R_i^{(2)} \left[ \cos(\theta^{(1)} - \varphi_i^{(2)}) - \cos(\varphi_i^{(2)}) \right], \qquad \textnormal{and}
$$
$$R_i^{(2)} \left[  \sin(\varphi_i^{(2)}) - \cos(\theta^{(1)} - \varphi_i^{(2)}) \right]
- R_i^{(1)} \left[ \sin(\varphi_i^{(1)}) - \cos(\theta^{(1)} - \varphi_i^{(1)}) \right],$$
respectively. It is now straightforward to check equality \eqref{eqref:key_equality}, which finishes the proof.
\end{proof}

\begin{remark}
    Notice that the only changes in the case of co-linear (but not equal) drift vectors is that the domain of the second and third integral in \eqref{eq:integral_breakdown} reduces to just one point, so the corresponding integrals trivialy vanish.
\end{remark}

Let us denote 
\begin{equation*}
    Y_i^{(1)}=  \left(Z_i^{(1)}-\boldsymbol{\mu}^{(1)}\right) \cdot (\e_{\theta^{(0)}}^\perp + \e_{\theta^{(1)}}),
    \qquad
    Y_i^{(2)}=  \left(Z_i^{(2)}-\boldsymbol{\mu}^{(2)}\right) \cdot (\e_{\theta^{(2)}} - \e_{\theta^{(0)}}^\perp),
\end{equation*}
and
\begin{equation*}
    W_{n,i}=\mathcal{L}_{n,i}-Y_i^{(1)}-Y_i^{(2)}.
\end{equation*}
In the following lemma we show that the error term $W_{n,i}$ is $L^2$-negligible under the scaling $\sqrt{n}$.

\begin{lemma} \label{sum goes to 0}
Assume \eqref{eq:per_assumption} and $\E\left[\|Z_1^{(k)}\|^2\right]<\infty$ for both $k \in \{1, 2\}$. Then,
\begin{equation*}
    \lim _{n \rightarrow \infty} \frac{1}{n} \sum_{i=1}^n \mathbb{E}\left[W_{n, i}^2\right]=0. 
\end{equation*}
\end{lemma}

\begin{proof}
Take $\varepsilon_1 > 0$, and let $\gamma \in (0, 1/2)$ and $\varepsilon >0$ be small enough (to be specified later). By the definition of $\mathcal{L}_{n,i}$ and $\Delta_n^{(i)}$, with the bound obtained in Lemma $\ref{lemma - omedjenost integrabilnom varijablom}$, we get
\begin{align*}
    \left|\mathcal{L}_{n,i}\right| &\le \int_{[0,\pi]}\E[|\Delta_n^{(i)}(\theta)|\mid \mathcal{F}_i]\, \mathrm{d} \theta  \le 
    2\int_{[0,\pi]}\E[ \|Z_{i}^{(1)}\|+ \|\tilde{Z}_{i}^{(1)}\|
    + \|Z_{i}^{(2)}\|+ \|\tilde{Z}_{i}^{(2)}\| \mid \mathcal{F}_i] \mathrm{d}\theta 
    \\
    &=2\pi\left(\|Z_{i}^{(1)}\|+ \E\|Z_{i}^{(1)}\|
    + \|Z_{i}^{(2)}\|+ \E\|Z_{i}^{(2)}\|\right).
\end{align*}
From the definition of $W_{n,i}$ and the triangle inequality, we obtain that
\begin{align*}
    &\left|W_{n,i}\right| \leq \left|\mathcal{L}_{n,i}\right| + \left|Y_i^{(1)}\right| + \left|Y_i^{(2)}\right|  
    \\
    &\le 2\pi\left(\|Z_{i}^{(1)}\|+ \E\|Z_{i}^{(1)}\|
    + \|Z_{i}^{(2)}\|+ \E\|Z_{i}^{(2)}\|\right) + 2(\|Z_i^{(1)}\|+\|\E[Z_1^{(1)}]\|) + 2(\|Z_i^{(2)}\|+\|\E[Z_1^{(2)}]\|) 
    \\
    &= (2\pi + 2)\left(\|Z_{i}^{(1)}\|+ \E\|Z_{i}^{(1)}\|
    + \|Z_{i}^{(2)}\|+ \E\|Z_{i}^{(2)}\|\right).
\end{align*}
Hence, under the assumption that $\E[\|Z_1^{(k)}\|^2]<\infty$ for $k \in \{1, 2\}$, there is a constant $C_0>0$ such that $\mathbb{E}[W_{n, i}^2] \leq C_0$, for all $n$ and $i$, where $C_0$ depends only on the distributions of $Z_1^{(1)}$ and $Z_1^{(2)}$. Since card$(I_{n,\gamma}^c)\le2\gamma n$, we have
\begin{align*}
    \frac{1}{n}\sum\limits_{i\not\in I_{n,\gamma}}\E[W_{n,i}^2] \le \frac{1}{n}2\gamma n C_0=2\gamma C_0.
\end{align*}
Recall that $I_{n, \gamma} = \{1, \ldots, n\} \cap [\gamma n, (1 - \gamma)n]$. Therefore, we can choose $\gamma$ small enough so that $2\gamma C_0 < \varepsilon_1$. On the other hand, for $i\in I_{n,\gamma}$, by Lemma $\ref{approximation lemma}$ we have that there exists a constant $C_1$ such that
\begin{align*}
    |W_{n, i}| \leq& C_1 ( \|Z_i^{(1)}\| + \|Z_i^{(2)}\|)\mathbb{P}\left[E_{n, i}^c(\varepsilon, \gamma) \mid \mathcal{F}_i\right]
         + C_1 \mathbb{E}\left[ (  \|\tilde{Z}_i^{(1)}\| + \|\tilde{Z}_i^{(2)}\|)  \mathbf{1}\left(E_{n, i}^c(\varepsilon, \gamma)\right) \mid \mathcal{F}_i\right] 
         \\
         &+ C_1 \lambda((\Theta^\varepsilon)^c) \left( \|Z_i^{(1)}\| + \|Z_i^{(2)}\| + \E \|Z_1^{(1)}\| + \E \|Z_1^{(2)}\| \right). 
\end{align*}
Let us start with bounding the second term. For arbitrary $B_1>0$, we have that
\begin{align*}
    &\mathbb{E}\left[ (  \|\tilde{Z}_i^{(1)}\| + \|\tilde{Z}_i^{(2)}\|)  \mathbf{1}\left(E_{n, i}^c(\varepsilon, \gamma)\right) \mid \mathcal{F}_i\right] 
    \\
    &\leq
    \mathbb{E}\left[ (  \|\tilde{Z}_i^{(1)}\| + \|\tilde{Z}_i^{(2)}\|)  \mathbf{1}\left( \{ \|\tilde{Z}_i^{(1)}\| > B_1 \} \cup \{ \|\tilde{Z}_i^{(2)}\| > B_1 \} \right) \mid \mathcal{F}_i\right] 
    + 2B_1 \mathbb{P} [E_{n, i}^c(\varepsilon, \gamma) \mid \mathcal{F}_i ] 
    \\
    &= \mathbb{E}\left[ (  \|\tilde{Z}_i^{(1)}\| + \|\tilde{Z}_i^{(2)}\|)  \mathbf{1}\left( \{ \|\tilde{Z}_i^{(1)}\| > B_1 \} \cup \{ \|\tilde{Z}_i^{(2)}\| > B_1 \} \right)\right] 
    + 2B_1 \mathbb{P} [E_{n, i}^c(\varepsilon, \gamma) \mid \mathcal{F}_i ], 
\end{align*}
where in the last line we used that $(  \|\tilde{Z}_i^{(1)}\| + \|\tilde{Z}_i^{(2)}\|)  \mathbf{1}( \{ \|\tilde{Z}_i^{(1)}\| > B_1 \} \cup \{ \|\tilde{Z}_i^{(2)}\| > B_1 \} )$ is independent of $\mathcal{F}_i$.  Because $E [\| \tilde{Z}_i^{(k)} \|] < \infty$, the dominated convergence theorem gives that 
$$
\lim_{B_1 \to \infty} \mathbb{E}\left[ (  \|\tilde{Z}_i^{(1)}\| + \|\tilde{Z}_i^{(2)}\|)  \mathbf{1}\left( \{ \|\tilde{Z}_i^{(1)}\| > B_1 \} \cup \{ \|\tilde{Z}_i^{(2)}\| > B_1 \} \right)\right] = 0.
$$
Thus, we can choose $B_1 = B_1(\varepsilon) > 1$ sufficiently large so that
\begin{align*}
  \mathbb{E}\left[ (  \|\tilde{Z}_i^{(1)}\| + \|\tilde{Z}_i^{(2)}\|)  \mathbf{1}\left( \{ \|\tilde{Z}_i^{(1)}\| > B_1 \} \cup \{ \|\tilde{Z}_i^{(2)}\| > B_1 \} \right)\right] 
  \leq
  \lambda((\Theta^\varepsilon)^c). 
\end{align*}
This implies that
\begin{align*}
    \mathbb{E}\left[ (  \|\tilde{Z}_i^{(1)}\| + \|\tilde{Z}_i^{(2)}\|)  \mathbf{1}\left(E_{n, i}^c(\varepsilon, \gamma)\right) \mid \mathcal{F}_i\right]  \leq 
    \lambda((\Theta^\varepsilon)^c) + 2B_1 \mathbb{P} [E_{n, i}^c(\varepsilon, \gamma) \mid \mathcal{F}_i ] \quad \textnormal{a.s.,}
\end{align*}
so we have that
\begin{align*}
    |W_{n, i}| \leq& C_1 ( \|Z_i^{(1)}\| + \|Z_i^{(2)}\|)\mathbb{P}\left[E_{n, i}^c(\varepsilon, \gamma) \mid \mathcal{F}_i\right]
         + C_1 \left( \lambda((\Theta^\varepsilon)^c) + 2B_1 \mathbb{P} [E_{n, i}^c(\varepsilon, \gamma) \mid \mathcal{F}_i ] \right) 
         \\
         &+ C_1 \lambda((\Theta^\varepsilon)^c)  \left( \|Z_i^{(1)}\| + \|Z_i^{(2)}\| + \E \|Z_1^{(1)}\| + \E \|Z_1^{(2)}\| \right). 
\end{align*}
Therefore, there exist a constant $C_2 > 0$ such that
\begin{align*}
     |W_{n, i}| \leq C_2 (1 + \|Z_i^{(1)}\| + \|Z_i^{(2)}\| + \E \|Z_1^{(1)}\| + \E \|Z_1^{(2)}\|) \left( 
     B_1 \mathbb{P}[E_{n, i}^c(\varepsilon, \gamma) \mid \mathcal{F}_i]
     +
     \lambda((\Theta^\varepsilon)^c)
     \right)  \quad \textnormal{a.s.}
\end{align*}
From this, we conclude that there is a constant $C_3>0$ such that
\begin{align*}
    W_{n, i}^2 \leq C_3 (1 + \|Z_i^{(1)}\| + \|Z_i^{(2)}\| + \E \|Z_1^{(1)}\| + \E \|Z_1^{(2)}\|)^2 \cdot \left( B_1^2 \mathbb{P}[E^c_{n, i}(\varepsilon, \gamma) \mid \mathcal{F}_i] + \lambda((\Theta^\varepsilon)^c )^2 \right).  
\end{align*}
Taking expectation of both sides, we get
\begin{align*}
    \E[W_{n, i}^2] \leq&\, C_3 B_1^2 \E [(1 + \|Z_i^{(1)}\| + \|Z_i^{(2)}\| + \E \|Z_1^{(1)}\| + \E \|Z_1^{(2)}\|)^2 \mathbb{P}[E^c_{n, i}(\varepsilon, \gamma) \mid \mathcal{F}_i] ] 
    \\
    &+ C_3 \lambda((\Theta^\varepsilon)^c )^2 \E[(1 + \|Z_i^{(1)}\| + \|Z_i^{(2)}\| + \E \|Z_1^{(1)}\| + \E \|Z_1^{(2)}\|)^2]. 
\end{align*}
Since $\E [\|Z_i^{(k)}\|^2] < \infty$, there exist a constant $C_4>0$ such that
\begin{align*}
    C_3 \lambda((\Theta^\varepsilon)^c )^2 \E[(1 + \|Z_i^{(1)}\| + \|Z_i^{(2)}\| + \E \|Z_1^{(1)}\| + \E \|Z_1^{(2)}\|)^2] 
    \leq C_4 \lambda((\Theta^\varepsilon)^c )^2 . 
\end{align*}
Fix $\varepsilon > 0$ sufficiently small such that $C_4 \lambda((\Theta^\varepsilon)^c )^2 < \varepsilon_1$. This is possible since $\lim_{\varepsilon\to0}\lambda((\Theta^\varepsilon)^c )=0$. Note that this choice also fixes $B_1$. Thus
\begin{align*}
    \E[W_{n, i}^2] \leq C_3 B_1^2 \E [(1 + \|Z_i^{(1)}\| + \|Z_i^{(2)}\| + \E \|Z_1^{(1)}\| + \E \|Z_1^{(2)}\|)^2 \mathbb{P}[E^c_{n, i}(\varepsilon, \gamma) \mid \mathcal{F}_i] ] + \varepsilon_1. 
\end{align*}
To deal with the first term, for $B_2>0$ we have that
\begin{align*}
    &(1 +  \|Z_i^{(1)}\| + \|Z_i^{(2)}\| + \E \|Z_1^{(1)}\| + \E \|Z_1^{(2)}\|)^2 \mathbb{P}[E^c_{n, i}(\varepsilon, \gamma) \mid \mathcal{F}_i] 
    \\
    & \leq 
    (1 + 2B_2  + \E \|Z_1^{(1)}\| + \E \|Z_1^{(2)}\|)^2 \mathbb{P}[E^c_{n, i}(\varepsilon, \gamma) \mid \mathcal{F}_i] 
    \\
    & \ \ \ \  + 
    (1 + \|Z_i^{(1)}\| + \|Z_i^{(2)}\| + \E \|Z_1^{(1)}\| + \E \|Z_1^{(2)}\|)^2 
    \cdot \mathbf{1}\left( \{ \|{Z}_i^{(1)}\| > B_2 \} \cup \{ \|{Z}_i^{(2)}\| > B_2 \} \right). 
\end{align*}
Once again, because $E[\|Z_i^{(k)}\|^2] < \infty$, the dominated convergence theorem gives us the existence of $B_2 = B_2(\varepsilon_1)$ such that
\begin{align*}
    C_3B_1^2 \E\left[(1 + \|Z_i^{(1)}\| + \|Z_i^{(2)}\| + \E \|Z_1^{(1)}\| + \E \|Z_1^{(2)}\|)^2
    \cdot \mathbf{1}\left( \{ \|{Z}_i^{(1)}\| > B_2 \} \cup \{ \|{Z}_i^{(2)}\| > B_2 \} \right)\right] <
    \varepsilon_1. 
\end{align*}
Therefore, 
\begin{align*}
    \E[W_{n, i}^2] \leq 2 \varepsilon_1 + C_3B_1^2 (1 + 2B_2  + \E \|Z_1^{(1)}\| + \E \|Z_1^{(2)}\|)^2 \mathbb{P}[E^c_{n, i}(\varepsilon, \gamma)]. 
\end{align*}
Finally, by Proposition $\ref{proposition - control of delta_n_i - perimeter}$, we may find $n_0\in\mathbb{N}
$ sufficiently large such that
\begin{align*}
    n \geq n_0 \implies  C_3B_1^2 (1 + 2B_2  + \E \|Z_1^{(1)}\| + \E \|Z_1^{(2)}\|)^2 \mathbb{P}[E^c_{n, i}(\varepsilon, \gamma)] < \varepsilon_1. 
\end{align*}
To conclude, for given $\varepsilon_1> 0$, we can find $n_0 \in \mathbb{N}$ such that for all $n \geq n_0$ we have that $\E[W_{n, i}^2] \leq 3 \varepsilon_1$ for all $i\in\I_{n,\gamma}$. Therefore
\begin{align*}
    \frac{1}{n} \sum_{i \in I_{n, \gamma}} \mathbb{E}\left[W_{n, i}^2\right] \leq 3 \varepsilon_1. 
\end{align*}
Combine the estimates for $i \not\in I_{n, \gamma}$ and $i \in I_{n, \gamma}$ to get
\begin{align*}
    \frac{1}{n} \sum_{i=1}^n \mathbb{E}\left[W_{n, i}^2\right] \leq 2 \gamma C_0+3 \varepsilon_1 \leq 4 \varepsilon_1
\end{align*}
for $n \geq n_0$. Recall that $\varepsilon_1$ was arbitrary, so this completes the proof. 
\end{proof}

\subsection{Proofs of the main results}

We are now in a position to prove the main results (for perimeter) of this paper. We start with the $L^2$ approximation result.

\begin{proof}[Proof of Theorem \ref{theorem - L2 convergence - perimeter}]
Note that
\begin{align*}
\mathbb{E}\left[W_{n, i} \mid \mathcal{F}_{i-1}\right]=\mathbb{E}\left[\mathcal{L}_{n, i} \mid \mathcal{F}_{i-1}\right]-\mathbb{E}[Y_i^{(1)} + Y_i^{(2)} \mid \mathcal{F}_{i-1}]=-\mathbb{E}[Y_i^{(1)} + Y_i^{(2)}],
\end{align*}
since $\mathcal{L}_{n, i}$ is a martingale difference sequence and $Y_i^{(1)} + Y_i^{(2)}$ is independent of $\mathcal{F}_{i-1}$. By definition, we have that $\mathbb{E}[Y_i^{(k)}]=0$, for $k \in \{1, 2\}$, thus $W_{n, i}$ is also a martingale difference sequence. Using orthogonality, we have that 
\begin{align*}
  n^{-1} \mathbb{E}\left[\left(\sum_{i=1}^n W_{n, i}\right)^2\right]=n^{-1} \sum_{i=1}^n \mathbb{E}\left[W_{n, i}^2\right],
\end{align*}
which, by Lemma \ref{sum goes to 0}, converges to zero, as $n \rightarrow$ $\infty$. Hence, we have that $n^{-1 / 2} \sum_{i=1}^n W_{n, i} \rightarrow 0$ in $L^2$. The assertion now follows from Lemma $\ref{lemma - MDS - perimeter}$.
\end{proof}

We now proceed with showing the variance asymptotics.

\begin{proof}[Proof of Theorem \ref{theorem - variance asymptotic - perimeter}]

Denote with 
\begin{align*}
    \xi_n=\frac{L_n-\mathbb{E}\left[L_n\right]}{\sqrt{n}}, \qquad \text { and } \qquad \zeta_n=\frac{1}{\sqrt{n}} \sum_{i=1}^n \left( Y_i^{(1)} + Y_i^{(2)} \right) . 
\end{align*}
Observe that $\V[\zeta_n] = \sigma^2_L$ for all $n$. From Theorem $\ref{theorem - L2 convergence - perimeter}$, $|\xi_n - \zeta_n|$ vanishes in the $L^2$ norm as $n \to \infty$. Finally, since
\begin{align*}
    \mathbb{E}\left[\left(\xi_n-\zeta_n\right) \zeta_n\right] \leq \mathbb{E}\left[\left(\xi_n-\zeta_n\right)^2\right]^{1/2} \E \left[\zeta_n^2\right]^{1/2},
\end{align*}
we conclude
\begin{align*}
    \lim_{n\to \infty}\frac{\V\left[L_n\right]}{n}
    & = \lim_{n\to \infty} \left(\mathbb{E}\left[\left(\xi_n-\zeta_n\right)^2\right]+\mathbb{E}\left[\zeta_n^2\right]+2 \mathbb{E}\left[\left(\xi_n-\zeta_n\right) \zeta_n\right] \right)  = \sigma_L^2.
\end{align*}
\end{proof}

Finally, we prove the Central Limit Theorem for the perimeter.
\begin{proof}[Proof of Theorem \ref{theorem - CLT - perimeter}]

We use the same notation as in the proof of Theorem \ref{theorem - variance asymptotic - perimeter}. By the classical central limit theorem for independent and identically distributed random variables (see \cite[Theorem 3.4.10]{durrett2019probability}),
$$
\lim _{n \rightarrow \infty} \mathbb{P}\left[\frac{\zeta_n}{\sqrt{\sigma^2_L}} \leq x\right]=\Phi(x), \qquad x \in \mathbb{R},
$$
where $\Phi$ denotes the cumulative distribution function of the standard normal distribution. By Theorem $\ref{theorem - L2 convergence - perimeter}$, $|\xi_n - \zeta_n| \rightarrow 0$ in probability. Slutsky's theorem \cite[Theorem 11.4]{gut2006probability} now implies
$$
\lim _{n \rightarrow \infty} \mathbb{P}\left[\frac{L_n-\mathbb{E}\left[L_n\right]}{\sqrt{\sigma^2_L n}} \leq x\right] = \lim _{n \rightarrow \infty} \mathbb{P}\left[\frac{\xi_n}{\sqrt{\sigma^2_L}} \leq x\right]=\Phi(x), \qquad x \in \mathbb{R}.
$$
Finally, again by Slutsky's theorem, we have
$$
\mathbb{P}\left[\frac{L_n-\mathbb{E}\left[L_n\right]}{\sqrt{\operatorname{Var}\left[L_n\right]}} \leq x\right]=\mathbb{P}\left[\frac{\xi_n \alpha_n}{\sqrt{\sigma^2_L n}} \leq x\right],
$$
where $\alpha_n := \sqrt{\frac{\sigma^2_L n}{\operatorname{Var}\left[L_n\right]}} \rightarrow 1$, as $n\to\infty$, by Theorem $\ref{theorem - variance asymptotic - perimeter}$.
\end{proof}

\section{Diameter}\label{sec:diameter}

We now turn our attention to the asymptotic behavior of the diameter process. First, we will slightly adjust the methodology we have already established for the perimeter process.

\subsection{Martingale difference sequence and Cauchy formula for diameter}

Recall that the diameter process is defined as
\begin{align*}
    D_n = \operatorname{diam} \left( \text{chull} \left\{ S_j^{(k)} : 0 \leq j \leq n, \ k = 1, 2 \right\} \right). 
\end{align*}
Similarly, as earlier, we consider the diameter of the convex hull of the resampled processes, and we denote it with 
\begin{align*}
    D_n^{(i)} := \operatorname{diam} \left( \text{chull} \left\{ S_j^{(k, i)} : 0 \leq j \leq n, \ k = 1, 2 \right\} \right). 
\end{align*}
For $0 \leq i \leq n$, define
\begin{align*}
    \mathcal{D}_{n, i}:=\mathbb{E}\left[D_{n}-D_{n}^{(i)} \mid \mathcal{F}_{i}\right],
\end{align*}
which can be interpreted as the expected change in the diameter length of the convex hull, given $\mathcal{F}_i$, on replacing the $i$-th increment in both random walks. Analogously as in Lemma \ref{lemma - MDS - perimeter}, we conclude the martingale difference sequence property of the diameter process.

\begin{lemma}\label{lemma - MDS - diameter}
Let $n \in \mathbb{N}$. Then, 
\begin{itemize}
    \item[(i)] $D_{n}-\mathbb{E}\left[D_{n}\right]=\sum_{i=1}^{n} \mathcal{D}_{n, i}$, 
    \item[(ii)] $\operatorname{Var}\left[D_{n}\right]=\sum_{i=1}^{n} \mathbb{E}\left[\mathcal{D}_{n, i}^{2}\right]$, whenever the latter sum is finite.
\end{itemize}
\end{lemma}
The standard definition of diameter is given by
$$
\operatorname{diam}(A) := \sup_{x, y \in A} \|x-y\|.
$$
An alternative representation of the diameter is given through the Cauchy formula. Denote by $\rho_A(\theta)$ the length of the set $A$ when projected on the line specified by the angle $\theta$:
$$
\rho_A(\theta) = \sup_{x \in A} (x \cdot \e_\theta) - \inf_{x \in A} (x \cdot \e_\theta).
$$
The Cauchy formula for diameter (see \cite[Lemma 6]{mcredmond2017expected}) is then given by
$$
\operatorname{diam}(A) = \sup_{\theta \in [0, \pi]} \rho_A(\theta).
$$
Figure \ref{fig: Cauchy-formula diameter} illustrates the formula. 
\begin{figure}[h!]
    \centering
    \begin{tikzpicture}[scale=1]
    \draw[->] (-3,0) -- (5,0) node[right] {$x$};
    \draw[->] (0,-1) -- (0,7) node[above] {$y$};
    
    \draw[thick] (1,1) -- (2.86,2.13) -- (3.32,3.09) -- (2.92,4.29) -- (-0.64,6.29) -- (-1.72,4.11) -- cycle;

    \draw[dashed, blue] (0, 0) -- (-1.9773,6.3781); 
    \draw[thick, dashed, black] (1, 1) -- (-0.64,6.29); 

    \draw[thick, blue] (-1.8351,5.9195) -- (-0.1952,0.6295); 
    \draw[dashed] (1, 1) -- (-0.1952,0.6295); 
    \draw[dashed] (-0.64,6.29) -- (-1.8351,5.9195); 

    \draw[black] (0.5,0) arc [start angle=0, end angle=107.22, radius=0.5];
    \node at (.2,.2) {$\theta$};
    
    \end{tikzpicture}
    \caption{Illustration of the Cauchy formula for the diameter.}
    \label{fig: Cauchy-formula diameter}
\end{figure}
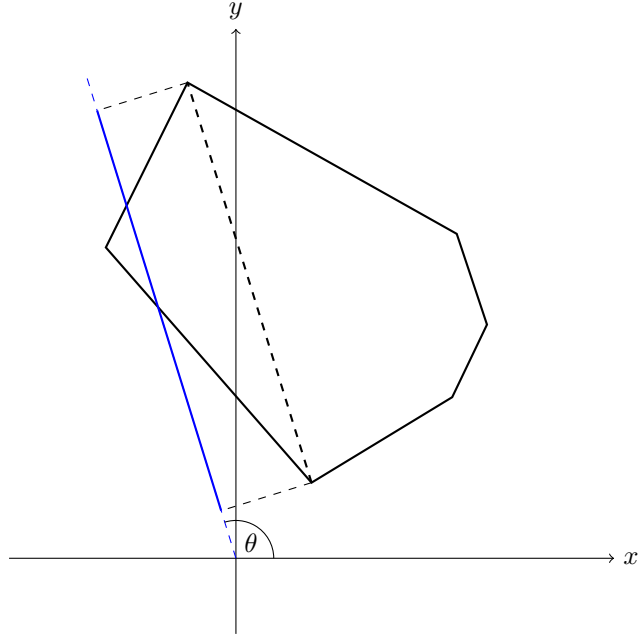
Recall that $M_n(\theta)$ and $m_n(\theta)$ denote the maximal and minimal projections, respectively, of our convex hull along the direction specified by $\e_\theta$. Using this notation, we have 
\begin{align}\label{al:cauchy_form_diam}
    D_{n}=\sup_{0 \leq \theta \leq \pi}\left(M_{n}(\theta)-m_{n}(\theta)\right)=\sup_{0 \leq \theta \leq \pi}R_{n}(\theta). 
\end{align}
Similarly, we represent $D_n^{(i)}$. Once again, we can perform the same linear transformations as for the perimeter case. So, in the further text, we again adopt the assumptions on the positions of the drift vectors presented at the beginning of Section \ref{GA}. 

Similarly as in the case of the perimeter process, we show that the convergence in the strong law of large numbers for the diameter process, presented in \eqref{eq:as_kvg_Ln_Dn}, holds also in $L^1$ sense.

\begin{corollary}\label{cor:L1_cvg_diam}
    Under the assumptions of Theorem \ref{tm:SLLN}, we have
    \begin{equation*}
        \frac{D_n}{n} \xrightarrow[n \to \infty]{L^1} \diam \left( \chull \{\boldsymbol{0}, \boldsymbol{\mu}^{(1)}, \boldsymbol{\mu}^{(2)}\} \right).
    \end{equation*}    
\end{corollary}
\begin{proof}
    Using the Cauchy formula from \eqref{al:cauchy_form_diam} we have
    \begin{align*}
        D_n = \sup_{0 \le \theta \le \pi} R_n(\theta) \le 2 \sup_{0 \le \theta \le \pi} \max_{\substack{0 \leq j \leq n \\ k = 1, 2}} |S_j^{(k)} \cdot e_{\theta}|  \le 2 \max_{\substack{0 \leq j \leq n \\ k = 1, 2}} \|S_j^{(k)}\| \le 2 \sum_{j = 0}^{n} \left( \|Z_j^{(1)}\| + \|Z_j^{(2)}\| \right).
    \end{align*}
    Using identical arguments as in the proof of Corollary \ref{cor:L1_cvg_per}, the claim follows.
\end{proof}

\subsection{Control of extrema}

Assuming \eqref{eq:diam_assumption} and, further, that $\|\boldsymbol{\mu}^{(1)}\|$ is the maximal element of the set $\{\|\boldsymbol{\mu}^{(1)}\|, \|\boldsymbol{\mu}^{(2)}\|, \|\boldsymbol{\mu}^{(1)} - \boldsymbol{\mu}^{(2)}\|\}$ (which can be done without the loss of generality) the diameter remains close to the drift direction of the first random walk. For $\delta > 0$ and $i \in \{1, \ldots, n\}$,  we define this event as follows
\begin{equation*}
A_{n, i}(\delta):=\left\{ \rho_H^1 \left( \left\{\theta^{(1)}\right\}, \;  \underset{0 \leq \theta \leq \pi}{\arg \max } \; R_n(\theta) \right) <\delta\right\} \cap\left\{\rho_H^1 \left( \left\{\theta^{(1)}\right\}, \;  \underset{0 \leq \theta \leq \pi}{\arg \max } \; R_n^{(i)}(\theta) \right) <\delta\right\}. 
\end{equation*}
 The key observation is that, with high probability, the angle of the diametral segment of the convex hull spanned by these two random  walks is close to the angle of the longest side of the triangle spanned by the corresponding drift vectors. 

\begin{thm} \label{lema s ani}
    Assume \eqref{eq:per_assumption}, \eqref{eq:diam_assumption}, the maximal element of the set from the assumption \eqref{eq:diam_assumption} is $\|\boldsymbol{\mu}^{(1)}\|$, and $\mathbb{E}[\|Z_i^{(k)}\|]<\infty$ for both $k \in \{1, 2\}$. Then, for arbitrary $\delta > 0$, 
    $$
    \lim _{n \rightarrow \infty} \min _{1 \leq i \leq n} \mathbb{P}\left(A_{n, i}(\delta)\right)=1. 
    $$
\end{thm}

Before proving Theorem \ref{lema s ani}, we introduce some notation and prove several auxiliary results we need. Denote by $(K^2, \rho_H^2)$ the space of all compact and convex sets in $\mathbb{R}^2$ equipped with the Hausdorff metric. Within this space, let $(P^2, \rho_H^2) \subseteq (K^2, \rho_H^2)$ represents the subset consisting of all compact and convex polygons. By convex polygon, we consider the convex hull of the finite set of at least two points. Further, define $(D^2, \rho_H^2) \subseteq (K^2, \rho_H^2)$ as the subset of convex and compact sets where the diameter is manifested along a unique segment. More formally, $D^2$ contains all convex and compact sets $A$ such that the set 
$$
     \underset{0 \leq \theta \leq \pi}{\arg \max } \; \rho_A(\theta)
$$ has exactly one element.
For $A\in P^2$, by $\operatorname{V}(A)$ we denote the set of its vertices.

\begin{proposition} \label{neprekidnost dijametralnog segmenta}
    The function $A \mapsto  \underset{0 \leq \theta \leq \pi}{\arg \max }\;\rho_A(\theta)$ is point-wise continuous on $(P^2\cap D^2, \rho_H^2)$. 
\end{proposition}

\begin{proof}
    In $P^2 \cap D^2$, we investigate a particular mapping that associates each polygon within the space to the line segment where it attains its diameter. Formally, we focus on the mapping given by
    $$
    (P^2 \cap D^2, \rho_H^2) \ni A \mapsto \overline{{a}_{i_1} {a}_{i_2}} \in (K^2, \rho_H^2), 
    $$
    where ${a}_{i_1}$ and ${a}_{i_2}$ represent the vertices defining the diameter. Our objective is to verify the point-wise continuity of this mapping, that is, we aim to prove that for any given $\varepsilon > 0$, there exists a corresponding $\delta = \delta(\varepsilon, A) > 0$ such that: 
    $$
    \forall \ B \in (P^2 \cap D^2, \rho_H^2)\quad \text{such that} \quad \rho_H^2(A, B) < \delta \implies \rho_H^2 \left( \overline{{a}_{i_1} {a}_{i_2} \vphantom{b_{j_1} b_{j_2}}}, \overline{b_{j_1} b_{j_2} \vphantom{{a}_{i_1} {a}_{i_2}}} \right) < \varepsilon. 
    $$
    Let $\varepsilon > 0$ and $A \in (P^2 \cap D^2, \rho_H^2)$ be arbitrarily selected. Observe first that if $A$ is a line segment, than the continuity easily follows from the triangle inequality by taking $\delta = \varepsilon/3$. In what follows we focus on polygons having at least three vertices. For such a polygon, label the vertices as ${a}_1, \ldots, {a}_n$ arranged in counterclockwise order. At each vertex, introduce $u_i^{(1)}$ and $u_i^{(2)}$ as the unit vectors oriented along the respective adjacent edges, where $u_i^{(1)}$ is directed towards $a_{i-1}$ and $u_i^{(2)}$ is directed towards $a_{i+1}$ (considering the indices modulo $n$). Furthermore, let $\varphi_i$ denote the size of the angle between the vectors $u_i^{(1)}$ and $u_i^{(1)} + u_i^{(2)}$, or $u_i^{(2)}$ and $u_i^{(1)} + u_i^{(2)}$, see Figure \ref{fig:vrh-poligona}.

\begin{figure}[h]
    \centering
    \begin{tikzpicture}[scale = 2]
        \node[circle, fill, inner sep=1.5pt, label={right:$a_i$}] (ai) at (0,0) {};
        \node[circle, fill, inner sep=1.5pt, label={below:$a_{i+1}$}] (ai+1) at (-2,-1) {};
        \node[circle, fill, inner sep=1.5pt, label={below:$a_{i-1}$}] (ai-1) at (2,-1) {};
    
        \draw[-] (ai) -- (ai+1);
        \draw[-] (ai) -- (ai-1);
        
        \draw[->, line width=1.5pt] (ai) -- (-1,-0.5) node[above] {$u_{i}^{(2)}$};
        \draw[->, line width=1.5pt] (ai) -- (1,-0.5) node[above] {$u_{i}^{(1)}$};
        
        \draw[-] (0,-0.8) arc (270:322:1) node[midway, below] {$\varphi_i$};
        \draw[-] (0,-0.9) arc (270:212:1) node[midway, below] {$\varphi_i$};
        
        \draw[dashed] (0,0) -- (0,-1.5);
    \end{tikzpicture}
    \caption{Vertex of a polygon with related vectors and angles.}
    \label{fig:vrh-poligona}
\end{figure}
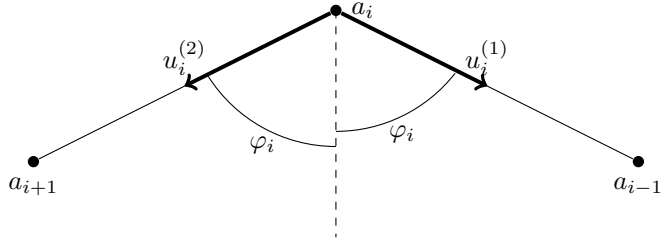
    
    Clearly, 
    \begin{align} \label{kutevi_varphi}
        \varphi_i \in \left( 0, \frac{\pi}{2} \right) 
    \end{align}
    for all $i \in \{1, \dots, n\}$. Therefore, we have that
 \begin{align*}
    \overline{\varphi} :&= \max \left\lbrace  \varphi_i :  i \in \{1, \dots, n\}\right\rbrace  \in \left( 0, \frac{\pi}{2} \right), \\
    \underline{\varphi} :&= \min \phantom{ } \left\lbrace \varphi_i :  i \in \{1, \dots, n\}\right\rbrace  \in \left( 0, \frac{\pi}{2} \right). 
\end{align*}
    Recall that ${a}_{i_1}, {a}_{i_2} \in \text{V}(A)$ is the unique pair of vertices for which 
    $$
    \operatorname{diam}(A) = \|a_{i_1} - a_{i_2}\|. 
    $$ 
    Therefore, the set
    $$
    \left\lbrace \|x - y\| : x, y \in V(A) \right\rbrace
    $$
    is finite, with a unique maximal value. Let $\varepsilon_0>0$ be defined as the difference between the two greatest values in this set. For $\delta > 0$ small enough (to be specified later), select an arbitrary polygon $B \in (P^2 \cap D^2, \rho_H^2)$ satisfying $\rho_H^2(A, B) < \delta$. Consequently, there exist points $b_{i_1}'$ and $b_{i_2}'$ in $\partial B$ such that 
    \begin{align*}
        \|a_{i_1} - b_{i_1}'\| < \delta, \qquad \textnormal{and} \qquad  \|a_{i_2} - b_{i_2}'\| < \delta. 
    \end{align*}
    However, it is worth noting that $b_{i_1}'$ and $b_{i_2}'$ are not necessarily vertices of $B$. To find vertices of $B$ that satisfy analogous relations (with modified right hand side) we proceed as follows. Consider two distinct linear optimization problems
    \begin{align*}
       (OP)_l =  \left\lbrace
            \begin{array}{c}
                 (u_{i_l}^{(1)} + u_{i_l}^{(2)})^T b \rightarrow \min \\
                 b \in B
            \end{array}
        \right. 
    \end{align*}
    for $l \in\{1, 2\}$. Given that $B$ is a convex polygon and the objective function under consideration is also convex, it follows that there exist $b_{i_1}, b_{i_2} \in V(B)$ for which $b_{i_l}$ is the solution to the $l$-th optimization problem, denoted as $(OP)_l$. Furthermore, $b_{i_l}$ must be situated in a right-angle triangle, one of whose catheti is the segment connecting 
    $$
    a_{i_l} + \delta \cdot \frac{u_{i_l}^{(1)} + u_{i_l}^{(2)}}{\|u_{i_l}^{(1)} + u_{i_l}^{(2)}\|}, \qquad \text{and} \qquad
    \ a_{i_l} - \frac{\delta}{\sin \varphi_{i_l}} \cdot \frac{u_{i_l}^{(1)} + u_{i_l}^{(2)}}{\|u_{i_l}^{(1)} + u_{i_l}^{(2)}\|}.
    $$
    Figure \ref{fig:ograda-na-kosinus-kuta} illustrates the reasoning. 
    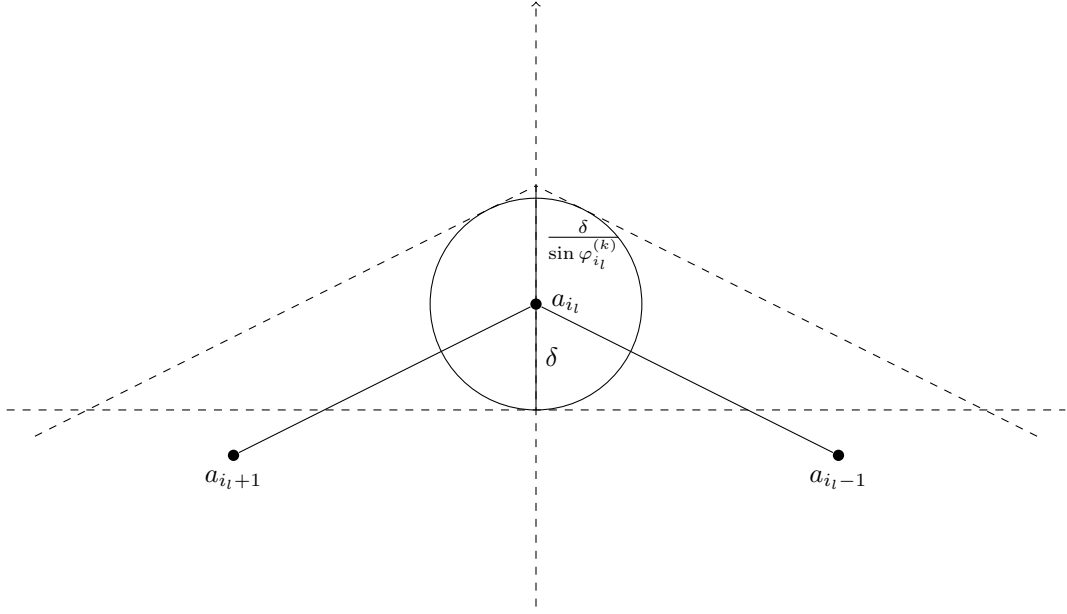
\begin{figure}[h!]
    \centering
\begin{tikzpicture}[scale = 2]
        \node[circle, fill, inner sep=1.5pt, label={right:$a_{i_l}$}] (ai) at (0,0) {};
        \node[circle, fill, inner sep=1.5pt, label={below:$a_{i_l+1}$}] (ai+1) at (-2,-1) {};
        \node[circle, fill, inner sep=1.5pt, label={below:$a_{i_l-1}$}] (ai-1) at (2,-1) {};
    
        \draw[-] (ai) -- (ai+1);
        \draw[-] (ai) -- (ai-1);
        
        

        
        \draw[->, dashed] (0,-2) -- (0,2);
        

        \draw (0,0) circle (0.7);


        \draw[dashed] (-3.31304952, -0.87390097) -- (0,0.7826237921249);
        \draw[dashed] (3.31304952, -0.87390097) -- (0,0.7826237921249);


        \draw[dashed] (-3.5,-0.7) -- (3.5,-0.7); 

        \draw (0,0) -- (0,-0.7) node[midway, right] {$\delta$};

        \draw (0,0) -- (0,0.7826237921249) node[midway, right] {$\frac{\delta}{\sin \varphi_{i_l}^{(k)}}$};
        
    \end{tikzpicture}
    \caption{Vertex of a polygon with related vectors and angles}
    \label{fig:ograda-na-kosinus-kuta}
\end{figure}
    Hence, the maximal distance between $a_{i_l}$ and $b_{i_l}$ is bounded by the length of the hypotenuse of the left or right triangle. Therefore, it has to be
    \begin{align} \label{udaljenost bil i ail}
        \| b_{i_l} - a_{i_l} \| 
        \leq 
        \frac{\delta + \frac{\delta}{\sin \varphi_{i_l}}}{\cos \varphi_{i_l}}
        \leq 
        \frac{2 \delta}{\cos \varphi_{i_l} \cdot \sin \varphi_{i_l}} 
        \leq 
        \frac{2 \delta}{\cos \overline{\varphi} \cdot \sin \underline{\varphi}}. 
    \end{align}
    Consequently, by the triangle inequality, we establish a lower bound for the distance between $b_{i_1}$ and $b_{i_2}$ given by  
    \begin{align} \label{udaljenost bi1 i bi2}
        \|b_{i_1} - b_{i_2}\| \geq \|a_{i_1} - a_{i_2}\| - \frac{4 \delta}{\cos \overline{\varphi} \cdot \sin \underline{\varphi}} = \operatorname{diam}(A) - \frac{4 \delta}{\cos \overline{\varphi} \cdot \sin \underline{\varphi}}. 
    \end{align}
    Thus, it implies that 
    \begin{align} \label{donja ograda an diamB}
        \operatorname{diam}(B) \geq \operatorname{diam}(A) - \frac{4 \delta}{\cos \overline{\varphi} \cdot \sin \underline{\varphi}}.
    \end{align}
    Next, consider the vertices $b_{j_1}$ and $b_{j_2}$ of $B$ at which the polygon $B$ attains its diameter. Our objective is to demonstrate that one vertex is close to $a_{i_1}$ and the other is close to $a_{i_2}$. To this end, there must exist points $a_{j_1}, a_{j_2} \in \partial A$ (which are not necessarily vertices) such that 
    \begin{align} \label{udaljenost bjl i ajl}
        \| b_{j_l} - a_{j_l}\| < \delta
    \end{align}
    for both $l \in\{ 1, 2\}$. We have the following lower bound on the distance between $a_{j_1}$ and $a_{j_2}$ 
    \begin{align*}
        \|a_{j_1} - a_{j_2} \| \geq \|b_{j_1} - b_{j_2} \| - 2 \delta = \operatorname{diam}(B) - 2 \delta,
    \end{align*}
    by the triangle inequality. Together with (\ref{donja ograda an diamB}), we get
    \begin{align} \label{ograda na aj1 i aj2}
        \|a_{j_1} - a_{j_2} \| 
        \geq 
        \operatorname{diam}(A) -\delta \left(
        \frac{4}{\cos \overline{\varphi} \cdot \sin \underline{\varphi}} + 2
        \right) 
        > \operatorname{diam}(A) - \varepsilon_0, 
    \end{align}
    where the second inequality holds if we choose 
    $$
    \delta < \frac{\varepsilon_0}{\left(\dfrac{4}{\cos \overline{\varphi} \cdot \sin \underline{\varphi}}+2\right)}. 
    $$
    Consider the function $f_0: \partial A \times \partial A \rightarrow \mathbb{R}^+$ defined by $f_0(x, y) = \|x-y\|$. This function is continuous if the domain is equipped with the maximum of the relative norms. Since $A$ obtains its diameter at the unique segment, the relative maxima of $f_0$ occur at pairs of vertices, with the unique maximal value attained at the pairs $(a_{i_1}, a_{i_2})$ and $(a_{i_2}, a_{i_1})$. Given that $f_0(a_{j_1}, a_{j_2}) > \operatorname{diam}(A) - \varepsilon_0$, there exists some $\delta_0 > 0$ such that the point $(a_{j_1}, a_{j_2})$ lies within a $\delta_0$-neighborhood of either $(a_{i_1}, a_{i_2})$ or $(a_{i_2}, a_{i_1})$ in the relative topology. Without loss of generality, we can assume that $(a_{j_1}, a_{j_2})$ is in a $\delta_0$-neighborhood of $(a_{i_1}, a_{i_2})$. Consequently, we obtain that 
    \begin{align*}
        \|a_{j_l} - a_{i_l} \| < \delta_0, 
    \end{align*}
    for both $l \in\{ 1, 2\}$. Therefore, we get
    \begin{align*}
        \|b_{j_l} - a_{i_l} \| \leq \|b_{j_l} - a_{j_l} \| + \|a_{j_l} - a_{i_l} \|  < \delta + \delta_0 
    \end{align*}
    for both $l \in\{ 1, 2\}$. It is worth noting that $\delta_0$ is a function of $\delta$. However, as $\delta$ approaches zero, the quantity $\|a_{j_1} - a_{j_2}\|$ approaches $\operatorname{diam}(A)$, as can be seen from inequality (\ref{ograda na aj1 i aj2}). Given this relationship, we require that $\delta > 0$ is sufficiently small such that $\delta_0$ is also sufficiently small and that $\delta + \delta_0 < \varepsilon$. Hence, it follows that if $b_{j_1}$ and $b_{j_2}$ are the vertices defining the diameter of $B$ we have that 
    \begin{align*}
        \rho_H^2 \left( \overline{{a}_{i_1} {a}_{i_2} \vphantom{b_{j_1} b_{j_2}}}, \overline{b_{j_1} b_{j_2} \vphantom{{a}_{i_1} {a}_{i_2}}} \right) < \delta + \delta_0 < \varepsilon.
    \end{align*}
    Therefore, we have successfully demonstrated that the mapping $A \mapsto \overline{a_{i_1} a_{i_2}}$ is point-wise continuous. To complete the proof, note that the function 
    $$
    \left(P^2 \cap D^2, \rho_H^2\right) \ni A \mapsto 
     \underset{0 \leq \theta \leq \pi}{\arg \max }\; \rho_A(\theta) \in [0, \pi]
    $$
    can be written as the composition of 
    $$
    \left(P^2 \cap D^2, \rho_H^2\right) \ni A \mapsto \overline{a_{i_1} a_{i_2}} \in\left(K, \rho_H^2\right), 
    $$
    and
    $$
    \left(K, \rho_H^2\right)
    \setminus
    \left\{\overline{a b}: a, b \in \mathbb{R}^2, a=b\right\} \ni \overline{a_{i_1} a_{i_2}} \mapsto 
    \arctan \left| \frac{\pi_y(a_{i_2}) - \pi_y(a_{i_1})}{\pi_x(a_{i_2}) - \pi_x(a_{i_1})}  \right| \in [0, \pi], 
    $$
    where $\pi_x$ and $\pi_y$ are projections to the $x$ and $y$ axis respectively, with the understanding that the last expression equals $\pi/2$ if the denominator is zero. Both of these functions are continuous, and therefore their composition must also be continuous.
\end{proof}

In the following corollary, we show that the function $A \mapsto \arg\max_{0 \le \theta \le \pi}{\rho_A(\theta)}$ is continuous at points from $P^2 \cap D^2$ in the space $(P^2, \rho_H^2)$. We first show the following auxiliary lemma.
\begin{lemma} \label{lemma3.1. - gustoća}
    The set $P^2 \cap D^2$ is dense in $(P^2, \rho_H^2)$. 
\end{lemma}

\begin{proof}
    Let $\varepsilon > 0$. For an arbitrary polygon $A \in (P^2, \rho_H^2)$, let $\theta'$ be an element of $\arg \max_{0 \leq \theta \leq \pi} \rho_A(\theta)$. Let the vertices of $A$ be denoted by $a_1, \ldots, a_n$ in a counterclockwise orientation. Assume that $a_{i_1}$ and $a_{i_2}$ are the vertices that correspond to the direction determined by $\theta'$. Note that some of the points $a_{i_l} + \varepsilon \mathbf{e}_{\theta'}$ lie in a direction opposite to the interior of $A$. Without  loss of generality, assume that $a_{i_1}$ is such a point. Consequently, the polygon
    $$
        A_{\varepsilon} := \operatorname{chull} \left\{ a_1, \ldots, a_{i_1 -1}, a_{i_1} + \varepsilon \mathbf{e}_{\theta'}, a_{i_1+1}, \ldots, a_n \right\},
    $$
    is an element of $P^2 \cap D^2$, and $\rho_H^2 (A, A_\varepsilon) = \varepsilon$. The statement is proven given that $\varepsilon > 0$ was chosen arbitrarily.
\end{proof}

\begin{remark} \label{remark3.1. - konstrukcija A_delta komentar}
    Observe that 
    $$
    \underset{0 \leq \theta \leq \pi}{\arg \max }\; \rho_{A_\varepsilon}(\theta) = \{ \theta' \}, 
    $$
    where $\theta'$ corresponds to the angle arbitrarily selected from $\arg \max_{0 \leq \theta \leq \pi} \rho_A(\theta)$. Thus, $A_{\varepsilon}$ is an element from $P^2 \cap D^2$ whose unique diameter is attained in the $\theta'$ direction, and at an arbitrarily small Hausdorff distance from $A$.
\end{remark}

\begin{corollary} \label{neprekidnost dijametralnog segments - korolar na proširenom prostoru}
    The function
    $$
        A \mapsto \underset{0 \leq \theta \leq \pi}{\arg \max } \;\rho_A(\theta),
    $$
is continuous at points from $P^2 \cap D^2$ in $(P^2, \rho_H^2)$.
\end{corollary}

\begin{proof}

    Let $A \in P^2 \cap D^2$ and $\varepsilon > 0$ be arbitrary. By Proposition \ref{neprekidnost dijametralnog segmenta}, there exists $\delta > 0$ satisfying:
    \begin{align} \label{korolar3.1. - deltatilda}
        \forall \ {B} \in P^2 \cap D^2 \quad \text{such that} \quad \rho_H^2(A, B) < \tilde{\delta} \implies | \theta_A - \theta_B | < \varepsilon, 
    \end{align}
    where $\theta_A$ is uniquely determined as $\{\theta_A\} = \arg \max_{0 \leq \theta \leq \pi} \rho_A(\theta)$ (the same applies to $B$). Now, consider a polygon $B \in P^2$ with $\rho_H^2(A, B) < \delta/2$, and let $\theta_B\in\arg \max_{0 \leq \theta \leq \pi} \rho_B(\theta)$ be arbitrary. According to Lemma \ref{lemma3.1. - gustoća} and Remark \ref{remark3.1. - konstrukcija A_delta komentar}, we obtain the existence of a polygon $B_{\theta_B} \in P^2 \cap D^2$ satisfying 
    $$\underset{0 \leq \theta \leq \pi}{\arg \max} \;\rho_{B_{\theta_B}}(\theta)=\{\theta_B\},
    $$ 
    with $\rho_H^2(B, B_{\theta_B}) = \delta/2$. It follows that 
    $$
        \rho_H^2(A, B_{\theta_B}) \leq \rho_H^2(A, B) + \rho_H^2(B, B_{\theta_B}) < \delta/2 + \delta/2 = \delta.
    $$
    Applying equation (\ref{korolar3.1. - deltatilda}), it becomes evident that $|\theta_B - \theta_A| < \varepsilon$. As $\theta_B$ was selected arbitrarily, we deduce that 
    $$
    \rho_H^1 \left( \theta_A, \underset{0 \leq \theta \leq \pi}{\arg \max}\; \rho_{B}(\theta) \right) = \sup \left\{ | \theta_A - \theta_B|: \theta_B \in  \underset{0 \leq \theta \leq \pi}{\arg \max}\; \rho_{B}(\theta) \right\} < \varepsilon,
    $$
    thereby confirming the corollary.
\end{proof}

\begin{remark}
The density of $P^2$ in $(K^2, \rho_H^2)$ (see \cite{schneider1981approximation}) suggests that the previous proof could be adapted to $(K^2, \rho_H^2)$ as the domain of interest. Namely, every convex and compact subset of $\mathbb{R}^2$ can be arbitrarily well approximated with the convex polygon, and by Lemma \ref{lemma3.1. - gustoća} every convex polygon can be arbitrarily well approximated with the convex polygon with the unique diametrical segment. Hence, if we apply the triangle inequality twice, we would obtain the claimed statement. 

Furthermore, the preceding corollary does not offer insights into the continuity of the mapping when applied to other polygons in $P^2$. In fact, it is possible to show with relative ease that the function manifests discontinuities when evaluated on polygons characterized by two or more diametrical segments.
\end{remark}

We are now in position to prove Theorem \ref{lema s ani}.
\begin{proof}[Proof of Theorem \ref{lema s ani}]
    From Theorem \ref{tm:SLLN}, we have that
    \begin{align} \label{teorem - prvo konvergencija ljuski}
        n^{-1} \operatorname{chull}\left\{S_j^{(k)}: 0 \leq j \leq n, k=1,2\right\} 
        \xrightarrow[n \to \infty]{a.s.}
        \operatorname{chull}\left\{\{\boldsymbol{0}\} \cup\left\{\boldsymbol{\mu}^{(k)}: k=1,2\right\}\right\} . 
    \end{align}
    Denote by $A$ the right hand side in \eqref{teorem - prvo konvergencija ljuski}. Since $A \in P^2 \cap D^2$, we find that 
    $$
        \underset{0 \leq \theta \leq \pi}{\arg \max } \;\rho_A(\theta) = \{ \theta^{(1)} \}.
    $$
    Furthermore, denote
    $$
        A_n = n^{-1} \operatorname{chull}\left\{S_j^{(k)}: 0 \leq j \leq n, \ k=1,2\right\}.
    $$
    Using Corollary \ref{neprekidnost dijametralnog segments - korolar na proširenom prostoru} with the continuous mapping theorem (see \cite[Theorem 3.2.10]{durrett2019probability}) yields the following
    \begin{align*}
        \underset{0 \leq \theta \leq \pi}{\arg \max } \ \rho_{A_n}(\theta)  
        \xrightarrow[n \to \infty]{a.s.}
        \theta^{(1)}. 
    \end{align*}
    It is worth noting that scaling does not impact the direction of the diametrical segment. Consequently, it holds that 
    $$
    \underset{0 \leq \theta \leq \pi}{\arg \max }\;\rho_{A_n}(\theta) = \underset{0 \leq \theta \leq \pi}{\arg \max } \; R_n(\theta).
    $$
    As a result, for a given $\delta > 0$, there exists an almost surely finite random variable $N_\delta$ such that:
    $$
    n \geq N_\delta \implies \rho_H^1\left(\left\{ \theta^{(1)} \right\}, \underset{0 \leq \theta \leq \pi}{\arg \max }\;R_n(\theta)\right) < \delta.
    $$
    Therefore, we get that 
    \begin{align} \label{konvergencija za rho_H^1}
         \mathbb{P} \left( \rho_H^1\left(\left\{ \theta^{(1)} \right\}, \underset{0 \leq \theta \leq \pi}{\arg \max } \;R_n(\theta)\right)<\delta \right) 
         \geq \mathbb{P}( n \geq N_\delta )  
         \rightarrow 
         \mathbb{P}( N_\delta < \infty ) = 1.  
    \end{align}
    It is important to observe that the distribution of 
    $$
    \underset{0 \leq \theta \leq \pi}{\arg \max }\;R_n^{(i)}(\theta)
    $$ 
    coincides with that of $\arg \max _{0 \leq \theta \leq \pi} R_n(\theta)$. Finally, we have that
    $$
    \min _{1 \leq i \leq n} \mathbb{P}\left(A_{n, i}(\delta)\right)
    \geq
    1 - 2 \mathbb{P} \left( \rho_H^1\left(\left\{ \theta^{(1)} \right\}, \underset{0 \leq \theta \leq \pi}{\arg \max }\;R_n(\theta)\right) \geq \delta \right),
    $$
    which concludes the proof.
\end{proof}

\subsection{Approximation lemma for diameter}

The main goal of this subsection is to prove an analogue of Lemma \ref{approximation lemma} for the diameter. More precisely, we aim to compare $\mathcal{D}_{n, i}$ with appropriately centered and projected $i$-th step of the walk with the dominating drift (this can be first walk, second walk, or the difference walk, but, without loss of generality, as before, we assume that this is the first walk). As in the case of the perimeter, the proof of the lemma depends on our earlier assumptions regarding the spatial orientation of the drift vectors relative to the $y$-axis. We again consider only the case where the $y$-axis is situated between the drift vectors. Notice further that the assumption that $\|\boldsymbol{\mu}^{(1)}\|$ is the maximal element of the set $\{\|\boldsymbol{\mu}^{(1)}\|, \|\boldsymbol{\mu}^{(2)}\|, \|\boldsymbol{\mu}^{(1)}\| - \|\boldsymbol{\mu}^{(2)}\|\}$ positions the drift vector $\boldsymbol{\mu}^{(1)}$ in the region $\Theta_{(1>2>0)}^{\varepsilon}$.

Having in mind the previous discussion, it is clear that for a given, sufficiently small, $\delta > 0$, we can select and fix sufficiently small $\varepsilon = \varepsilon(\delta)$ such that
$$
(\theta^{(1)} - \delta, \theta^{(1)} + \delta) \subseteq \Theta_{(1>2>0)}^{\varepsilon}.
$$ 

In the following lemma we describe the behavior of the parametrized range function in this interval.

\begin{lemma} \label{lemma3.2.}
    Let $\gamma \in(0,1 / 2)$. Then for $\delta > 0$ and $\varepsilon > 0$ from the upper discussion, and any $i \in I_{n, \gamma}$, on $E_{n, i}(\varepsilon, \gamma)$,
\begin{equation*}
    \left|\sup _{|\theta-\theta^{(1)}| \leq \delta} R_n(\theta)-\sup _{|\theta-\theta^{(1)}| \leq \delta} R_n^{(i)}(\theta)-\left(Z_i^{(1)}-\tilde{Z}_i^{(1)}\right) \cdot \e_{\theta^{(1)}}\right| \leq 2\delta\left\|Z_i^{(1)}-\tilde{Z}_i^{(1)}\right\| .
\end{equation*}
\end{lemma}

\begin{proof}
We assert that for every $i$ belonging to the set $I_{n, \gamma}$, and for any $\theta_1$ and $\theta_2$ within the interval $\left(\theta^{(1)}-\delta, \theta^{(1)}+\delta\right)$ satisfying $\theta_1<\theta_2$, the following condition holds on the event $E_{n, i}(\varepsilon, \gamma)$:
\begin{equation} \label{eq:razlika_supremuma_1}
  \inf_{\theta_1 \leq \theta \leq \theta_2}\left(Z_i^{(1)}-\tilde{Z}_i^{(1)}\right) \cdot \mathbf{e}_\theta 
\leq 
\sup_{\theta_1 \leq \theta \leq \theta_2} R_n(\theta)
-
\sup _{\theta_1 \leq \theta \leq \theta_2} R_n^{(i)}(\theta) 
\leq 
\sup_{\theta_1 \leq \theta 
\leq\theta_2}\left(Z_i^{(1)}-\tilde{Z}_i^{(1)}\right) \cdot \mathbf{e}_\theta .  
\end{equation}
Furthermore, it can be easily verified that for any $x \in \mathbb{R}^2$, and any $\theta_1, \theta_2 \in\mathbb{R}$, the following holds
\begin{equation*}
\left|{x} \cdot \mathbf{e}_{\theta_1}-{x} \cdot \mathbf{e}_{\theta_2}\right| 
\leq
\|{x}\|
\mid \theta_1-\theta_2 \mid .
\end{equation*}
From this, we have that
$$
\begin{aligned}
\sup_{\theta_1 \leq \theta 
\leq\theta_2} \left(Z_i^{(1)}-\tilde{Z}_i^{(1)}\right) \cdot \mathbf{e}_\theta
 & = \sup_{\theta_1 \leq \theta 
\leq\theta_2} \left(Z_i^{(1)}-\tilde{Z}_i^{(1)}\right) \cdot (\mathbf{e}_{\theta^{(1)}} + \mathbf{e}_\theta - \mathbf{e}_{\theta^{(1)}}) \\
& \leq\left(Z_i^{(1)}-\tilde{Z}_i^{(1)}\right) \cdot \mathbf{e}_{\theta^{(1)}}+2\delta\left\|Z_i^{(1)}-\tilde{Z}_i^{(1)}\right\|.
\end{aligned}
$$
By analogous argumentation we conclude
\begin{equation*}
    \inf_{\theta_1 \leq \theta 
\leq\theta_2} \left(Z_i^{(1)}-\tilde{Z}_i^{(1)}\right) \cdot \mathbf{e}_\theta \geq\left(Z_i^{(1)}-\tilde{Z}_i^{(1)}\right) \cdot \mathbf{e}_{\theta^{(1)}}-2\delta\left\|Z_i^{(1)}-\tilde{Z}_i^{(1)}\right\|.
\end{equation*}
The assertion of the lemma follows by taking $\theta_1=\theta^{(1)}-\delta$ and $\theta_2=\theta^{(1)}+\delta$. We are left to prove \eqref{eq:razlika_supremuma_1}. Observe that for functions $f, g: \mathbb{R} \rightarrow \mathbb{R}$, satisfying $\sup _{\theta \in I}|f(\theta)|<\infty$ and $\sup _{\theta \in I}|g(\theta)|<\infty$, for $I \subseteq \mathbb{R}$,
\begin{equation*}
\inf _{\theta \in I}(f(\theta)-g(\theta)) \leq \sup _{\theta \in I} f(\theta)-\sup _{\theta \in I} g(\theta) \leq \sup _{\theta \in I}(f(\theta)-g(\theta)) .
\end{equation*}
In particular, if $I=[\theta_1, \theta_2]$ with $\theta_1, \theta_2 \in (\theta^{(1)}-\delta, \theta^{(1)}+\delta)$, we have
$$
\inf _{\theta_1 \leq \theta \leq \theta_2}\left(R_n(\theta)-R_n^{(i)}(\theta)\right) \leq \sup _{\theta_1 \leq \theta \leq \theta_2} R_n(\theta)-\sup _{\theta_1 \leq \theta \leq \theta_2} R_n^{(i)}(\theta) \leq \sup _{\theta_1 \leq \theta \leq \theta_2}\left(R_n(\theta)-R_n^{(i)}(\theta)\right). 
$$
Moreover, on the event $E_{n, i}(\varepsilon, \gamma)$, according to Proposition \ref{proposition - control of delta_n_i - perimeter}, for all $\theta \in [\theta_1, \theta_2 ]$ we have
$$
R_n(\theta)-R_n^{(i)}(\theta)=\left(Z_i^{(1)}-\tilde{Z}_i^{(1)}\right) \cdot \mathbf{e}_\theta.
$$
This proves claim \eqref{eq:razlika_supremuma_1}.
\end{proof}

We are now ready to prove the approximation result for $\mathcal{D}_{n, i}$. First, let us denote 
$$
B_{n, i}(\gamma, \delta) := E_{n, i}(\varepsilon(\delta), \gamma) \cap A_{n, i}(\delta).
$$ 

\begin{lemma} \label{lemma - approximation lemma - diameter}
    Assume that $E[\|Z_1^{(k)}\|] < \infty$ for both $k \in \{1,2\}$. Let $\gamma \in (0,1/2)$, and let $ \varepsilon$ and $ \delta$ be as in the previous lemma.  Then, for any $i \in I_{n, \gamma}$, the following inequality holds a.s.
    \begin{equation}\label{approximation lemma - inequality}
    \begin{aligned}
        \left|\mathcal{D}_{n, i}-(Z_i^{(1)}-\boldsymbol{\mu}^{(1)}) \cdot \e_{\theta^{(1)}}\right|
        & \leq  
        3\left( \|Z_i^{(1)}\| +  \| Z_i^{(2)}\| \right) \mathbb{P}\left(B_{n, i}^{\mathrm{c}}(\gamma, \delta) \mid \mathcal{F}_i\right) + 2\delta\left(\|Z_i^{(1)}\|+\mathbb{E}\|Z_1^{(1)}\|\right)
        \\
        & \quad +
        3 \mathbb{E}\left[ \left( \|\tilde{Z}_i^{(1)}\| + \|\tilde{Z}_i^{(2)}\|\right) \mathbf{1}\left(B_{n, i}^{\mathrm{c}}(\gamma, \delta)\right) \mid \mathcal{F}_i\right].
    \end{aligned}
    \end{equation}
\end{lemma}

\begin{proof}
Given that the random variable $Z_i^{(1)}$ is $\mathcal{F}_i$-measurable and that $\tilde{Z}_i^{(1)}$ is independent of $\mathcal{F}_i$, it follows that
$$
\mathcal{D}_{n, i}-\left(Z_i^{(1)}-\boldsymbol{\mu}^{(1)}\right) \cdot \e_{\theta^{(1)}}
=
\mathbb{E}\left[D_n-D_n^{(i)}-\left(Z_i^{(1)}-\tilde{Z}_i^{(1)}\right) \cdot \e_{\theta^{(1)}} \mid \mathcal{F}_i\right],
$$
from which it follows
$$
\begin{aligned}
\left|\mathcal{D}_{n, i}-\left(Z_i^{(1)}-\boldsymbol{\mu}^{(1)}\right) \cdot \e_{\theta^{(1)}}\right| 
&\leq 
\mathbb{E}\left[\left|D_n-D_n^{(i)}-\left(Z_i^{(1)}-\tilde{Z}_i^{(1)}\right) \cdot \e_{\theta^{(1)}}\right| \mathbf{1}\left(B_{n, i}(\gamma, \delta)\right) \mid \mathcal{F}_i\right] \\
&\quad +\mathbb{E}\left[\left|D_n-D_n^{(i)}-\left(Z_i^{(1)}-\tilde{Z}_i^{(1)}\right) \cdot \e_{\theta^{(1)}}\right| \mathbf{1}\left(B_{n, i}^{\mathrm{c}}(\gamma, \delta)\right) \mid \mathcal{F}_i\right] .
\end{aligned}
$$
From Lemma \ref{lemma - omedjenost integrabilnom varijablom}, we next establish
$$
\mathbb{E}\left[\left|D_n-D_n^{(i)}-\left(Z_i^{(1)}-\tilde{Z}_i^{(1)}\right) \cdot \e_{\theta^{(1)}}\right| \mathbf{1}\left(B_{n, i}^{\mathrm{c}}(\gamma, \delta)\right) \mid \mathcal{F}_i\right] \leq  \mathbb{E}\left[S \cdot  \mathbf{1}\left(B_{n, i}^{\mathrm{c}}(\gamma, \delta)\right) \mid \mathcal{F}_i\right], 
$$
where
$$
S = 3\left(\left\|Z_i^{(1)}\right\|+\left\|\tilde{Z}_i^{(1)}\right\|\right)+2 \left(\left\|Z_i^{(2)}\right\|+\left\|\tilde{Z}_i^{(2)}\right\|\right).  
$$
Now, on the event $A_{n, i}(\delta)$, we have that
$$
D_n=\sup _{|\theta-\theta^{(1)}| \leq \delta} R_n(\theta), \qquad \text { and } \qquad D_n^{(i)}=\sup _{|\theta-\theta^{(1)}| \leq \delta} R_n^{(i)}(\theta),
$$
and hence, by Lemma \ref{lemma3.2.}, on the event $B_{n, i}(\gamma, \delta)$,
$$
\left|D_n-D_n^{(i)}-\left(Z_i^{(1)}-\tilde{Z}_i^{(1)}\right) \cdot \e_{\theta^{(1)}}\right| \leq 2\delta\left\|Z_i^{(1)}-\tilde{Z}_i^{(1)}\right\| .
$$
Consequently
$$
\mathbb{E}\left[\left| D_n-D_n^{(i)}-\left(Z_i^{(1)}-\tilde{Z}_i^{(1)}\right) \cdot \e_{\theta^{(1)}}\right| \mathbf{1}\left(B_{n, i}(\gamma, \delta)\right) \mid \mathcal{F}_i\right] \leq 2\delta \mathbb{E}\left[\left\|Z_i^{(1)}\right\|+\left\|\tilde{Z}_i^{(1)}\right\| \mid \mathcal{F}_i\right].
$$
\end{proof}

Let us denote $V_i := (Z_i^{(1)} - \boldsymbol{\mu}^{(1)}) \cdot \e_{\theta^{(1)}}$, and $U_{n, i} := \mathcal{D}_{n, i} - V_i$. In the following lemma we show that the error term $U_{n,i}$ is $L^2$-negligible under the scaling $\sqrt{n}$.

\begin{lemma} \label{lemma5.1. - diameter - sum goes to 0}
Assume \eqref{eq:per_assumption}, \eqref{eq:diam_assumption} and that $\E[\|Z_1^{(k)}\|^2]<\infty$ for both $k\in \{1,2\}$. Then
\begin{equation*}
    \lim _{n \rightarrow \infty} \frac{1}{n} \sum_{i=1}^n \mathbb{E}\left[U_{n, i}^2\right]=0. 
\end{equation*}
\end{lemma}

\begin{proof}
    For a given $\varepsilon \in (0,1)$, let $\gamma \in (0, 1/2)$ and $\delta > 0$ be sufficiently small, the specifics of which will be clarified later. From Lemma \ref{lemma - omedjenost integrabilnom varijablom}, we have that
    $$
    \left|U_{n, i}\right| \leq 
    3\left(\|Z_i^{(1)}\|+\mathbb{E}\|Z_1^{(1)}\| + \|Z_i^{(2)}\|+\mathbb{E}\|Z_1^{(2)}\|\right).
    $$
    Thus, $\mathbb{E}\left( U_{n, i}^2 \right) \leq C_0$ for all $n$ and all $i$, and for some constant $C_0>0$ whose value depends solely on the distributions of $Z_i^{(k)}$. Therefore, we can conclude that 
$$
\frac{1}{n} \sum_{i \notin I_{n, \gamma}} \mathbb{E}\left(U_{n, i}^2\right) \leq 2 \gamma C_0 .
$$
We choose and fix $\gamma > 0$ sufficiently small to ensure that $2 \gamma C_0 < \varepsilon$. Now, for $i \in I_{n, \gamma}$, Lemma \ref{lemma - approximation lemma - diameter} provides an upper bound on $\left| U_{n, i} \right|$. Note that for any constant $C_1^{(k)}>0$, given that $\tilde{Z}_i^{(k)}$ is independent of $\mathcal{F}_i$, we can conclude the following
$$
\mathbb{E}\left[\|\tilde{Z}_i^{(k)}\| \mathbf{1}\left(B_{n, i}^{\mathrm{c}}(\gamma, \delta)\right) \mid \mathcal{F}_i\right] 
\leq 
\mathbb{E}\left[\|Z_i^{(k)}\| \mathbf{1}\left(\|Z_i^{(k)}\| \geq C_1^{(k)}\right)\right]
+
C_1^{(k)} \mathbb{P}\left[B_{n, i}^{\mathrm{c}}(\gamma, \delta) \mid \mathcal{F}_i\right]. 
$$
Given $\varepsilon \in (0,1)$, we can select $C_1 = C_1(\varepsilon)>0$ sufficiently large such that 
$$
\mathbb{E}\left[\| Z_i^{(k)} \| \mathbf{1}\left( \| Z_i^{(k)} \| \geq C_1 \right)\right] < \varepsilon
$$
for both $k \in \{1, 2\}$. For the sake of convenience, we also choose $C_1 > 1$ and $C_1 > \mathbb{E}[\| Z_1^{(k)} \|]$ for both $k \in \{1, 2\}$. Consequently, by Lemma \ref{lemma - approximation lemma - diameter}, we obtain that
$$
\left|U_{n, i}\right| \leq 3\left(\|Z_i^{(1)}\| + \|Z_i^{(2)}\|+2C_1\right) \mathbb{P}\left[B_{n, i}^{\mathrm{c}}(\gamma, \delta) \mid \mathcal{F}_i\right]
+6 \varepsilon
+2\delta\left(\|Z_i^{(1)}\|+\mathbb{E}\|Z_1^{(1)}\|\right). 
$$
Using $\mathbb{P}\left[B_{n, i}^{\mathrm{c}}(\gamma, \delta) \mid \mathcal{F}_i\right] \leq 1$, $\varepsilon \leq 1$, $\delta \leq 1$, and the elementary inequality $(a + b + c)^2 \le 3(a^2 + b^2 + c^2)$ for positive $a, b, c \in \mathbb{R}$, we conclude
\begin{equation*}
U_{n, i}^2 
\leq 
27 \left( \|Z_i^{(1)}\| + \|Z_i^{(2)}\|  + 2C_1 \right)^2 
\mathbb{P}\left[B_{n, i}^{\mathrm{c}}(\gamma, \delta) \mid \mathcal{F}_i\right]
+
108 \varepsilon
+
12 \delta \left( \|Z_i^{(1)}\| + \E \|Z_1^{(1)}\| \right)^2.
\end{equation*}
By assumption, for a given $\varepsilon$, there is $\delta$ small enough such that

$$
12 \delta \left( \|Z_i^{(1)}\| + \E \|Z_1^{(1)}\| \right)^2 < \varepsilon.
$$
We shall fix such $\delta>0$ for the remainder of the discussion. We then have
$$
\mathbb{E}\left(U_{n, i}^2\right) 
\leq 
27 \E \left[ \left( \|Z_i^{(1)}\| + \|Z_i^{(2)}\|  + 2C_1 \right)^2 
\mathbb{P}\left[B_{n, i}^{\mathrm{c}}(\gamma, \delta) \mid \mathcal{F}_i\right] \right] 
+
109 \varepsilon. 
$$
Next, for any $C_2>0$, we have that
\begin{align*}
&\E \left[ \left( \|Z_i^{(1)}\| + \|Z_i^{(2)}\|  + 2C_1 \right)^2 
\mathbb{P}\left[B_{n, i}^{\mathrm{c}}(\gamma, \delta) \mid \mathcal{F}_i\right] \right]\\
&\leq
C_2^2 \mathbb{P}\left(B_{n, i}^{\mathrm{c}}(\gamma, \delta)\right) 
+
\mathbb{E}\left[ \left( \|Z_i^{(1)}\| + \|Z_i^{(2)}\|  + 2C_1 \right)^2 
\mathbf{1}\left( \|Z_i^{(1)}\| + \|Z_i^{(2)}\|  + 2C_1 \geq C_2\right)\right],
\end{align*}
and using the dominated donvergence theorem, it is possible to choose a value $C_2$ sufficiently large such that the last term is less than ${\varepsilon}/{27}$. Consequently, with this choice of $C_2$, we have
$$
\mathbb{E}\left(U_{n, i}^2\right) \leq 110 \varepsilon+27 C_2^2\mathbb{P}\left(B_{n, i}^{\mathrm{c}}(\gamma, \delta)\right). 
$$
By Theorem \ref{lema s ani} and Proposition \ref{proposition - control of delta_n_i - perimeter}, we conclude
$$
\lim_{n\to\infty}\max _{1 \leq i \leq n} \mathbb{P}\left(B_{n, i}^{\mathrm{c}}(\gamma, \delta)\right) =0, 
$$
so that, for given $\varepsilon>0$ (and hence $C_1$ and $C_2$) we may select $n_0 \in \mathbb{N}$ sufficiently large so that $\max _{i \in I_{n, \gamma}} \mathbb{E}\left(U_{n, i}^2\right) \leq 111 \varepsilon$, for $n \ge n_0$. Consequently,
$$
\frac{1}{n} \sum_{i \in I_{n, \gamma}} \mathbb{E}\left(U_{n, i}^2\right) \leq 111 \varepsilon,
$$
for all $n \geq n_0$. Combining this with the earlier approximation for $i \notin I_{n, \gamma}$, we arrive at
$$
\frac{1}{n} \sum_{i=1}^n \mathbb{E}\left(U_{n, i}^2\right) \leq 112 \varepsilon,
$$
for all $n \geq n_0$. Since $\varepsilon > 0$ was arbitrarily chosen, the conclusion follows. 
\end{proof}

\subsection{Proofs of the main results}

The proofs of the main results for the diameter in the largest part follow as in the perimeter case. In the proof of Theorem \ref{theorem - diameter L2 convergence} instead of using Lemma \ref{lemma - MDS - perimeter} and Lemma \ref{sum goes to 0}, we employ Lemma \ref{lemma - MDS - diameter} and Lemma \ref{lemma5.1. - diameter - sum goes to 0}. The proof of Theorem \ref{theorem - variance asymptotic - diameter} follows analogously as the proof of Theorem \ref{theorem - variance asymptotic - perimeter} with the only difference being in the definition of the sequence $(\zeta_n)_{n = 1}^{\infty}$:
\begin{equation}
    \zeta_n = \frac{1}{\sqrt{n}} \sum_{i = 1}^{n} \left( Z_i^{(1)} - \boldsymbol{\mu}^{(1)} \right) \cdot \e_{\theta^{(1)}}.
\end{equation}
The proof of the Central Limit Theorem presented in Theorem \ref{theorem - CLT - diameter} remains the same, by replacing the constant $\sigma_L$ with the constant $\sigma_D$.

\section{Final discussions}\label{sec:final_discussions}

In this section, we turn our attention to certain situations that may have remained unexplained. To begin with, an interesting question arises when we assume that $\|\boldsymbol{\mu}^{(1)} - \boldsymbol{\mu}^{(2)}\|$ is the unique maximal element of the set from the assumption \eqref{eq:diam_assumption}. Under this assumption, Theorem \ref{theorem - diameter L2 convergence} should be restated in the following form
$$
n^{-1 / 2}\left|D_n-\mathbb{E} D_n- \left( \left(S_n^{(1)}-\mathbb{E} S_n^{(1)}\right) - \left(S_n^{(2)}-\mathbb{E} S_n^{(2)}\right)  \right)  \cdot \mathbf{e}_{\theta}\right| \xrightarrow[n \to \infty]{L^2} 0. 
$$
Here, $\theta \in [0, \pi)$ denotes the angle corresponding to the direction of $\boldsymbol{\mu}^{(1)} - \boldsymbol{\mu}^{(2)}$. Consequently, both random walks under consideration contribute to the asymptotic behavior of the diameter in this particular scenario.

Furthermore, the analysis of the diameter of the convex hull spanned by two random walks can be adapted to the more general scenario involving $m$ independent random walks. The asymptotic behavior of the diameter is controlled by the geometric characteristics of the set
$$
\left\lbrace \|x-y\| : {x}, {y} \in \{0\} \cup \left\{ \boldsymbol{\mu}^{(k)} : k = 1, \dots, m \right\} \right\rbrace.
$$
Namely, if this set possesses a unique maximal element, it can be demonstrated that
$$
n^{-1 / 2}\left|D_n-\mathbb{E} D_n- \left( \left(S_n^{(k_{1})}-\mathbb{E} S_n^{(k_1)}\right) - \left(S_n^{(k_2)}-\mathbb{E} S_n^{(k_2)}\right)  \right)  \cdot \mathbf{e}_{\theta}\right| \xrightarrow[n \to \infty]{L^2} 0,
$$
where $k_1, k_2 \in \{1, \dots, m\}$ are selected such that $\|\boldsymbol{\mu}^{(k_1)} - \boldsymbol{\mu}^{(k_2)}\|$ is the (unique) largest element in the set mentioned above, and $\theta$ keeps its role as the angle corresponding to the direction $\boldsymbol{\mu}^{(k_1)} - \boldsymbol{\mu}^{(k_2)}$. Similarly, the present discussion does not readily extend to scenarios involving multiple maximal elements within the set.

    
    
    

    


On the other hand, the problem of the perimeter of the convex hull generated by multiple independent random walks is much more demanding. To more effectively handle the extrema ($M_n(\theta)$ and $m_n(\theta)$), one might consider employing an alternative version of the Cauchy formula for the perimeter, given by
$$
L_n =  \int_0^{2 \pi} M_n(\theta) \text{d} \theta. 
$$
Yet, the specific random walks contributing to this integral depend upon the geometric properties of the convex hull formed by their respective drift vectors. If the convex hull of these drift vectors coincides with the convex hull of a subset (of the original set) of drift vectors where all are non-zero, the argument above can be adjusted to arrive at a Gaussian limiting distribution. An example of such a setting is presented in Figure \ref{fig:different situations of k0 walks} on the right graph. On the other hand, if a zero-drift random walk exists and zero lies on the boundary of the convex hull, a non-Gaussian limit can be anticipated. An example of such a setting is presented in Figure \ref{fig:different situations of k0 walks} on the left graph.

    
    
    





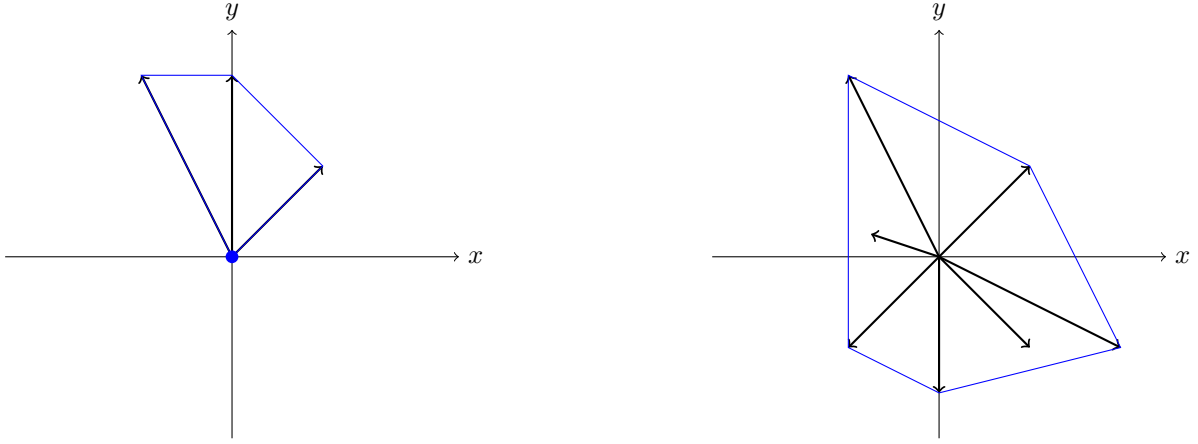
\begin{figure}[h!]
  \centering
  \begin{minipage}{0.45\textwidth}
    \centering
    \begin{tikzpicture}[scale=1.2]
      \draw[->] (-2.5,0) -- (2.5,0) node[right] {$x$};
      \draw[->] (0,-2) -- (0,2.5) node[above] {$y$};
      \draw[->, thick] (0,0) -- (1,1); 
      \draw[->, thick] (0,0) -- (-1,2); 
      \draw[->, thick] (0,0) -- (0,2); 
      \draw[blue] (0,0) -- (1,1) -- (0,2) -- (-1,2) -- cycle;
      \fill[blue] (0,0) circle (2pt); 
    \end{tikzpicture}
  \end{minipage}
  \hfill
  \begin{minipage}{0.45\textwidth}
    \centering
    \begin{tikzpicture}[scale=1.2]
      \draw[->] (-2.5,0) -- (2.5,0) node[right] {$x$};
      \draw[->] (0,-2) -- (0,2.5) node[above] {$y$};
      \draw[->, thick] (0,0) -- (1,1); 
      \draw[->, thick] (0,0) -- (-1,2); 
      \draw[->, thick] (0,0) -- (-1,-1); 
      \draw[->, thick] (0,0) -- (0,-1.5); 
      \draw[->, thick] (0,0) -- (2,-1); 
      \draw[->, thick] (0,0) -- (1,-1); 
      \draw[->, thick] (0,0) -- (-0.75,0.25); 
      \draw[blue] (1,1) -- (-1,2) -- (-1,-1) -- (0,-1.5) -- (2,-1) -- cycle;
    \end{tikzpicture}
    
  \end{minipage}
  \caption{Different positions of drift vectors}
  \label{fig:different situations of k0 walks}
\end{figure}

\section{Simulation study and open problems}\label{sec:simulations}

In \cite{wade2015convex2} it has been shown that for a single zero-drift planar random walk, $(L_n - \mathbb{E}[L_n])/\sqrt{n}$ converges in distribution to a non-degenerate non-Gaussian limit. We conjecture a similar phenomena in the case of two planar random walks when the assumption \eqref{eq:per_assumption} is not satisfied, that is, $\boldsymbol{0} \in \{\boldsymbol{\mu}^{(1)}, \boldsymbol{\mu}^{(2)}, \boldsymbol{\mu}^{(1)} - \boldsymbol{\mu}^{(2)}\}$. In the case when $\{\boldsymbol{\mu}^{(1)}, \boldsymbol{\mu}^{(2)}, \boldsymbol{\mu}^{(1)} - \boldsymbol{\mu}^{(2)}\} = \{\boldsymbol{0}\}$ this can be proved by completely the same arguments as in \cite{wade2015convex2}. The limiting object can be expressed in terms of the perimeter of the convex hull spanned by two independent planar Brownian motions.

More delicate situation arises when $\boldsymbol{0} \in \{\boldsymbol{\mu}^{(1)}, \boldsymbol{\mu}^{(2)}, \boldsymbol{\mu}^{(1)} - \boldsymbol{\mu}^{(2)}\} \neq \{\boldsymbol{0}\}$. In what follows we provide a simulation study that supports our conjecture, leaving the clarification of our conjecture open. In the case when $\boldsymbol{\mu}^{(1)} \neq \boldsymbol{\mu}^{(2)} = \boldsymbol{0}$, or $\boldsymbol{\mu}^{(2)} \neq \boldsymbol{\mu}^{(1)} = \boldsymbol{0}$, our intuition was that, since both walks contribute to the convex hull (see Figure \ref{fig:nul_i_nenul_drift}), distributional limit will not be Gaussian. On one hand, a single (non-degenerate) planar random walk with non-zero drift generates a convex hull whose perimeter has a Gaussian behavior (see \cite[Theorem 1.2]{wade2015convex}), but a single zero-drift planar random walk generates a convex hull whose perimeter does not have a Gaussian behavior (see \cite[Corollary 2.6 and Proposition 3.7]{wade2015convex2}), and this non-Gaussian part affects the convex hull generated by both walks combined. We ran some simulations and the results are shown in Figure \ref{fig:per_fig+100+0_+0+0}.

\begin{figure}[h!]
    \centering
    \vspace{25px}
    \begin{tikzpicture}[x=1pt,y=1pt]
\definecolor{fillColor}{RGB}{255,255,255}
\begin{scope}
\definecolor{drawColor}{RGB}{0,0,255}

\path[draw=drawColor,line width= 0.8pt,line join=round,line cap=round] ( 95.02,302.61) --
	( 91.38,291.74) --
	(110.06,280.17) --
	(135.10,278.17) --
	(143.23,274.23) --
	(174.31,280.41) --
	(157.56,263.59) --
	(175.67,280.49) --
	(206.94,280.23) --
	(222.29,276.71) --
	(238.31,271.57) --
	(272.52,289.68) --
	(298.27,283.03) --
	(304.66,284.04) --
	(323.54,280.27) --
	(344.35,306.42) --
	(379.93,294.52) --
	(395.76,293.42) --
	(428.64,256.32) --
	(444.77,268.91) --
	(464.71,258.88);
\end{scope}
\begin{scope}
\definecolor{drawColor}{RGB}{255,0,0}

\path[draw=drawColor,line width= 0.8pt,line join=round,line cap=round] ( 95.02,302.61) --
	( 91.04,291.15) --
	(108.08,262.53) --
	(101.49,252.06) --
	( 96.50,257.49) --
	(104.86,282.02) --
	(116.96,279.14) --
	(120.46,265.02) --
	( 98.64,246.50) --
	( 90.90,232.36) --
	( 68.72,224.35) --
	( 65.18,249.18) --
	( 92.99,241.42) --
	(104.46,229.25) --
	(111.53,229.95) --
	(107.20,227.12) --
	(111.30,211.47) --
	( 90.64,214.88) --
	( 78.26,220.28) --
	( 91.77,251.89) --
	( 92.21,268.89);
\definecolor{drawColor}{RGB}{0,0,0}

\path[draw=drawColor,line width= 0.8pt,line join=round,line cap=round] (464.71,258.88) --
	(111.30,211.47) --
	( 90.64,214.88) --
	( 68.72,224.35) --
	( 65.18,249.18) --
	( 95.02,302.61) --
	(344.35,306.42) --
	(395.76,293.42) --
	(444.77,268.91) --
	(464.71,258.88);
\end{scope}
\end{tikzpicture}
    \vspace{10px}
    \caption{The convex hull of two independet planar random walks with parameters $\boldsymbol{\mu}^{(1)} = (1.5, 0)$ (blue), $\boldsymbol{\mu}^{(2)} = (0, 0)$ (red), $\boldsymbol{\Sigma}^{(1)} = \boldsymbol{\Sigma}^{(2)} = I_2$, where $I_2$ is the two-dimensional identity matrix.}
    \label{fig:nul_i_nenul_drift}
\end{figure}
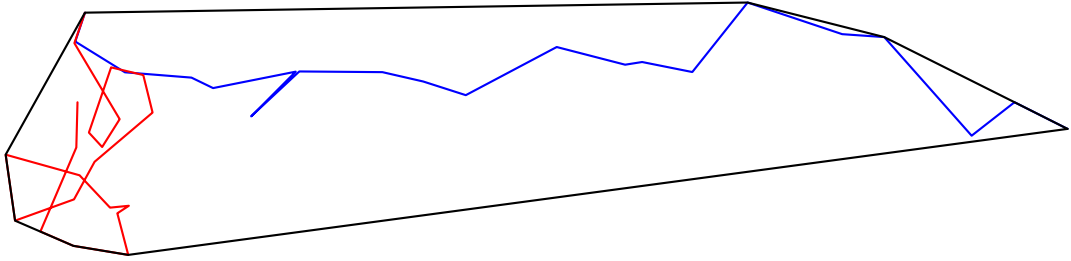

\begin{figure}[h!]
\begin{center}
    \includegraphics[width=.7\linewidth]{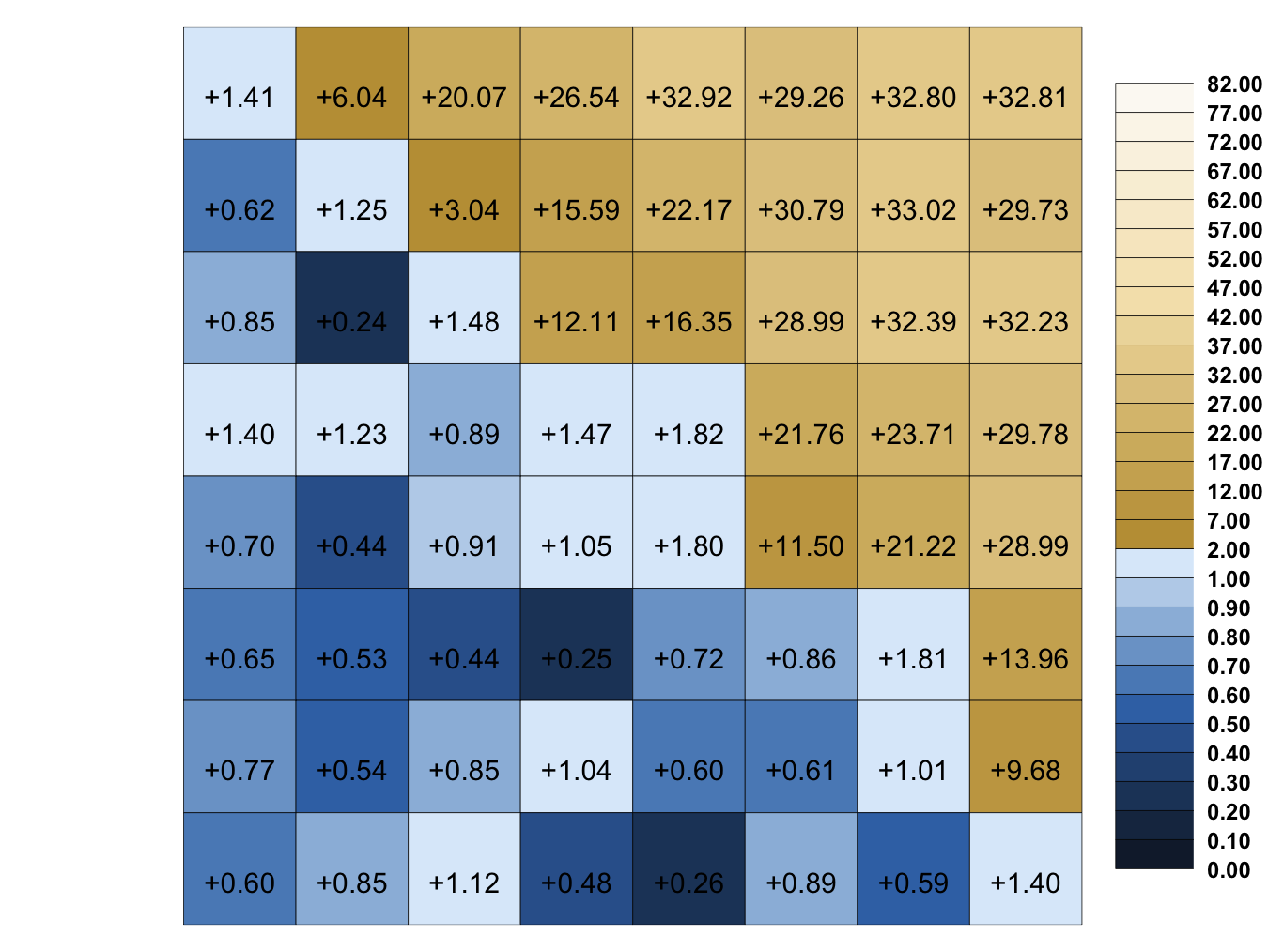}
	\caption{Simulation results for the perimeter process -- $\boldsymbol{\mu}^{(1)} = (100, 0)$, $\boldsymbol{\mu}^{(2)} = (0, 0)$.}\label{fig:per_fig+100+0_+0+0}
\end{center}
\end{figure}

Since we were simulating two planar random walks, we had some freedom in the design of our simulation study, but we kept everything as simple as possible. Namely, covariance matrices were always multiples of the identity matrix, and the steps of random walks were generated from multivariate normal distribution. To see what happens in the scenario when one of the two walks has a non-zero drift, and the other one has a zero drift (illustrated in Figure \ref{fig:nul_i_nenul_drift}), we set one drift vector to $(100, 0)$, and the other one, clearly, to $(0, 0)$. As mentioned, the covariance matrices of both walks were always of the shape $\sigma I_2$ (where $I_2$ is the two-dimensional identity matrix). We varied the value of $\sigma$ across all the elements from the set $\{0.1, 0.5, 1, 5, 10, 50, 100, 500\}$. More precisely, for every combination of $\sigma_1, \sigma_2 \in \{0.1, 0.5, 1, 5, 10, 50, 100, 500\}$ we simulated $10^3$ random walks with parameters $\boldsymbol{\mu}^{(1)} = (100, 0)$, $\boldsymbol{\Sigma}^{(1)} = \sigma_1 I_2$ and $\boldsymbol{\mu}^{(2)} = (0, 0)$, $\boldsymbol{\Sigma}^{(2)} = \sigma_2 I_2$. In each of those $10^3$ simulations, we simulated $10^4$ steps of both random walks, determined the convex hull generated by the trajectories of both walks, and then calculated the perimeter of the resulting convex hull. Hence, for each combination of values of $\sigma_1$ and $\sigma_2$ we had $10^3$ realizations of a random variable $L_n$, for $n = 10^4$. We then tested those $10^3$ realizations for normality and calculated the $p$-value. To gain additionally stability of our simulations, we repeated the procedure $5$ times and averaged all the $p$-values obtained. Since we varied the values of $\sigma_1$ and $\sigma_2$ across $8$ different values, we ended up with $8 \times 8$ matrix of averaged $p$-values. We then transformed the elements of the matrix with the mapping $x \mapsto - \log x$ so that it is easier to present the results. After this transformation, the values in the matrix that were less than or equal to $2$ corresponded to $p$-values big enough to suggest not to reject the hypothesis of Gaussian distribution. Bigger values in the matrix correspond to smaller $p$-values and point in the direction of non-Gaussian behavior. To stress this difference between values less than or equal to $2$ and values bigger than $2$, we use different color palettes for those two ranges of values. In Figure \ref{fig:per_fig+100+0_+0+0}, as in all the figures that follow, the color on the position $(i, j)$ (starting from top left corner) corresponds to the simulation in which $\sigma_1$ is equal to the $i$-th value, and $\sigma_2$ to the $j$-th value from the set $\{0.1, 0.5, 1, 5, 10, 50, 100, 500\}$. As one can see in Figure \ref{fig:per_fig+100+0_+0+0}, if the variability of the zero-drift random walk is smaller than or equal to the variability of the random walk with the non-zero drift simulations suggest not to reject the hypothesis of Gaussian distribution. However, we believe that in this case the impact of the non-Gaussian part is too small to be detected by the test. As soon as the variance of the zero-drift random walk is bigger, the simulations clearly suggest non-Gaussian behavior.

For illustration, we conduct the same experiment in the case when the assumption \eqref{eq:per_assumption} is not violated. We can see in Figure \ref{fig:A1_not_violated} that the same design of simulation study as above captures the behavior proven in Theorem \ref{theorem - CLT - perimeter}.

\begin{figure}[h!]
    \begin{center}
	\begin{subfigure}{.4\linewidth}\centering
		\includegraphics[width=\linewidth]{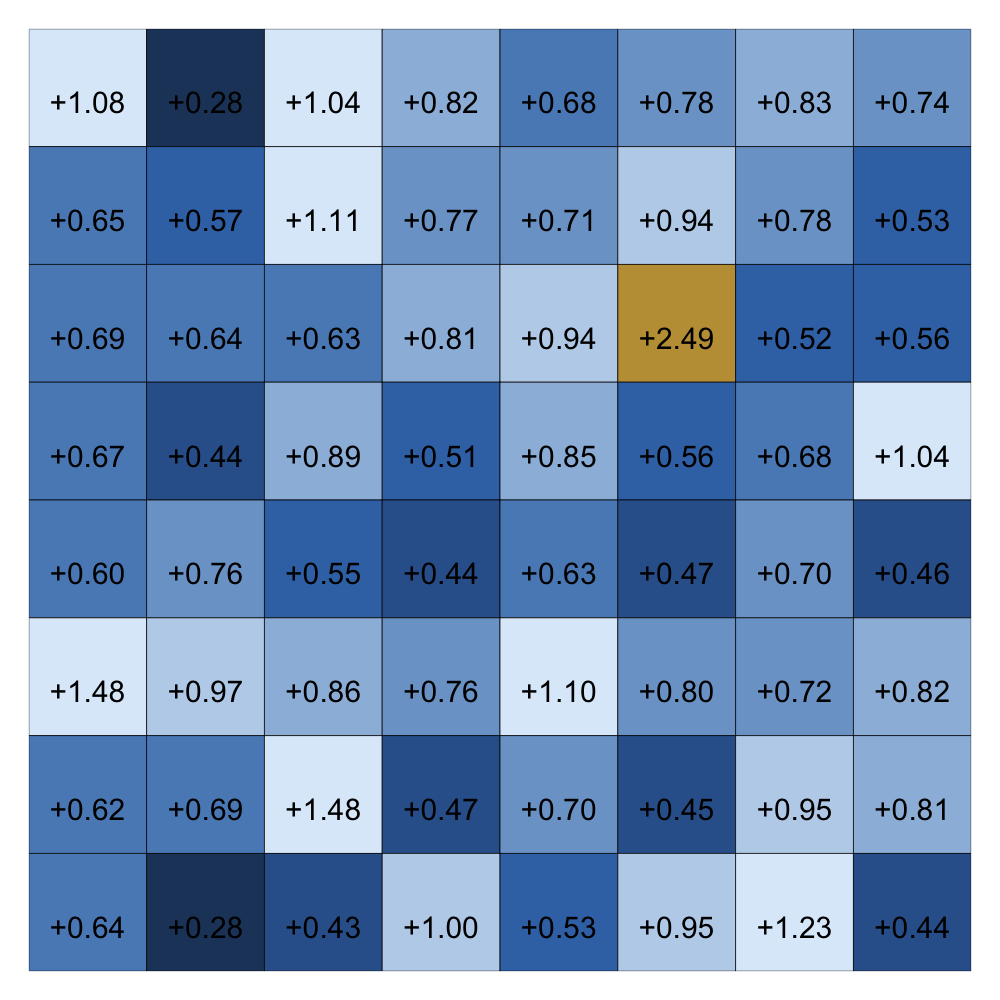}
            \caption{$\boldsymbol{\mu}^{(1)} = (200, 100)$, $\boldsymbol{\mu}^{(2)} = (-100, 100)$}
	\end{subfigure}
	\begin{subfigure}{.4\linewidth}\centering
		\includegraphics[width=\linewidth]{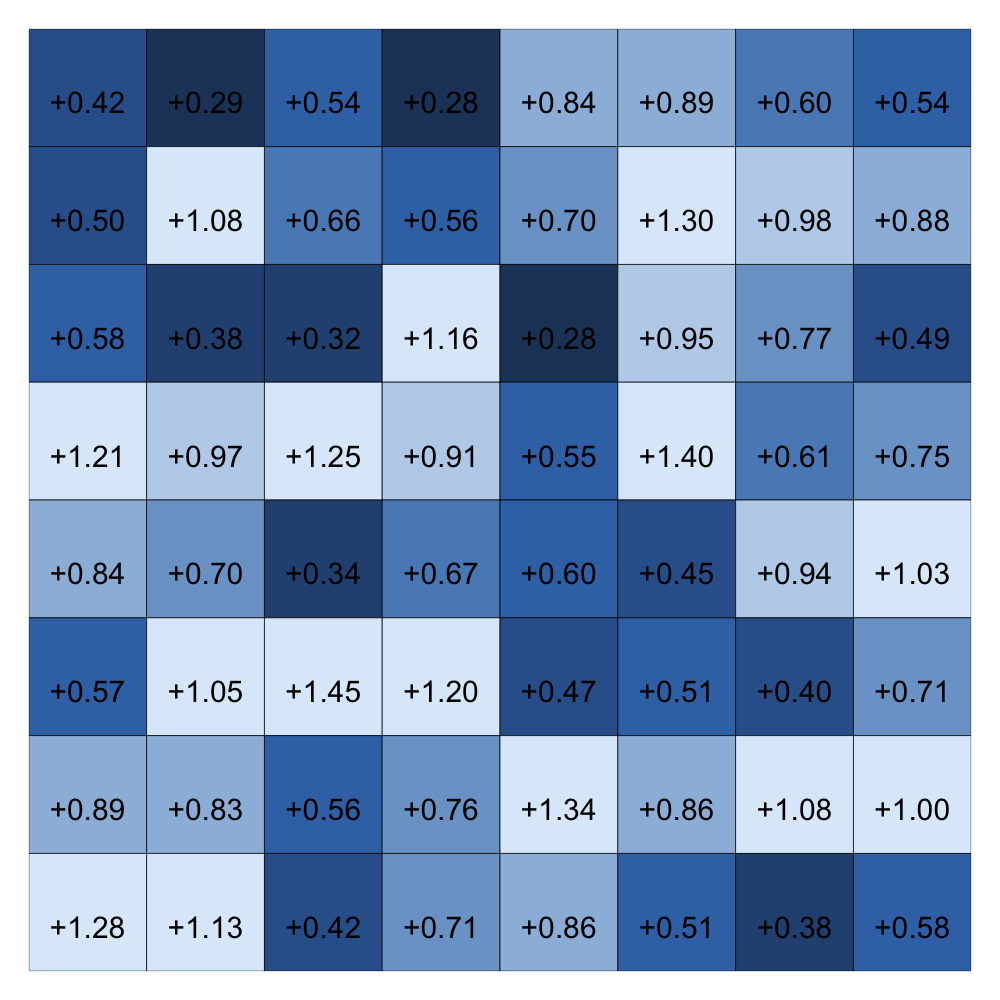}
            \caption{$\boldsymbol{\mu}^{(1)} = (200, 100)$, $\boldsymbol{\mu}^{(2)} = (100, 100)$}
	\end{subfigure}\\
	\begin{subfigure}{.4\linewidth}\centering
		\includegraphics[width=\linewidth]{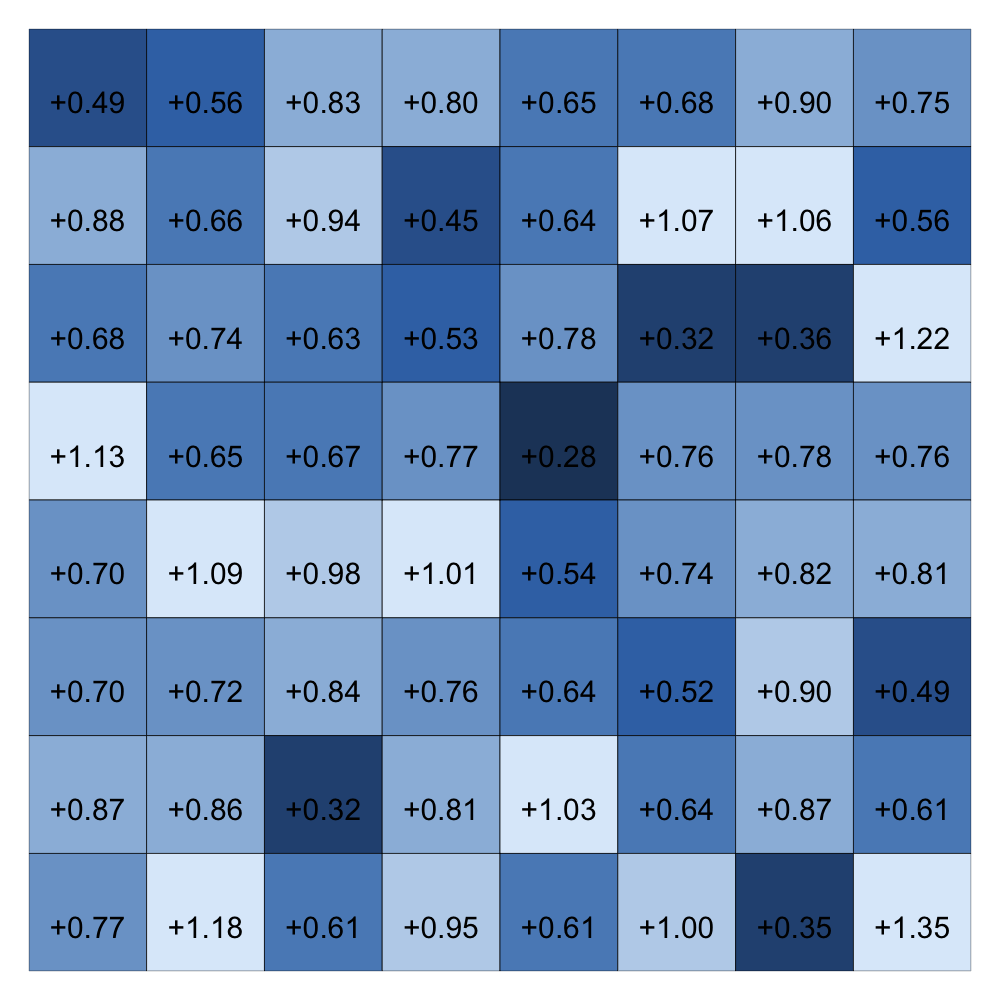}
            \caption{$\boldsymbol{\mu}^{(1)} = (200, 0)$, $\boldsymbol{\mu}^{(2)} = (-100, 0)$}
	\end{subfigure}
	\begin{subfigure}{.4\linewidth}\centering
		\includegraphics[width=\linewidth]{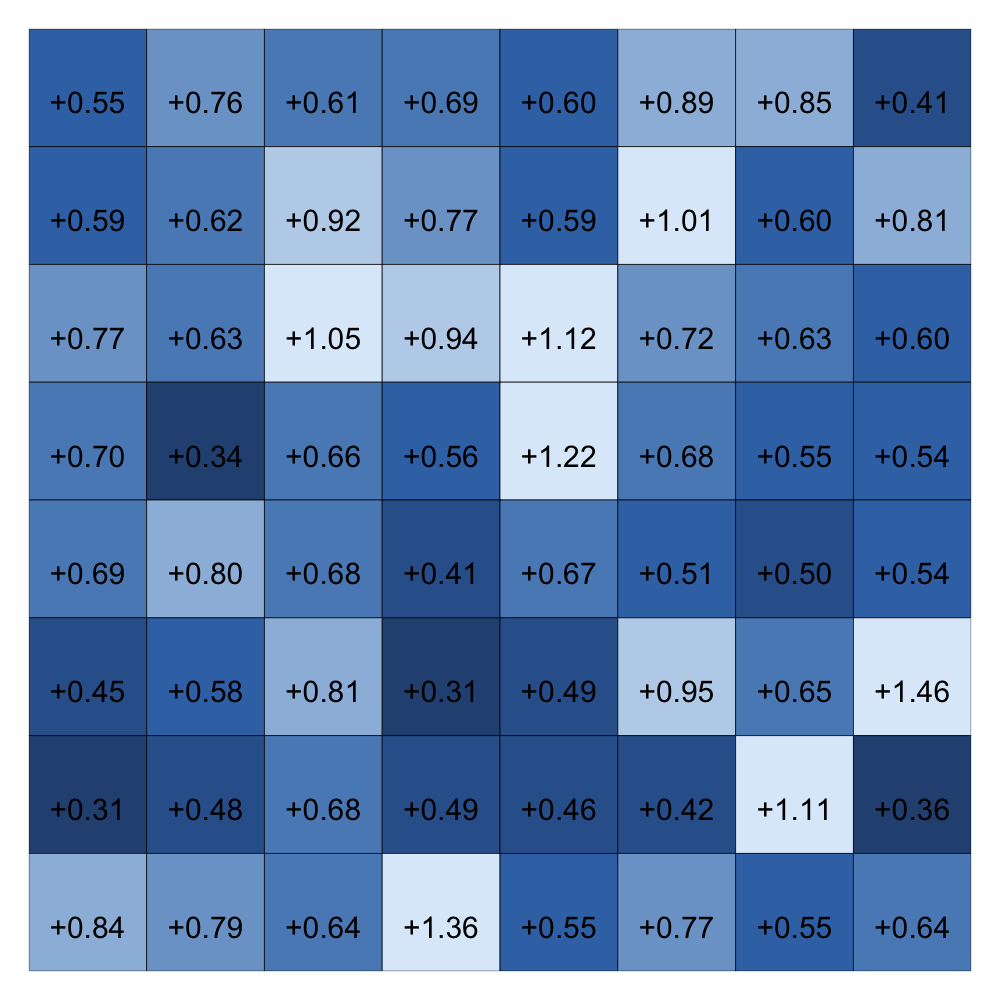}
            \caption{$\boldsymbol{\mu}^{(1)} = (200, 0)$, $\boldsymbol{\mu}^{(2)} = (100, 0)$}
	\end{subfigure}
	\caption{Simulation results (for the perimeter process) illustrating Theorem \ref{theorem - CLT - perimeter}}\label{fig:A1_not_violated}
    \end{center}
\end{figure}

The last scenario in which assumption \eqref{eq:per_assumption} is not satisfied is the one where $\boldsymbol{\mu}^{(1)} = \boldsymbol{\mu}^{(2)} \neq \boldsymbol{0}$. It was clear to us that our approach to the proof of Theorem \ref{theorem - CLT - perimeter} cannot cover this case, but our first impression was that the normality will still hold. Hence, it was somewhat surprising to us when simulations suggested that in this case we again do not have normal behavior (see Figure \ref{fig:per_fig+100+0_+100+0}). These simulation results motivated the formulation of the assumption \eqref{eq:per_assumption} in the present form. Possible justification is that one has to consider the triangle spanned by the drift vectors, and as soon as one of the three sides has length zero, the normality does not hold.

\begin{figure}[h!]
\begin{center}
    \includegraphics[width=.4\linewidth]{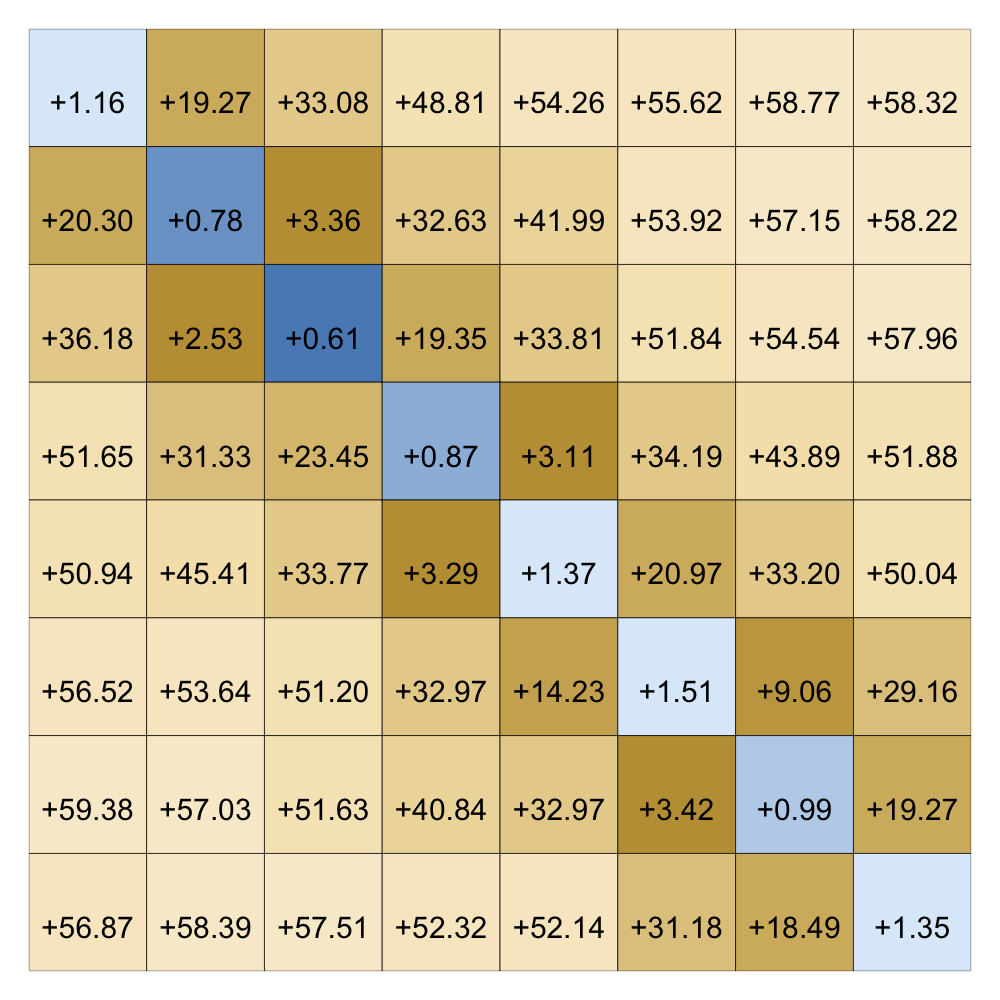}
	\caption{Simulation results for the perimeter process -- $\boldsymbol{\mu}^{(1)} = (100, 0)$, $\boldsymbol{\mu}^{(2)} = (100, 0)$.}\label{fig:per_fig+100+0_+100+0}
\end{center}
\end{figure}

When it comes to the assumption \eqref{eq:per_assumption} in the diameter case, we have completely analogous situation as above. We repeated all the experiments as above and got analogous results (see Figure \ref{fig:diam_rep}). However, in the central limit theorem for the diameter (Theorem \ref{theorem - CLT - diameter}) we have additional assumption \eqref{eq:diam_assumption}. Again, our method of proof did not work in this case, but our intuition was that the normality should still hold. We were quite surprised to see that simulations suggest non-Gaussian behavior when the set $\{\|\boldsymbol{\mu}^{(1)}\|, \|\boldsymbol{\mu}^{(2)}\|, \|\boldsymbol{\mu}^{(1)} - \boldsymbol{\mu}^{(2)}\|\}$ does not have a unique maximal element. Those simulation results are shown in Figure \ref{fig:A2_violated}.

\begin{figure}[h!]
    \begin{center}
	\begin{subfigure}{.4\linewidth}\centering
		\includegraphics[width=\linewidth]{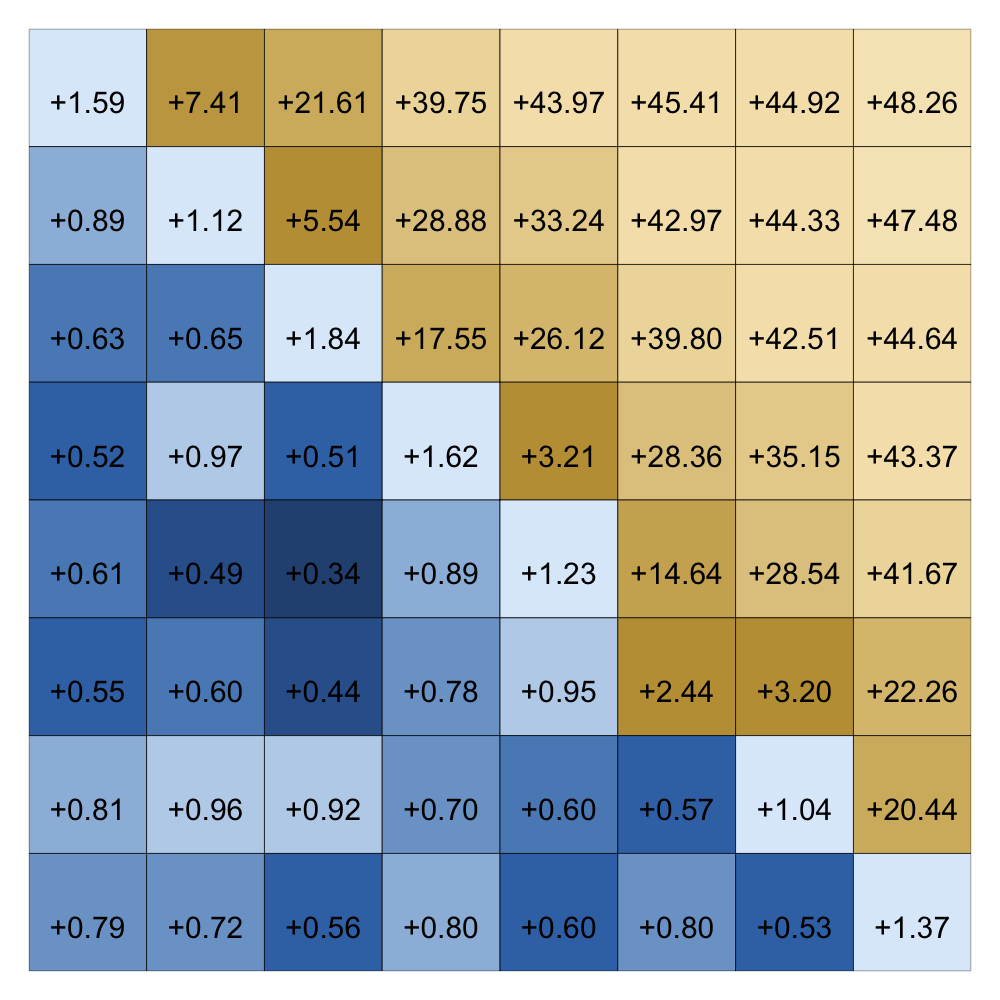}
            \caption{$\boldsymbol{\mu}^{(1)} = (100, 0)$, $\boldsymbol{\mu}^{(2)} = (0, 0)$}
	\end{subfigure}
	\begin{subfigure}{.4\linewidth}\centering
		\includegraphics[width=\linewidth]{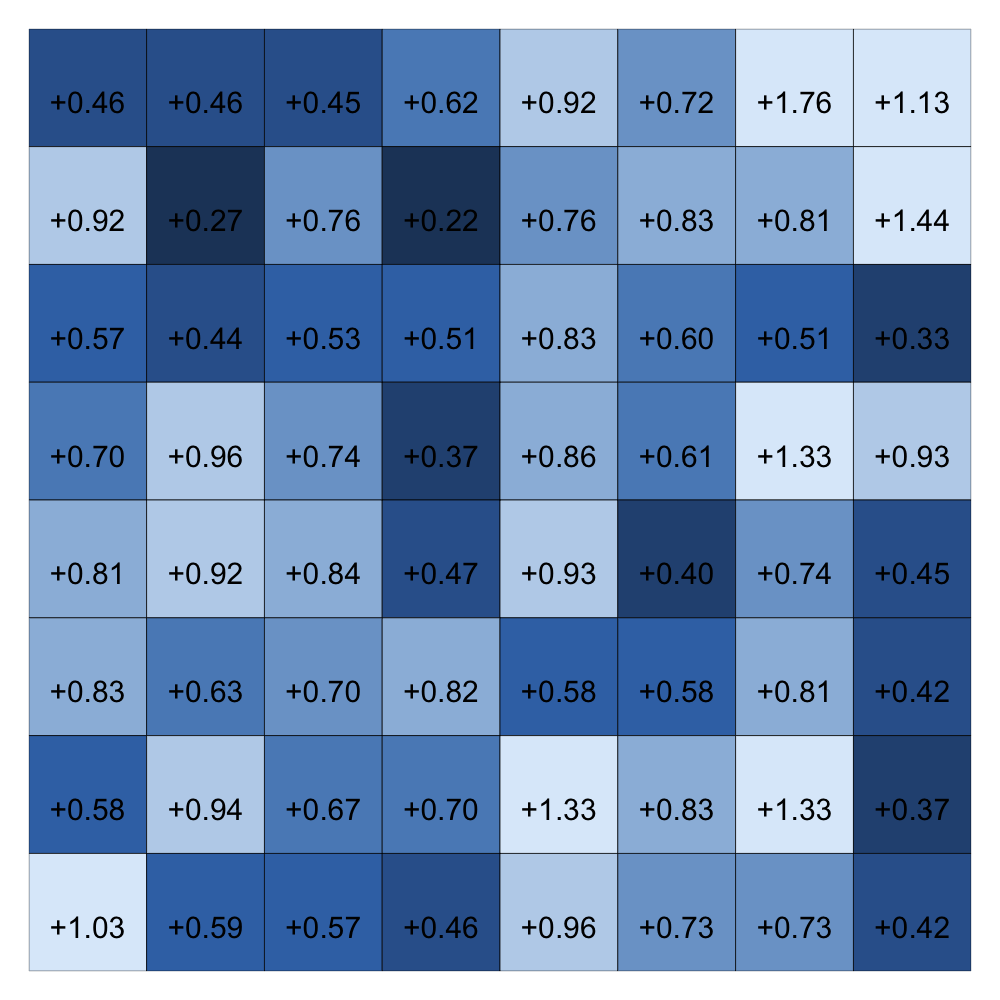}
            \caption{$\boldsymbol{\mu}^{(1)} = (200, 100)$, $\boldsymbol{\mu}^{(2)} = (-100, 100)$}
	\end{subfigure} \\
        \begin{subfigure}{.4\linewidth}\centering
		\includegraphics[width=\linewidth]{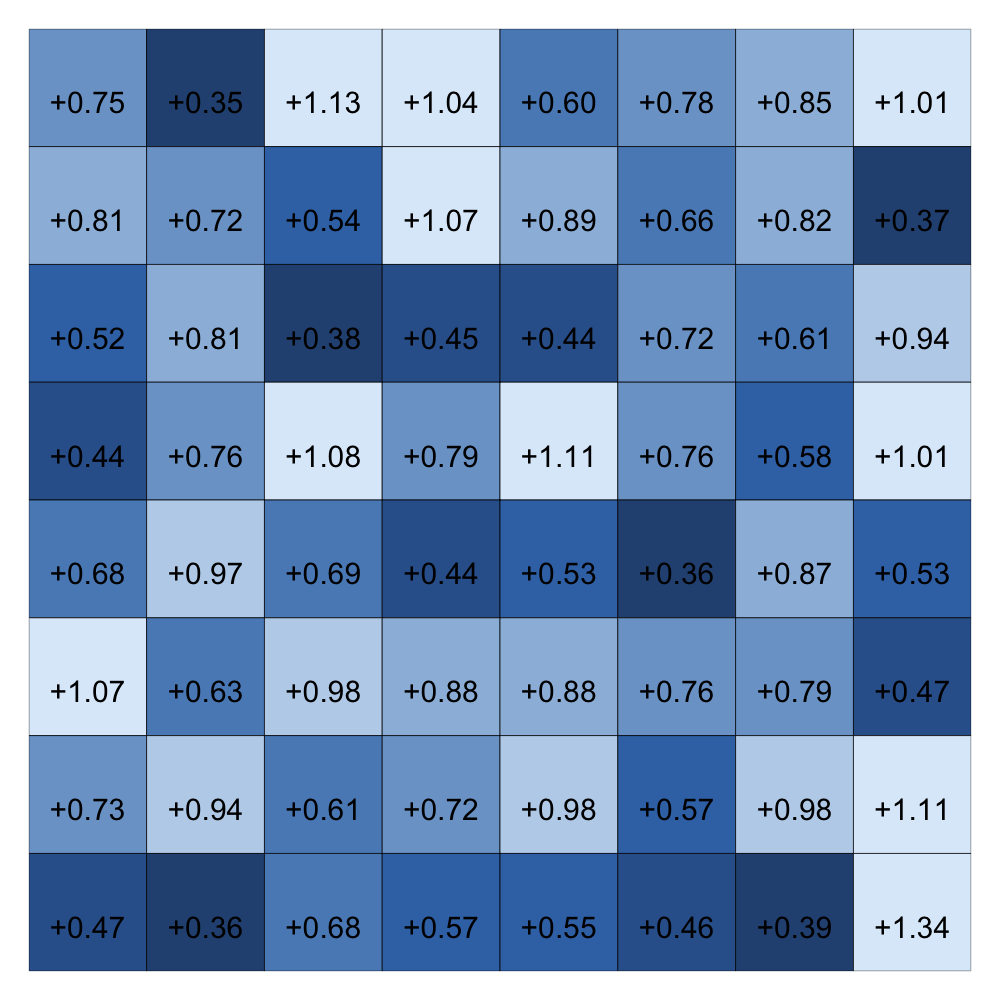}
            \caption{$\boldsymbol{\mu}^{(1)} = (200, 100)$, $\boldsymbol{\mu}^{(2)} = (100, 100)$}
	\end{subfigure}
	\begin{subfigure}{.4\linewidth}\centering
		\includegraphics[width=\linewidth]{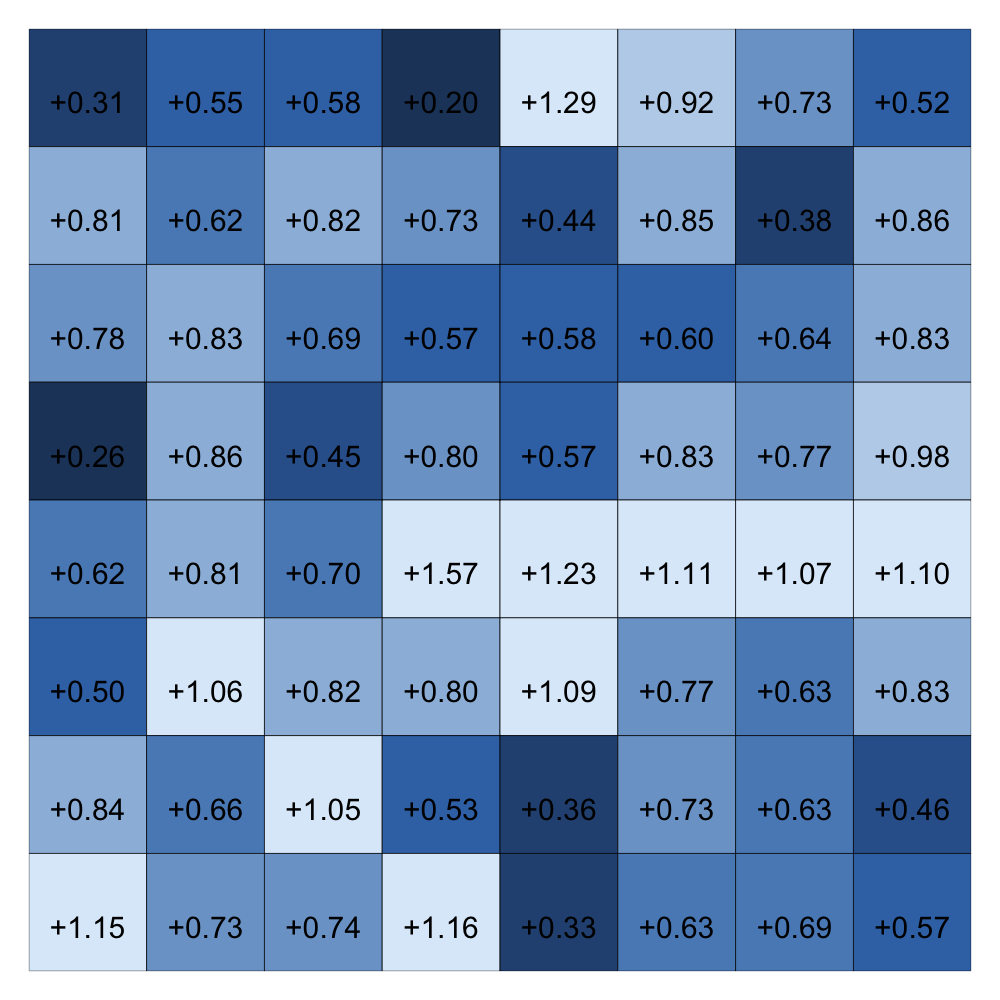}
            \caption{$\boldsymbol{\mu}^{(1)} = (200, 0)$, $\boldsymbol{\mu}^{(2)} = (-100, 0)$}
	\end{subfigure} \\
	\begin{subfigure}{.4\linewidth}\centering
		\includegraphics[width=\linewidth]{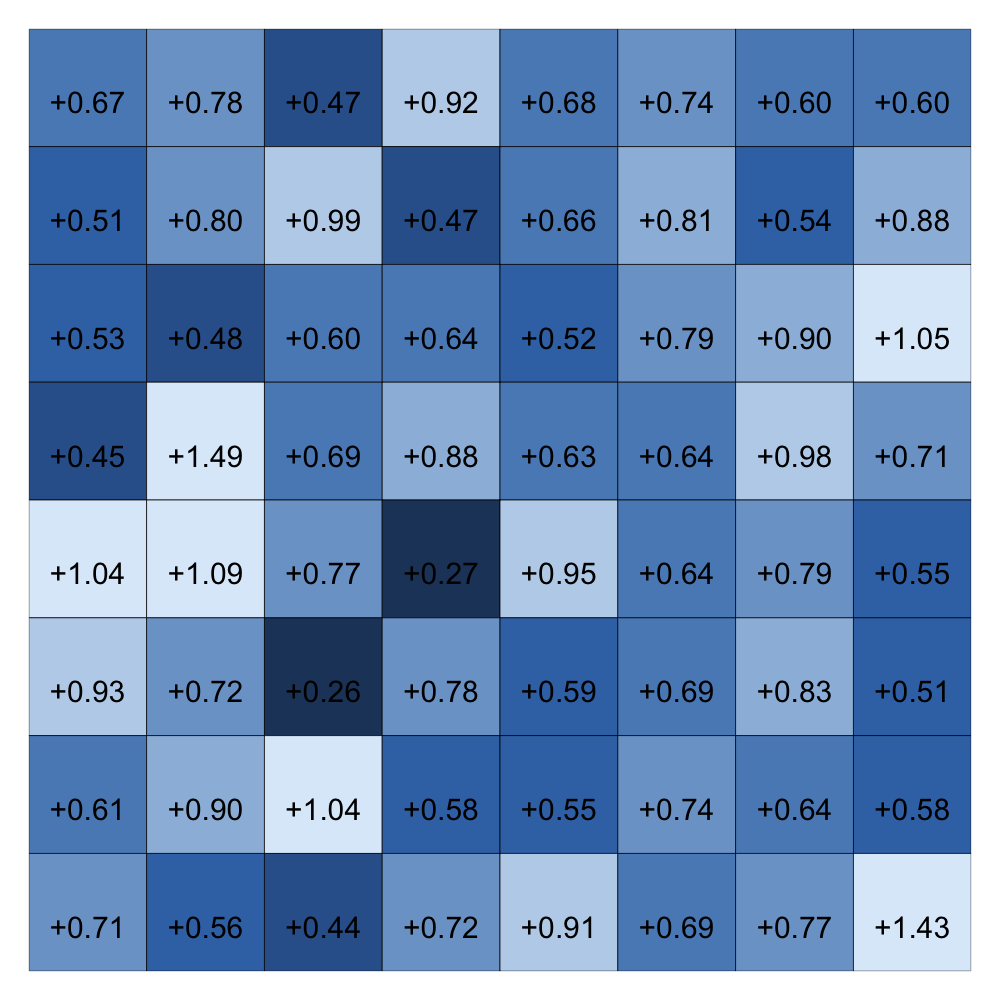}
            \caption{$\boldsymbol{\mu}^{(1)} = (200, 0)$, $\boldsymbol{\mu}^{(2)} = (100, 0)$}
	\end{subfigure}
	\begin{subfigure}{.4\linewidth}\centering
		\includegraphics[width=\linewidth]{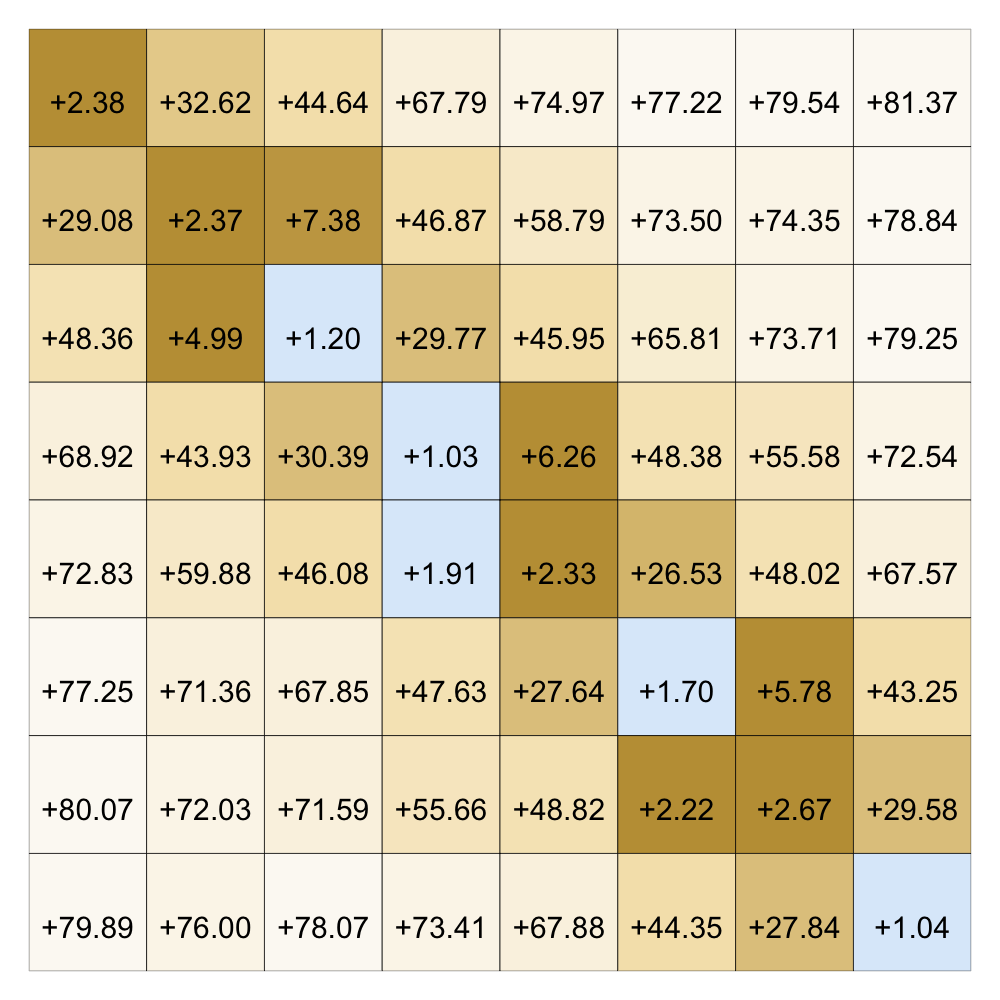}
            \caption{$\boldsymbol{\mu}^{(1)} = (100, 0)$, $\boldsymbol{\mu}^{(2)} = (100, 0)$}
	\end{subfigure}
	\caption{Simulation results for the diameter process in the same scenarios as the ones analyzed in the context of the perimeter process.}\label{fig:diam_rep}
    \end{center}
\end{figure}

\begin{figure}[h!]
    \begin{center}
	\begin{subfigure}{.4\linewidth}\centering
		\includegraphics[width=\linewidth]{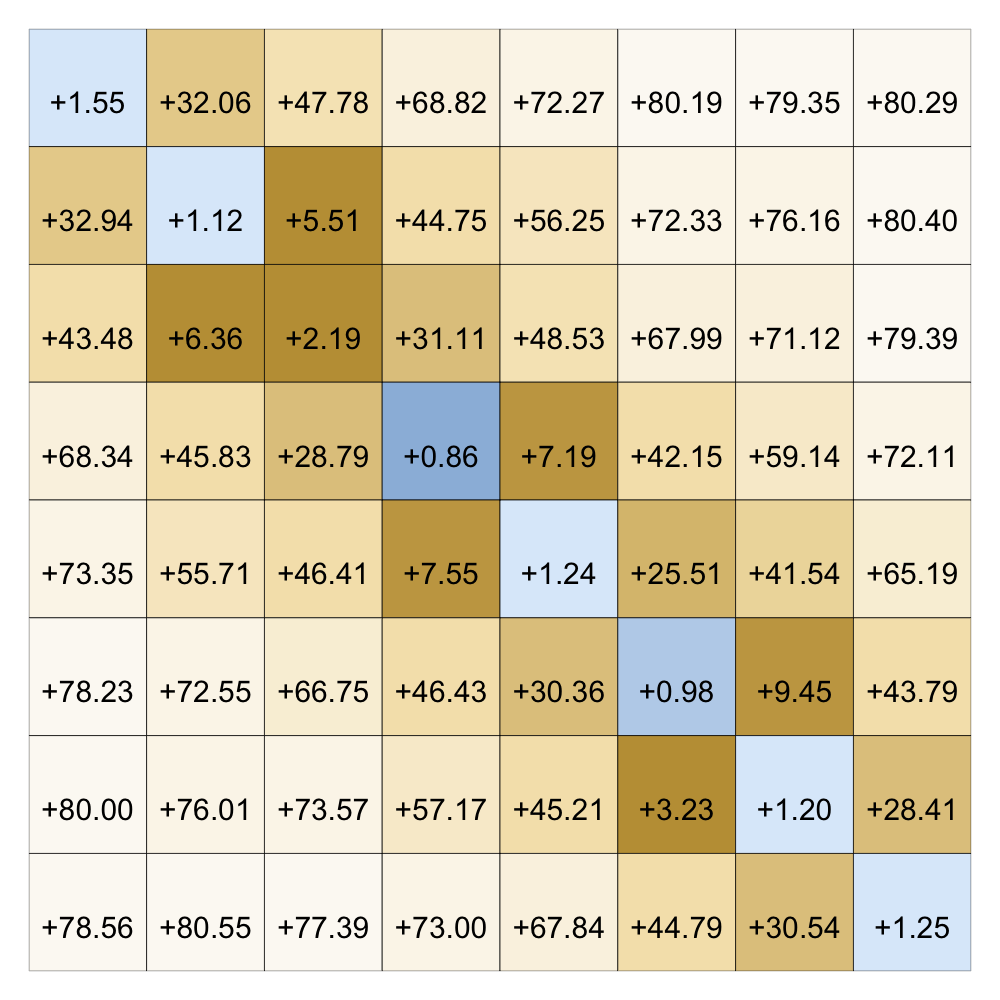}
            \caption{$\boldsymbol{\mu}^{(1)} = (100, 200)$, $\boldsymbol{\mu}^{(2)} = (-100, 200)$}
	\end{subfigure}
	\begin{subfigure}{.4\linewidth}\centering
		\includegraphics[width=\linewidth]{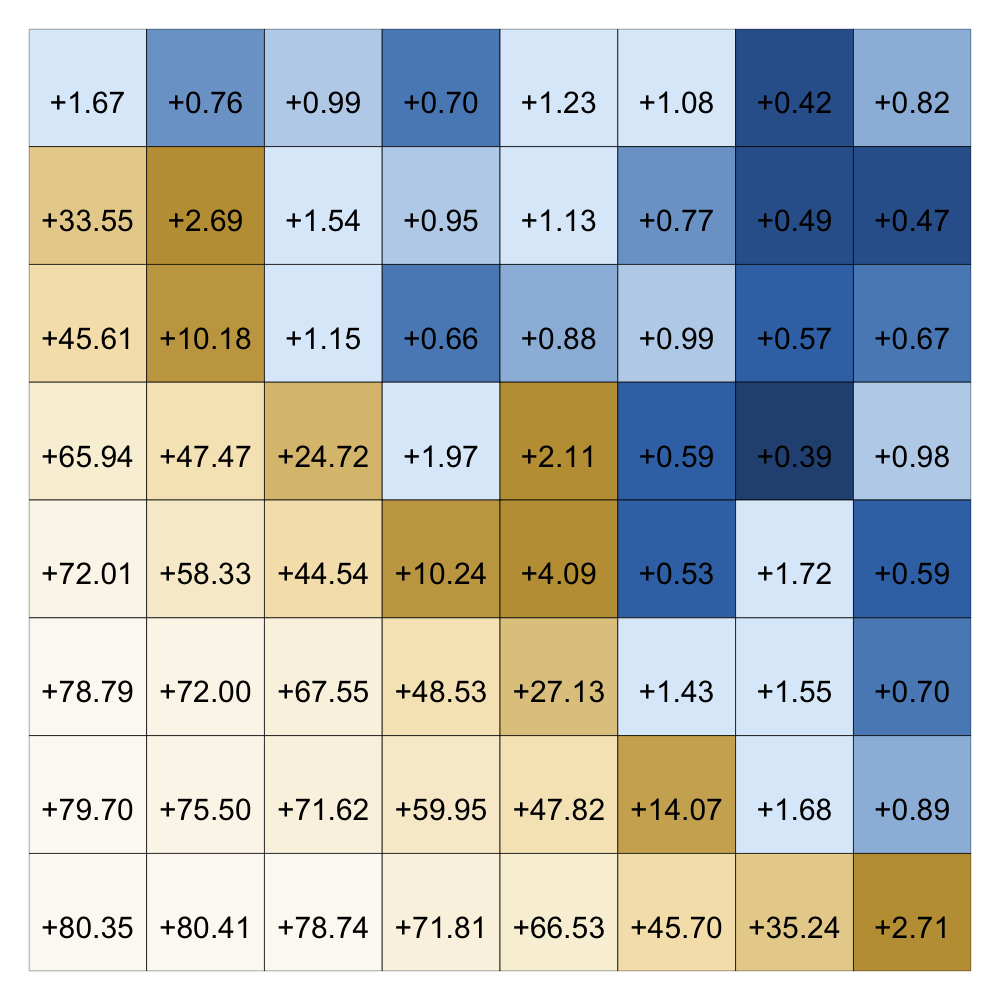}
            \caption{$\boldsymbol{\mu}^{(1)} = (0, 200)$, $\boldsymbol{\mu}^{(2)} = (200, 100)$}
	\end{subfigure}
	\caption{Simulation results (for the diameter process) in the case when assumption \eqref{eq:diam_assumption} is violated.}\label{fig:A2_violated}
 \end{center}
\end{figure}

One thing that the simulations suggest is that the variability of the walks does not change the limiting behavior of the studied processes (as long as $\sigma_L > 0$ and $\sigma_D > 0$), but it has an effect on simulations. Therefore, it could be that because of a bad simulation study design we conjectured something that does not hold. Regardless of that, it seems that scenarios excluded with assumptions \eqref{eq:per_assumption} and \eqref{eq:diam_assumption} require additional work and a different approach, and the efforts to extend our results in this direction are currently underway.

\section*{Acknowledgments}

We thank Andrew Wade (Durham University) and Wojciech Cygan (University of Wroc\l{}av) for fruitful discussions and comments. Financial support of the Croatian Science Foundation through project 2277 is gratefully acknowledged.

\noindent
\textit{Disclosure of Generative AI Tool Use}: During the preparation of this article, the authors made limited use of generative AI tools. Specifically, ChatGPT, based on GPT-4 (OpenAI, 2023–2024), was used to assist with minor language editing and to accelerate the generation of LaTeX/TikZ code for illustrative figures and simulations made in the programming language R. All outputs were reviewed and verified by the authors to ensure accuracy and scientific integrity.

\bibliography{references}

\end{document}